\RequirePackage{fix-cm}
\documentclass[smallextended]{svjour3}       
\smartqed  
\usepackage{graphicx}

\usepackage{amsfonts}
\usepackage{amsmath}
\usepackage{amssymb}
\usepackage{multirow}
\usepackage{booktabs}
\usepackage{tabularx}
\usepackage[rgb,dvipsnames]{xcolor}
\usepackage{ntheorem}
\usepackage{cite}
\usepackage[update]{epstopdf}
\usepackage[process=auto,crop=pdfcrop,cleanup={.dvi,.ps,.pdf,.log,.aux,.bbl,.glo,.ist,.mw,.out,.slo,.out.ps,.tex}]{pstool}
\usepackage[toc,page]{appendix}
\usepackage{url}
\usepackage[
  hidelinks,
  bookmarksnumbered
]{hyperref}

\let\appendixpagenameorig\appendixpagename
\renewcommand{\appendixpagename}{\normalsize \appendixpagenameorig}

\newcommand{\R}{\mathbb{R}}
\newcommand{\Rn}{\mathbb{R}^n}
\newcommand{\operator}[1]{{\normalfont \texttt{#1}}}

\newcommand{\CE}{R}
\newcommand{\Comp}{\textit{Complexity: }}

\newcommand{\Oc}{\mathcal{O}}
\newcommand{\ind}{r}

\newcolumntype{C}[1]{>{\centering\arraybackslash}m{#1}}

\definecolor{Green}{rgb}{0,0.5,0}

\newcommand{\myparagraph}[1]{\paragraph{#1}\mbox{}\\}
\DeclareMathSymbol{\shortminus}{\mathbin}{AMSa}{"39}

\allowdisplaybreaks

\AtBeginDocument{%
  \paperwidth=\dimexpr
    1in + \oddsidemargin
    + \textwidth
    + 1in + \oddsidemargin
  \relax
  \paperheight=\dimexpr
    1in + \topmargin
    + \headheight + \headsep
    + \textheight
    + 1in + \topmargin
  \relax
  \usepackage[pass]{geometry}\relax
}

\journalname{}

\begin{document}

\title{Constrained Polynomial Zonotopes}

\author{Niklas Kochdumper \and Matthias Althoff}

\institute{Niklas Kochdumper \at
           Technical University of Munich\\
           \email{niklas.kochdumper@tum.de}
           \and
           Matthias Althoff\at
           Technical University of Munich\\
           \email{althoff@tum.de}   
}

\date{Received: date / Accepted: date}

\maketitle

\begin{abstract}
We introduce constrained polynomial zonotopes, a novel non-convex set representation that is closed under linear map, Minkowski sum, Cartesian product, convex hull, intersection, union, and quadratic as well as higher-order maps. We show that the computational complexity of the above-mentioned set operations for constrained polynomial zonotopes is at most polynomial in the representation size. The fact that constrained polynomial zonotopes are generalizations of zonotopes, polytopes, polynomial zonotopes, Taylor models, and ellipsoids further substantiates the relevance of this new set representation. In addition, the conversion from other set representations to constrained polynomial zonotopes is at most polynomial with respect to the dimension, and we present efficient methods for representation size reduction and for enclosing constrained polynomial zonotopes by simpler set representations. 
\end{abstract}

\section{Introduction}

Many applications like, e.g., controller synthesis, state estimation, and formal verification, are based on algorithms that compute with sets \cite{Scott2014,Combastel2015,Bravo2006,Asarin2006}. The performance of these algorithms therefore mainly depends on efficient set representations. Ideally, a set representation is not only closed under all relevant set operations, but can also compute these efficiently. We introduce constrained polynomial zonotopes, a novel non-convex set representation that is closed under linear map, Minkowski sum, Cartesian product, convex hull, intersection, union, and quadratic as well as higher-order maps. The computational complexity for these operations is at most polynomial in the representation size. Together with our efficient methods for representation size reduction, constrained polynomial zonotopes are well suited for many algorithms computing with sets.

\subsection{Related Work}

Over the past years, many different set representations have been used in or developed for set-based computations. Relations between typical set representations are illustrated in Fig. \ref{fig:RelationsSetRepresentations}. Moreover, Table \ref{tab:setRep} shows which set representations are closed under relevant set operations.

\begin{figure}
\begin{center}
	\includegraphics[width = 0.75 \textwidth]{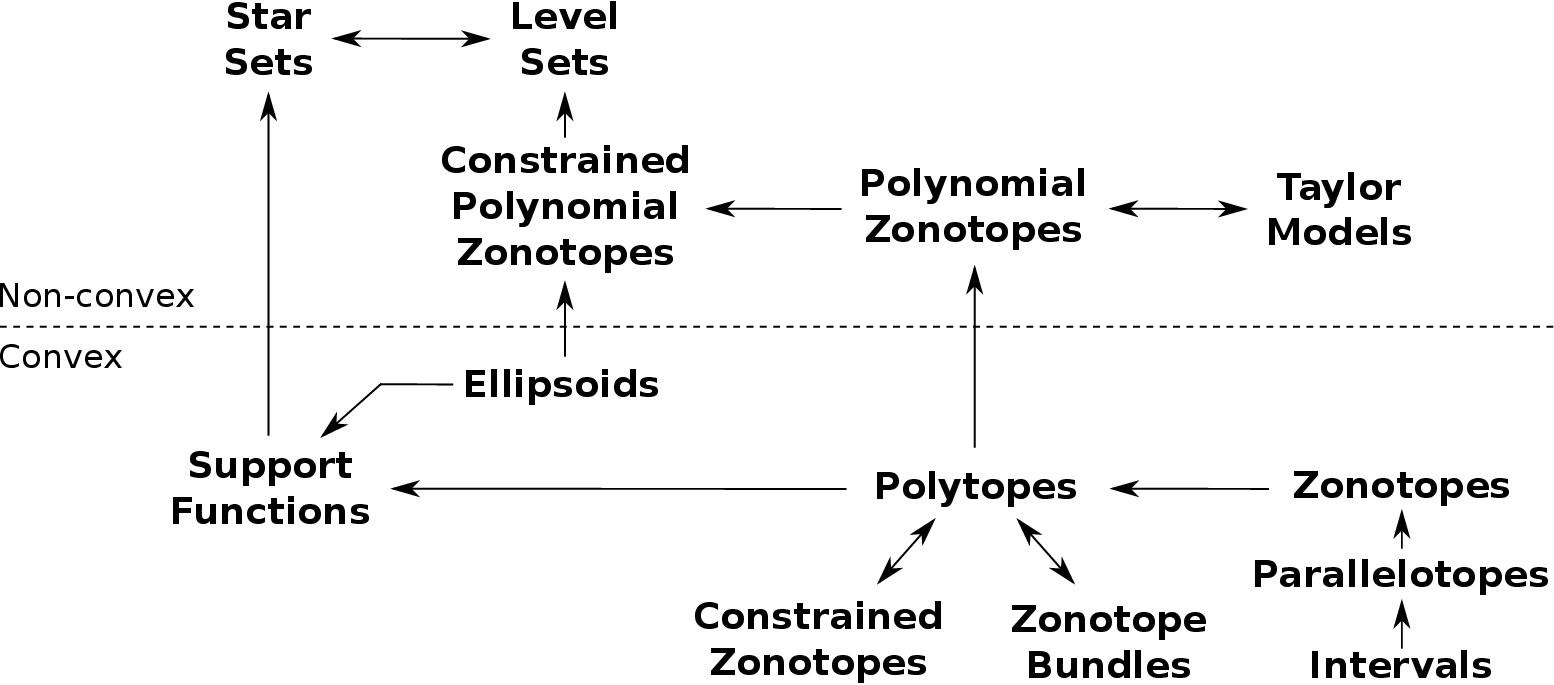}
	\caption{Visualization of the relations between the different set representations, where A $\rightarrow$ B denotes that B is a generalization of A.}
	\label{fig:RelationsSetRepresentations}
	\end{center}
	\vspace{-0.3cm}
\end{figure}

\begin{table*}
\begin{center}
\caption{Relation between set representations and set operations. The symbol $\surd$ indicates that the set representation is closed under the corresponding set operation and a closed-form expression for the computation exists, where $(\surd)$ indicates that this holds for linear maps represented by invertible matrices only. The symbol $-$ indicates that the set representation is closed under the corresponding set operation, but no closed-form expression for the computation is known (iterative algorithms, such as Fourier-Motzkin elimination, are not counted as closed-form expressions). The symbol $\times$ indicates that the set representation is not closed under the corresponding set operation. An overview for the computational complexity of set operations is provided in \cite[Table~1]{Althoff2020}.}
\label{tab:setRep}
\renewcommand{\arraystretch}{1.2}
\begin{tabular}{ p{3.6cm} C{0.95cm} C{0.95cm} C{0.95cm} C{0.95cm} C{0.95cm} C{1.05cm} C{0.95cm}}
 \toprule
 \textbf{Set Representation} & \textbf{Lin. Map} & \textbf{Mink. Sum} & \textbf{Cart. Prod.} & \textbf{Conv. Hull} & \textbf{Quad. Map} & \textbf{Inter- section} & \textbf{Union} \\ \midrule 
 Intervals & $\times$ & $\surd$ & $\surd$ & $\times$ & $\times$ & $\surd$ & $\times$ \\
 Parallelotopes & $(\surd)$ & $\times$ & $\surd$ & $\times$ & $\times$ & $\times$ & $\times$ \\
 Zonotopes & $\surd$ & $\surd$ & $\surd$ & $\times$ & $\times$ & $\times$ & $\times$  \\ 
 Polytopes (Halfspace Rep.) & $(\surd)$ & $-$ & $\surd$ & $-$ & $\times$ & $\surd$ & $\times$ \\
 Polytopes (Vertex Rep.) & $\surd$ & $\surd$ & $\surd$ & $\surd$ & $\times$ & $-$ & $\times$ \\
 Constrained Zonotopes & $\surd$ & $\surd$ & $\surd$ & $\surd$ & $\times$ & $\surd$ & $\times$ \\
 Zonotope Bundles & $(\surd)$ & $-$ & $\surd$ & $-$ & $\times$ & $\surd$ & $\times$ \\
 Ellipsoids & $\surd$  & $\times$ & $\times$ & $\times$ & $\times$ & $\times$ & $\times$ \\
 Support Functions & $\surd$ & $\surd$ & $\surd$ & $\surd$ & $\times$ & $-$ & $\times$ \\
 Taylor Models & $\surd$ & $\surd$ & $\surd$ & $\surd$ & $\surd$ & $\times$ & $\times$ \\
 Polynomial Zonotopes & $\surd$ & $\surd$ & $\surd$ & $\surd$ & $\surd$  & $\times$ & $\times$ \\
 Level Sets & $(\surd)$  & $-$ & $\surd$ & $-$ & $-$ & $\surd$ & $\surd$\\
 Star Sets & $\surd$ & $\surd$ & $\surd$ & $-$ & $-$ & $-$ & $-$\\
 Con. Poly. Zonotopes & $\surd$ & $\surd$ & $\surd$ & $\surd$ & $\surd$  & $\surd$ & $\surd$\\ 
 \bottomrule 
\end{tabular}
\end{center}
\vspace{-0.3cm}
\end{table*}

All convex sets can equivalently be represented by their support function \cite[Chapter~C.2]{Hiriart2012}. In addition, linear map, Minkowski sum, Cartesian product, and convex hull are trivial to compute for support functions \cite[Prop.~2]{Guernic2010}. Even though support functions are closed under intersection, there exists no closed-form expression for the computation of this operation, and support functions are not closed under union and quadratic maps. Ellipsoids and polytopes are special cases of sets represented by support functions \cite[Prop.~1]{Guernic2010}. While ellipsoids are only closed under linear map (see Table~\ref{tab:setRep}), polytopes are closed under linear map, Minkowski sum, Cartesian product, convex hull, and intersection \cite[Chapter~3.1]{Gruenbaum2003}. The computational complexity of the set operations for polytopes depends on the used representation \cite{Tiwary2008}, where the two main representations for polytopes are the halfspace representation and the vertex representation: Linear maps represented by invertible matrices and intersections are cheap to compute for the halfspace representation, while linear maps represented by non-invertible matrices, Minkowski sums, and convex hulls are computationally expensive \cite{Tiwary2008}. If redundant points are not removed, computation of linear maps, Minkowski sums, and convex hulls is trivial for the vertex representation, whereas calculating intersections is NP-hard \cite{Tiwary2008}.

An important subclass of polytopes are zonotopes \cite[Chapter~7.3]{Ziegler1995}. Since zonotopes can be represented compactly by so-called generators, they are well suited for the representation of high-dimensional sets. In addition, linear maps, Minkowski sums, and Cartesian products can be computed exactly and efficiently \cite[Table~1]{Althoff2016c}. Two extensions of zonotopes are zonotope bundles \cite{Althoff2011f} and constrained zonotopes \cite{Scott2016}, which are both able to represent any bounded polytope. Constrained zonotopes additionally consider linear equality constraints for the zonotope factors, whereas zonotope bundles represent the set implicitly by the intersection of several zonotopes. Two special cases of zonotopes are parallelotopes, which are zonotopes with linearly independent generators, and multi-dimensional intervals. Since intervals are not closed under linear map, algorithms computing with intervals often split them to obtain a desired accuracy \cite{Jaulin2006}.

Common non-convex set representations are star sets, level sets, Taylor models, and polynomial zonotopes. The concept of star sets \cite{Duggirala2016,Bak2017b} is similar to the one of constrained zonotopes, but logical predicates instead of linear equality constraints are used to constrain the values of the zonotope factors. Level sets of nonlinear functions \cite{Osher2006} can represent any shape. While star sets and level sets are very expressive (see Fig. \ref{fig:RelationsSetRepresentations}), it is for many of the relevant operations unclear how they are computed (see Table~\ref{tab:setRep}). Taylor models \cite{Makino2003} consist of a polynomial and an interval remainder part. A set representation that is very similar to Taylor models are polynomial zonotopes, which were first introduced in \cite{Althoff2013a}. A computationally efficient sparse representation of polynomial zonotopes was recently proposed in \cite{Kochdumper2019}. Due to their polynomial nature, Taylor models and polynomial zonotopes are both closed under quadratic and higher-order maps (see Table~\ref{tab:setRep}). 

In this work we introduce constrained polynomial zonotopes, a novel non-convex set representation that combines the concept of adding equality constraints for the zonotope factors used by constrained zonotopes \cite{Scott2016} with the sparse polynomial zonotope representation in \cite{Kochdumper2019}. Constrained polynomial zonotopes are closed under all relevant set operations (see Table~\ref{tab:setRep}) and can represent any set in Fig.~\ref{fig:RelationsSetRepresentations}, except star sets, level sets, and sets defined by their support function. As shown in Table~\ref{tab:setRep}, constrained polynomial zonotopes are the only set representation for which closed-form expressions for the calculation of all relevant set operations are known.

\vspace{-10pt}

\subsection{Notation and Assumptions}

In the remainder of this work, we use the following notations: Sets are denoted by calligraphic letters, matrices by uppercase letters, and vectors by lowercase letters. Moreover, the set of natural numbers is denoted by $\mathbb{N} = \{1, 2,\dots \}$, the set of natural numbers including zero is denoted by $\mathbb{N}_0 = \{0, 1, 2, \dots \}$, and the set of real numbers is denoted by $\mathbb{R}$. Given a set $\mathcal{H} = \{h_1,\dots,h_n\}$, $|\mathcal{H}| = n$ denotes the cardinality of the set. Given a vector $b \in \mathbb{R}^n$, $b_{(i)}$ refers to the $i$-th entry. Likewise, given a matrix $A \in \mathbb{R}^{n \times w}$, $A_{(i,\cdot)}$ represents the $i$-th matrix row, $A_{(\cdot,j)}$ the $j$-th column, and $A_{(i,j)}$ the $j$-th entry of matrix row $i$. Given a set of positive integer indices $\mathcal{H} = \{h_1,\dots,h_{|\mathcal{H}|} \}$ with $\forall i \in \{ 1, \dots, |\mathcal{H}| \},~1 \leq h_i \leq w$, notation $A_{(\cdot,\mathcal{H})}$ is used for $[ A_{( \cdot,h_1 )} ~ \dots ~ A_{( \cdot, h_{|\mathcal{H}|} )} ]$, where $[C~D]$ denotes the concatenation of two matrices $C$ and $D$. The symbols $\mathbf{0}$ and $\mathbf{1}$ represent matrices and vectors of zeros and ones of proper dimension, and $\text{diag}(a)$ returns a square matrix with $a \in \mathbb{R}^n$ on the diagonal. The empty matrix is denoted by $[~]$ and the identity matrix of dimension $n \times n$ is denoted by $I_n \in \R^{n \times n}$. Moreover, we use the shorthand $\mathcal{I} = [l,u]$ for an $n$-dimensional interval $\mathcal{I} := \{x \in \Rn~|~l_{(i)} \leq x_{(i)} \leq u_{(i)},~i = 1,\dots,n\}$. For the derivation of computational complexity, we consider all binary operations, except concatenations; initializations are also not considered.


\section{Definitions}
\label{sec:Preliminaries}

Let us first provide some definitions that are important for the remainder of the paper. We begin with zonotopes:
\begin{definition}
	(Zonotope) \cite[Def.~1]{Girard2005} Given a constant offset $c \in \Rn$ and a generator matrix $G \in \R^{n \times p}$, a zonotope $\mathcal{Z} \subset \Rn$ is defined as
	\begin{equation*}
 	\mathcal{Z} := \bigg \{ c + \sum_{k=1}^{p} \alpha_k \, G_{(\cdot,k)} ~ \bigg |  ~ \alpha_k \in [\shortminus 1,1] \bigg \}.
 \end{equation*}
 The scalars $\alpha_k$ are called factors and we use the shorthand $\mathcal{Z} = \langle c,G \rangle_Z$. \label{def:zonotope} \hfill $\square$
\end{definition}
Constrained zonotopes \cite{Scott2016} can represent arbitrary bounded polytopes:

\begin{definition}
 (Constrained Zonotope) \cite[Def. 3]{Scott2016} Given a constant offset $c \in \Rn$, a generator matrix $G \in \R^{n \times p}$, a constraint matrix $A \in \R^{m \times p}$, and a constraint vector $b \in \R^m$, a constrained zonotope $\mathcal{CZ} \subset \Rn$ is defined as
 \begin{equation*}
 	\mathcal{CZ} := \bigg \{ c + \sum_{k=1}^{p} \alpha_k \, G_{(\cdot,k)} ~ \bigg | ~ \sum_{k=1}^p \alpha_k \, A_{(\cdot,k)} = b, ~ \alpha_k \in [\shortminus 1,1] \bigg \}.
 \end{equation*}
 We use the shorthand $\mathcal{CZ} = \langle c,G,A,b \rangle_{CZ}$. 
 \label{def:conZonotope} \hfill $\square$
\end{definition}
Polynomial zonotopes are a non-convex set representation first introduced in \cite{Althoff2013a}. We use the sparse representation of polynomial zonotopes \cite{Kochdumper2019}:

\begin{definition}
  (Polynomial Zonotope) \cite[Def. 1]{Kochdumper2019} Given a constant offset $c \in \R^n$, a generator matrix $G \in \mathbb{R}^{n \times h}$, and an exponent matrix $E \in \mathbb{N}_{0}^{p \times h}$, a polynomial zonotope $\mathcal{PZ} \subset \mathbb{R}^n$ is defined as  
  \begin{equation*}
    \mathcal{PZ} := \bigg \{  c + \sum _{i=1}^h \bigg( \prod _{k=1}^p \alpha _k ^{E_{(k,i)}} \bigg) G_{(\cdot,i)} ~ \bigg| ~\alpha_k \in [\shortminus 1,1]\bigg\}.
  \end{equation*}
  In contrast to \cite[Def.~1]{Kochdumper2019}, we explicitly do not integrate the constant offset $c$ in $G$, and we do not consider independent generators since each polynomial zonotope with independent generators can be equivalently represented as a polynomial zonotope without independent generators \cite[Prop.~1]{Kochdumper2020c}. We use the shorthand $\mathcal{PZ} = \langle c,G,E \rangle_{PZ}$. 
  \label{def:polyZonotope} \hfill $\square$
\end{definition} 
An ellipsoid is defined as follows:
\begin{definition}
	(Ellipsoid) \cite[Eq. 2.3]{Boyd2004}  Given a constant offset $c \in \mathbb{R}^n$ and a symmetric and positive definite matrix $Q \in \mathbb{R}^{n \times n}$, an ellipsoid $\mathcal{E} \subset \mathbb{R}^n$ is defined as
	\begin{equation*}
		\mathcal{E} := \big \{ x ~ \big | ~ (x-c)^T Q^{-1} (x-c) \leq 1 \big \}.
	\end{equation*}
	We use the shorthand $\mathcal{E} = \langle c,Q \rangle_{E}$. 
	\label{def:ellipsoid} \hfill $\square$
\end{definition}
In this paper we consider the standard set operations listed in Table~\ref{tab:setRep}. Given two sets $\mathcal{S}_1, \mathcal{S}_2 \subset \mathbb{R}^n$, a set $\mathcal{S}_3 \subset \R^w$, a matrix $M \in \mathbb{R}^{w \times n}$, and a discrete set of matrices $\mathcal{Q} = \{Q_1,\dots,Q_w\}$ with $Q_i \in \mathbb{R}^{n \times n}$, $i = 1, \dots, w$, these operations are defined as follows: 
\begin{alignat}{2}
	& \text{Linear map:} ~~ && M \otimes \mathcal{S}_1 := \big \{ M s_1 ~\big |~ s_1 \in \mathcal{S}_1 \big \} \label{eq:defLinTrans}\\[7pt]
	& \text{Minkowski sum:} && \mathcal{S}_1 \oplus \mathcal{S}_2 := \big \{ s_1 + s_2 ~\big |~ s_1 \in \mathcal{S}_1,~ s_2 \in \mathcal{S}_2 \big \} 
	\label{eq:defMinSum} \\[7pt]
	& \text{Cartesian prod.:} ~~~ && \mathcal{S}_1 \times \mathcal{S}_3 := \big \{ [s_1^T ~ s_3^T ]^T ~\big |~ s_1 \in \mathcal{S}_1,~ s_3 \in \mathcal{S}_3 \big \} 
	\label{eq:defCartProduct} \\
	& \text{Convex hull\footnotemark:} && conv(\mathcal{S}_1,\mathcal{S}_2) := \bigg \{\sum_{i=1}^{n+1} \lambda_i\,s_i~\bigg|~s_i \in \mathcal{S}_1 \cup \mathcal{S}_2,~\lambda_i \geq 0,~\sum_{i=1}^{n+1} \lambda_i = 1 \bigg\} \label{eq:defConvHull} \\
	& \text{Quadratic map:} && sq(\mathcal{Q},\mathcal{S}_1) := \big \{ x ~\big|~ x_{(i)} = s_1^T Q_i s_1, ~s_1 \in \mathcal{S}_1,~ i = 1, \dots, w \big \} \label{eq:defQuadMap} \\[7pt]
	& \text{Intersection:} && \mathcal{S}_1 \cap \mathcal{S}_2 := \big \{ s ~\big|~ s \in \mathcal{S}_1,~ s \in \mathcal{S}_2 \big \} \label{eq:defIntersection} \\[7pt]
	& \text{Union:} && \mathcal{S}_1 \cup \mathcal{S}_2 := \big \{ s ~\big|~ s \in \mathcal{S}_1 \vee s \in \mathcal{S}_2 \big \} \label{eq:defUnion} 	
\end{alignat}
\footnotetext{Definition according to Caratheodory's theorem \cite{Barany1982}. This rather complex definition is required since constrained polynomial zonotopes can represent disjoint sets.}

\vspace{-10pt}

\noindent Moreover, we consider another set operation that we refer to as the linear combination of two sets:
\begin{equation}
	comb(\mathcal{S}_1,\mathcal{S}_2) := \bigg \{ \frac{1}{2} (1+\lambda) s_1 + \frac{1}{2}(1-\lambda) s_2 ~\bigg |~ s_1 \in \mathcal{S}_1,~s_2 \in \mathcal{S}_2,~\lambda \in [\shortminus 1,1] \bigg\}. \label{eq:defLinComb}
\end{equation}
For convex sets, the convex hull and the linear combination are identical. However, for non-convex sets as considered in this paper, the two operations differ\footnote{The convex hull that we consider in our previous work on sparse polynomial zonotopes in \cite{Kochdumper2019} actually defines a linear combination since polynomial zonotopes are non-convex.}. We consider both operations since for many algorithms, such as reachability analysis \cite[Eq. (3.4)]{Althoff2010a}, it is sufficient to compute the linear combination instead of the convex hull.


\section{Constrained Polynomial Zonotopes}

In this section, we introduce constrained polynomial zonotopes (CPZs). A CPZ is constructed by adding polynomial equality constraints to a polynomial zonotope:
\begin{definition}
	(Constrained Polynomial Zonotope) Given a constant offset $c \in \Rn$, a generator matrix $G \in \R^{n \times h}$, an exponent matrix $E \in \mathbb{N}_{0}^{p \times h}$, a constraint generator matrix $A \in \R^{m \times q}$, a constraint vector $b \in \R^m$, and a constraint exponent matrix $\CE \in \mathbb{N}_{0}^{p \times q}$, a constrained polynomial zonotope is defined as
	\begin{equation*}
		\mathcal{CPZ} := \bigg \{ c + \sum_{i=1}^{h} \bigg( \prod_{k=1}^p \alpha_k^{E_{(k,i)}} \bigg) G_{(\cdot,i)} ~ \bigg | ~ \sum_{i=1}^{q} \bigg( \prod_{k=1}^p \alpha_k^{\CE_{(k,i)}} \bigg) A_{(\cdot,i)} = b, ~ \alpha_k \in [\shortminus 1,1]   \bigg \}.
	\end{equation*}
	The constrained polynomial zonotope is regular if the exponent matrix $E$ and the constrained exponent matrix $R$ do not contain duplicate columns or all-zero columns:
	\begin{equation*}
		\forall i,j \in \{1,\dots,h\}, ~ (i \neq j) \Rightarrow \big(E_{(\cdot,i)} \neq E_{(\cdot,j)}\big) ~~ \mathrm{and} ~~ \forall i \in \{1,\dots,h\}, ~ E_{(\cdot,i)} \neq \mathbf{0},
	\end{equation*}
	and
	\begin{equation*}
		\forall i,j \in \{1,\dots,q \},~ (i \neq j) \Rightarrow \big(\CE_{(\cdot,i)} \neq \CE_{(\cdot,j)}\big) ~~ \mathrm{and} ~~ \forall i \in \{1,\dots,q \}, ~ \CE_{(\cdot,i)} \neq \mathbf{0}.
	\end{equation*}
	The scalars $\alpha_k$ are called factors, where the number of factors is $p$, the number of generators $G_{(\cdot,i)}$ is $h$, the number of constraints is $m$, and the number of constraint generators $A_{(\cdot,i)}$ is $q$. The order $\rho = \frac{h+q}{n}$ estimates the complexity of a constrained polynomial zonotope. We use the shorthand $\mathcal{CPZ} = \langle c,G,E,A,b,\CE \rangle_{CPZ}$.
	\label{def:CPZ} \hfill $\square$
\end{definition}
All components of a set $\square_i$ have index $i$, e.g., the parameter $p_i$, $h_i$, $m_i$, and $q_i$ as defined in Def.~\ref{def:CPZ} belong to $\mathcal{CPZ}_i$. The quantity of scalar numbers $\mu$ required to store a CPZ is 
\begin{equation}
	\mu = (n+p)h + n + (m+p)q + m
	\label{eq:repSize}
\end{equation}
since $c$ has $n$ entries, $G$ has $nh$ entries, $E$ has $ph$ entries, $A$ has $mq$ entries, $b$ has $m$ entries, and $\CE$ has $pq$ entries. We call $\mu$ the representation size of the CPZ. Moreover, we call the polynomial zonotope $\mathcal{PZ} = \langle c,G,E \rangle_{PZ}$ corresponding to $\mathcal{CPZ} = \langle c,G,E,A,b,\CE \rangle_{CPZ}$ the constructing polynomial zonotope. For the derivation of the computational complexity of set operations with respect to the dimension $n$, we make the assumption that 
\begin{equation}
	p = a_p n, ~ h = a_h n, ~ q = a_q n, ~ m = a_m n, 
	\label{eq:complexity}
\end{equation}
with $a_p,a_h,a_q,a_m \in \R_{\geq 0}$. This assumption is justified by the fact that one usually reduces the representation size to a desired upper bound when computing with CPZs. 

\newpage \noindent We demonstrate the concept of CPZs by an example:

\begin{example}
	The CPZ 
	\begin{equation*}
		\mathcal{CPZ} = \bigg \langle \begin{bmatrix} 0 \\ 0 \end{bmatrix}, \begin{bmatrix} 1 & 0 & 1 & \shortminus 1 \\ 0 & 1 & 1 & 1 \end{bmatrix}, \begin{bmatrix} 1 & 0 & 1 & 2 \\ 0 & 1 & 1 & 0 \\ 0 & 0 & 1 & 1 \end{bmatrix}, \begin{bmatrix} 1 & \shortminus0.5 & 0.5 \end{bmatrix}, 0.5, \begin{bmatrix} 0 & 1 & 2 \\ 1 & 0 & 0 \\ 0 & 1 & ~0 \end{bmatrix} \bigg \rangle_{CPZ}
	\end{equation*}
	defines the set
	\begin{equation*}
		\begin{split}
		\mathcal{CPZ} = \bigg \{ & \begin{bmatrix} 0 \\ 0 \end{bmatrix} + \begin{bmatrix} 1 \\ 0 \end{bmatrix} \alpha_1 + \begin{bmatrix} 0 \\ 1 \end{bmatrix} \alpha_2 + \begin{bmatrix} 1 \\ 1 \end{bmatrix} \alpha_1 \alpha_2 \alpha_3 + \begin{bmatrix} \shortminus 1 \\ 1 \end{bmatrix} \alpha_1^2 \alpha_3 ~ \bigg |  \\
		& ~~~~ \alpha_2 - 0.5 \, \alpha_1 \alpha_3 + 0.5 \, \alpha_1^2 = 0.5, ~ \alpha_1,\alpha_2,\alpha_3 \in [\shortminus 1,1] \bigg \},
		\end{split}
	\end{equation*}
	which is visualized in Fig. \ref{fig:Example}.
	\label{ex:CPZ}
\end{example}

\begin{figure}
  \centering
  \psfragfig[width=0.95\columnwidth]{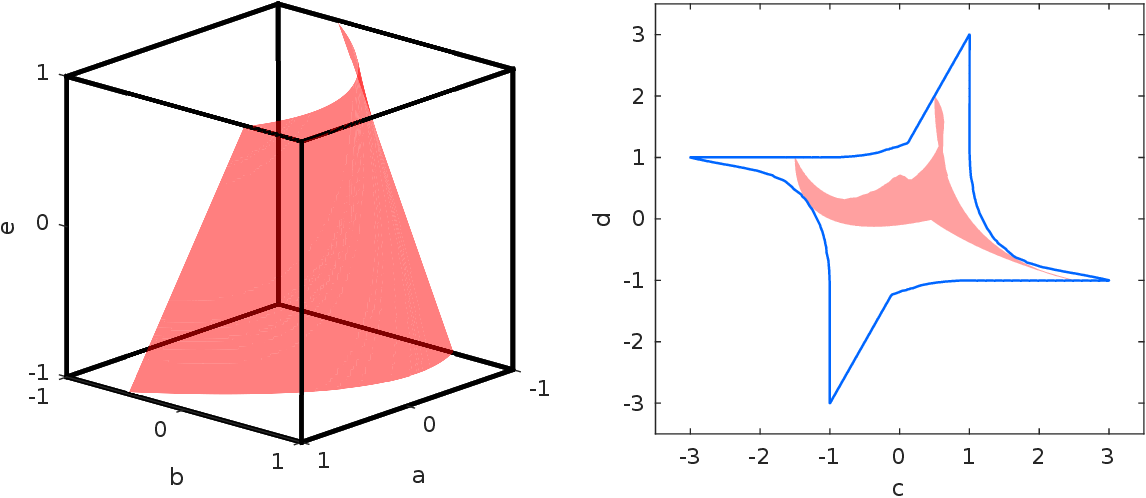}{
  \psfrag{a}[c][c]{$\alpha_1$}
  \psfrag{b}[c][c]{$\alpha_2$}
  \psfrag{c}[c][c]{$x_1$}
  \psfrag{d}[c][c]{\rotatebox[origin=c]{180}{$x_2$}}
  \psfrag{e}[c][c]{\rotatebox[origin=c]{180}{$\alpha_3$}}
  }
  \caption{Visualization of the polynomial constraint (left), the constrained polynomial zonotope (right, red), and the corresponding constructing polynomial zonotope (right, blue) for $\mathcal{CPZ}$ from Example \ref{ex:CPZ}.}
  \label{fig:Example}
\end{figure}


\section{Preliminaries}
\label{sec:preliminaryOperations}

We begin with some preliminary results that are required throughout this paper. 

\subsection{Identities}

Let us first establish some identities that are useful for subsequent derivations. According to the definition of CPZs in Def.~\ref{def:CPZ}, it holds that 
\begin{equation}
	\begin{split}
	& \bigg \{c + \sum_{i=1}^{h_1} \bigg( \prod_{k=1}^p \alpha_k^{E_{1(k,i)}} \bigg) G_{1(\cdot,i)} + \sum_{i=1}^{h_2} \bigg( \prod_{k=1}^p \alpha_k^{E_{2(k,i)}} \bigg) G_{2(\cdot,i)} ~ \bigg | \\
	& ~~~ \sum_{i=1}^{q} \bigg( \prod_{k=1}^p \alpha_k^{\CE_{(k,i)}} \bigg) A_{(\cdot,i)} = b, ~ \alpha_k \in [\shortminus 1,1]   \bigg \} = \big \langle c,[G_1~G_2],[E_1~E_2],A,b,\CE \big \rangle_{CPZ}
	\end{split}
	\label{eq:sumIdentity}
\end{equation}
and
\begin{align}
	& \bigg \{ c + \sum_{i=1}^{h} \bigg( \prod_{k=1}^p \alpha_k^{E_{(k,i)}} \bigg) G_{(\cdot,i)} ~ \bigg | ~ \sum_{i=1}^{q_1} \bigg( \prod_{k=1}^{p} \alpha_k^{\CE_{1(k,i)}} \bigg) A_{1(\cdot,i)} = b_1, \nonumber \\
	& ~~~~~~~~~~~~~~~~~~~~~~~~~~~~~~~~~~~~~~~~ \sum_{i=1}^{q_2} \bigg( \prod_{k=1}^{p} \alpha_k^{\CE_{2(k,i)}} \bigg) A_{2(\cdot,i)} = b_2, ~ \alpha_k \in [\shortminus 1,1]   \bigg \} \label{eq:conIdentity} \\
	& ~~ \nonumber \\
	& = \bigg \langle c,G,E,\begin{bmatrix}A_1 & \mathbf{0} \\ \mathbf{0} & A_2 \end{bmatrix},\begin{bmatrix} b_1 \\ b_2 \end{bmatrix},\begin{bmatrix} \CE_1 & \CE_2 \end{bmatrix} \bigg \rangle_{CPZ}. \nonumber
\end{align}

\subsection{Transformation to a Regular Representation}

Some set operations result in a CPZ that is not regular. We therefore introduce operations that transform a non-regular CPZ into a regular one. The \operator{compactGen} operation returns a CPZ with a regular exponent matrix:

\begin{proposition}
	(Compact Generators) Given $\mathcal{CPZ} = \langle c,G,E,A,b,\CE \rangle_{CPZ} \subset \Rn$, the operation \operator{compactGen} returns a representation of $\mathcal{CPZ}$ with a regular exponent matrix and has complexity $\mathcal{O}(p h \log (h) + nh)$:
	\begin{equation*}
		\operator{compactGen}(\mathcal{CPZ}) = \bigg \langle \underbrace{c + \sum_{i \in \mathcal{K}} G_{(\cdot,i)}}_{\overline{c}}, \underbrace{\bigg[ \sum_{i \in \mathcal{H}_1} G_{(\cdot,i)} ~ \dots ~ \sum_{i \in \mathcal{H}_w} G_{(\cdot,i)} \bigg]}_{\overline{G}}, \overline{E}, A,b,\CE \bigg \rangle_{CPZ}
	\end{equation*}
	with
	\setlength{\jot}{12pt}
	\begin{gather*}
		\mathcal{K} = \big\{i~\big|~\forall k \in \{1,\dots,p\},~ E_{(k,i)} = 0 \big \}, ~~ \overline{E} = \operator{uniqueColumns}\big( E_{(\cdot,\mathcal{N})} \big) \in \mathbb{N}_{0}^{p \times w}, \\
		\mathcal{N} = \{1,\dots,h\} \setminus \mathcal{K},~~ \mathcal{H}_j = \big \{ i~ \big|~ \forall k \in \{1, \dots, p\}, ~ \overline{E}_{(k,j)} = E_{(k,i)}  \big \},~~ j = 1,\dots,w,
	\end{gather*}
	where the operation \operator{uniqueColumns} removes identical matrix columns until all columns are unique.
	\label{prop:compactGen}
\end{proposition}

\begin{proof}
	For a CPZ where the exponent matrix $E = [e ~ e]$ consists of two identical columns $e \in \mathbb{N}_{0}^{p}$, it holds that 
	\begin{equation*}
	\begin{split}
		\bigg \{ \bigg( \prod _{k=1}^p \alpha _k ^{e_{(k)}} \bigg) G_{(\cdot,1)} + \bigg( \prod _{k=1}^p \alpha _k ^{e_{(k)}} \bigg) & G_{(\cdot,2)} ~\bigg| ~ \alpha_k \in [\shortminus 1,1] \bigg \} = \\
		& \bigg \{  \bigg( \prod _{k=1}^p \alpha _k ^{e_{(k)}} \bigg) \bigg ( G_{(\cdot,1)} + G_{(\cdot,2)} \bigg) ~\bigg| ~ \alpha_k \in [\shortminus 1,1] \bigg \}.
	\end{split}
	\end{equation*}
	Summation of the generators for terms $\alpha_1^{e_{(1)}} \cdot \ldots \cdot \alpha_p^{e_{(p)}}$ with identical exponents therefore does not change the set, which proves that $\operator{compactGen}(\mathcal{CPZ}) = \mathcal{CPZ}$. In addition, since the operation \operator{uniqueColomns} removes all identical matrix columns and we add all-zero columns to the constant offset, it holds that the resulting exponent matrix $\overline{E}$ is regular according to Def.~\ref{def:CPZ}.

\Comp We assume that the operation \operator{uniqueColumns} in combination with the construction of the sets $\mathcal{H}_j$ is implemented by first sorting the matrix columns, followed by an identification of identical neighbors. Moreover, we assume that in order to sort the matrix columns one first sorts the entries in the first row. For all columns with identical entries in the first row one then sorts the columns according to the entries in the second row. Since this process is continued for all $p$ matrix rows and the complexity for sorting one row of the matrix $E_{(\cdot,\mathcal{N})} \in \R^{p \times |\mathcal{N}|}$ is $\mathcal{O}(|\mathcal{N}| \log(|\mathcal{N}|))$ \cite[Chapter~5]{Knuth1997}, sorting the matrix columns has a worst-case complexity of $\mathcal{O}(p |\mathcal{N}| \log (|\mathcal{N}|))$, which is $\Oc(ph \log(h))$ since $|\mathcal{N}| \leq h$. The identification and removal of identical neighbors requires at most $p(h-1)$ comparison operations and therefore has worst-case complexity $\mathcal{O}(p(h-1))$. Moreover, construction of the sets $\mathcal{K}$ and $\mathcal{N}$ has complexity $\Oc(ph)$ in the worst case. Finally, the construction of the constant offset $\overline{c}$ and the generator matrix $\overline{G}$ has complexity $\mathcal{O}(nh)$ in the worst case. The overall complexity is therefore $\mathcal{O}(p h \log (h)) + \mathcal{O}(p(h-1)) + \Oc(ph) + \mathcal{O}(nh) = \mathcal{O}(p h \log (h) + nh)$. \hfill $\square$
\end{proof}
The \operator{compactCon} operation returns a CPZ with a regular constraint exponent matrix:	

\begin{proposition}
	(Compact Constraints) Given $\mathcal{CPZ} = \langle c,G,E,A,b,\CE \rangle_{CPZ} \subset \Rn$, the operation \operator{compactCon} returns a representation of $\mathcal{CPZ}$ with a regular constraint exponent matrix and has complexity $\mathcal{O}(p q \log (q) + mq)$:
	\begin{equation*}
		\operator{compactCon}(\mathcal{CPZ}) = \bigg\langle c, G, E, \bigg[ \sum_{i \in \mathcal{H}_1} A_{(\cdot,i)} ~ \dots ~ \sum_{i \in \mathcal{H}_w} A_{(\cdot,i)} \bigg], b - \sum_{i\in \mathcal{K}} A_{(\cdot,i)}, \overline{\CE} \bigg \rangle_{CPZ}
	\end{equation*}
	with
	\setlength{\jot}{12pt}
	\begin{gather*}
		\mathcal{K} = \big \{ i~ \big | ~ \forall k \in \{1,\dots,p\},~ \CE_{(k,i)} = 0 \big \},~~ \overline{\CE} = \operator{uniqueColumns}( \CE_{(\cdot,\mathcal{N})} ) \in \mathbb{N}_{0}^{p \times w},\\
		\mathcal{N} = \{1,\dots,q\} \setminus \mathcal{K}, ~~ \mathcal{H}_j = \big \{ i~ \big|~ \forall k \in \{1, \dots, p\}, ~ \overline{\CE}_{(k,j)} = \CE_{(k,i)} \big \},~~ j = 1,\dots,w,
	\end{gather*}
	where the operation \operator{uniqueColumns} removes identical matrix columns until all columns are unique.
	\label{prop:compactCon}
\end{proposition}

\begin{proof}
	The proof is analogous to the proof for Prop. \ref{prop:compactGen}. \hfill $\square$
\end{proof}

\subsection{Lifted Polynomial Zonotopes}

Finally, we introduce the lifted polynomial zonotope corresponding to a CPZ in the following lemma, which is inspired by \cite[Prop.~3]{Scott2016}:
\begin{lemma}
	(Lifted Polynomial Zonotope) Given $\mathcal{CPZ} = \langle c,G,E,A,b,\CE \rangle_{CPZ} \subset \R^n$, the corresponding lifted polynomial zonotope $\mathcal{PZ}^+ \subset \R^{n+m}$ defined as
	\begin{equation}
		\mathcal{PZ}^+ = \bigg\langle \begin{bmatrix} c \\ \shortminus b \end{bmatrix}, \underbrace{\begin{bmatrix} G & \mathbf{0} \\ \mathbf{0} & A \end{bmatrix}}_{\overline{G}},\underbrace{\begin{bmatrix} E & \CE \end{bmatrix}}_{\overline{E}} \bigg\rangle_{PZ}
		\label{eq:liftedPolyZono}
	\end{equation}
	satisfies
	\begin{equation*}
		\forall x \in \Rn,~~ \big( x \in \mathcal{CPZ} \big) \Leftrightarrow \bigg( \begin{bmatrix} x \\ \mathbf{0} \end{bmatrix} \in \mathcal{PZ}^+ \bigg). 
	\end{equation*}	
	\label{lemma:liftCPZ}
\end{lemma}
\begin{proof}
With the definition of CPZs in Def.~\ref{def:CPZ} we obtain
	\begin{equation*}
	\setlength{\jot}{12pt} 
	\begin{split}
		& \big( x \in \mathcal{CPZ} \big) \overset{\substack{\text{Def.}~\ref{def:CPZ}\\ \vspace{-2pt}}}{\Leftrightarrow}\\
		& \bigg( \exists \alpha \in [\shortminus \mathbf{1},\mathbf{1}],~ \bigg( x = c + \sum_{i=1}^{h} \bigg( \prod_{k=1}^p \alpha_k^{E_{(k,i)}} \bigg) G_{(\cdot,i)} \bigg) \wedge \bigg( \sum_{i=1}^{q} \bigg( \prod_{k=1}^p \alpha_k^{\CE_{(k,i)}} \bigg) A_{(\cdot,i)} = b \bigg) \bigg) \\
		& \overset{\substack{\eqref{eq:liftedPolyZono}\\ \vspace{-2pt}}}{\Leftrightarrow} \bigg( \exists \alpha \in [\shortminus \mathbf{1},\mathbf{1}],~ \bigg( \begin{bmatrix} x \\ \mathbf{0} \end{bmatrix} = \begin{bmatrix}c \\ \shortminus b\end{bmatrix} + \sum_{i=1}^{h+q} \bigg( \prod_{k=1}^p \alpha_k^{\overline{E}_{(k,i)}} \bigg) \overline{G}_{(\cdot,i)} \bigg) \overset{\substack{\eqref{eq:liftedPolyZono} \\ \vspace{-2pt}}}{\Leftrightarrow} \bigg( \begin{bmatrix} x \\ \mathbf{0} \end{bmatrix} \in \mathcal{PZ}^+ \bigg),
	\end{split}
	\end{equation*}
	where $\alpha = [\alpha_1 ~\dots ~\alpha_p]^T$. \hfill $\square$
\end{proof}
According to Lemma~\ref{lemma:liftCPZ}, a CPZ can be interpreted as the intersection of the lifted polynomial zonotope $\mathcal{PZ}^+$ with the subspace $\{x \in \R^{n+m}~|~ x_{(n+1)},\dots,x_{(n+m)} = 0 \}$. Moreover, with the lifted polynomial zonotope we can transfer results for polynomial zonotopes to CPZs, as we demonstrate later. Potential redundancies in the lifted polynomial zonotope due to common columns in the exponent and the constraint exponent matrix can be removed using the \operator{compact} operation for polynomial zonotopes in \cite[Prop.~2]{Kochdumper2019}.

\subsection{Rescaling}
\label{sec:rescaling}

Later, in Sec.~\ref{sec:enclosure} and Sec.~\ref{sec:complexityReduction}, we describe how to enclose CPZs by other set representations and how to reduce the representation size of a CPZ by enclosing it with a simpler CPZ. The tightness of these enclosures mainly depends on the size of the corresponding constructing polynomial zonotope. Since the constraints often intersect only part of the factor hypercube $\alpha_1,\dots,\alpha_p \in [\shortminus 1,1]$, we can reduce the size of the constructing polynomial zonotope in advance to obtain tighter results. This can be achieved with a contractor:
\begin{definition}
	(Contractor) \cite[Chapter~4.1]{Jaulin2006} Given an interval $\mathcal{I} \subset \R^p$ and a vector field $f:~\R^p \to \R^m$ which defines the constraint $f(x) = \mathbf{0}$, the operation $\operator{contract}$ returns an interval that satisfies
	\begin{equation*}
		\operator{contract}\big(f(x),\mathcal{I}\big) \subseteq \mathcal{I} 
	\end{equation*}
	and
	\begin{equation*}
		\forall x \in \mathcal{I}, ~~ \big( f(x) = \mathbf{0}\big) \Rightarrow \big( x \in \operator{contract}\big(f(x),\mathcal{I}\big)\big),
	\end{equation*}
	so that it is guaranteed that all solutions for $f(x) = \mathbf{0}$ contained in $\mathcal{I}$ are also contained in the contracted interval.
	\label{def:contractor}
\end{definition}
There exist many sophisticated approaches for implementing a contractor, an overview of which is provided in \cite[Chapter~4]{Jaulin2006}. Given $\mathcal{CPZ} = \langle c,G,E,A,b,\CE \rangle_{CPZ} \subset \Rn$, we can compute a tighter domain $\alpha_1,\dots,\alpha_p \in [l,u] \subseteq [\shortminus \mathbf{1},\mathbf{1}]$ for the factors by applying a contractor to the polynomial constraint of the CPZ:
\begin{equation*}
	[l,u] = \operator{contract}(f(x),[\shortminus \mathbf{1},\mathbf{1}]), ~~ f(x) = \sum_{i=1}^{q} \bigg( \prod_{k=1}^p x_{(k)}^{\CE_{(k,i)}} \bigg) A_{(\cdot,i)} - b.
\end{equation*}
Using the contracted domain $[l,u]$, the $\mathcal{CPZ}$ can be equivalently represented as
\begin{equation*}
	\mathcal{CPZ} = \bigg \{ c + \sum_{i=1}^{h} \bigg( \prod_{k=1}^p \alpha_k^{E_{(k,i)}} \bigg) G_{(\cdot,i)} \, \bigg | \, \sum_{i=1}^{q} \bigg( \prod_{k=1}^p \alpha_k^{\CE_{(k,i)}} \bigg) A_{(\cdot,i)} = b, \, [\alpha_1~\dots~\alpha_p]^T \in [l,u] \bigg \}.
\end{equation*}
We show in Appendix~\ref{app:rescaling} that this set can be represented as a CPZ. Let us demonstrate rescaling by an example:

\begin{example}
	We consider the CPZ 
	\begin{equation*}
	\begin{split}
		\mathcal{CPZ} = \bigg \langle & \begin{bmatrix} 0 \\ 0 \end{bmatrix}, \begin{bmatrix} 2 & 0 & 0 & 0.4 \\ 0 & \shortminus 2 & 1 & 0.2 \end{bmatrix}, \begin{bmatrix} 1 & 1 & 2 & 0 \\ 0 & 1 & 1 & 0 \\ 0 & 0 & 0 & 1 \end{bmatrix},  \begin{bmatrix} 1 & 2 & 1 & 2 & 1 \end{bmatrix}, \shortminus 2, \begin{bmatrix} 2 & 1 & 0 & 0 & 0 \\ 0 & 0 & 2 & 1 & 0 \\ 0 & 0 & 0 & 0 & 1 \end{bmatrix} \bigg \rangle_{CPZ},
	\end{split}
	\end{equation*}
	which is visualized in Fig.~\ref{fig:rescaleCPZ}. As depicted on the left side of Fig.~\ref{fig:rescaleCPZ}, the constraint only intersects a small part of the factor domain $\alpha_1,\alpha_2,\alpha_3 \in [\shortminus 1,1]$, so that the domain can be contracted to $\alpha_1,\alpha_2,\alpha_3 \in [\shortminus 1,0]$. Rescaling therefore significantly reduces the size of the constructing polynomial zonotope, as visualized on the right side of Fig.~\ref{fig:rescaleCPZ}.
	\label{ex:rescale}
\end{example}

\begin{figure}
  \centering
  \psfragfig[width=0.95\columnwidth]{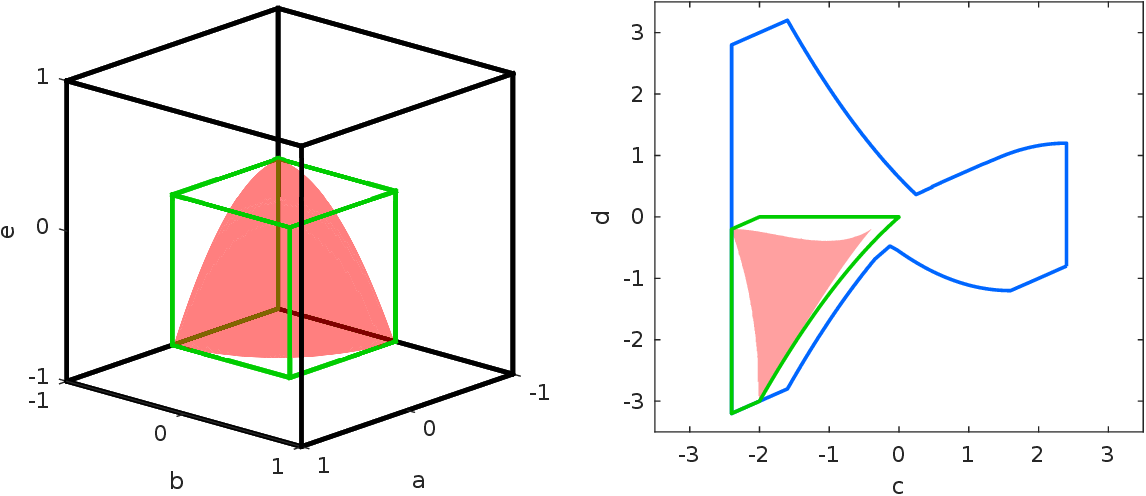}{
  \psfrag{a}[c][c]{$\alpha_1$}
  \psfrag{b}[c][c]{$\alpha_2$}
  \psfrag{c}[c][c]{$x_1$}
  \psfrag{d}[c][c]{\rotatebox[origin=c]{180}{$x_2$}}
  \psfrag{e}[c][c]{\rotatebox[origin=c]{180}{$\alpha_3$}}
  }
  \caption{Visualization of rescaling for $\mathcal{CPZ}$ from Example~\ref{ex:rescale} (red,~right), where the corresponding constraint is visualized on the left. The constructing polynomial zonotope before rescaling is shown in blue, and the constructing polynomial zonotope after rescaling is shown in green.}
  \label{fig:rescaleCPZ}
\end{figure}


\section{Conversion from Other Set Representations}

This section shows how other set representations can be converted to CPZs.

\subsection{Taylor Models, Intervals, and Zonotopic Set Representations}
\label{subsec:polynomialZonotope}

Since a polynomial zonotope is simply a CPZ without constraints, the conversion is trivial in this case. For polynomial zonotopes that are defined with additional independent generators as in \cite[Def.~1]{Kochdumper2019}, one can first convert the polynomial zonotope to a polynomial zonotope without independent generators using \cite[Prop.~1]{Kochdumper2020c}. According to \cite[Prop. 4]{Kochdumper2019}, the set defined by a Taylor model can be equivalently represented as a polynomial zonotope. Moreover, according to \cite[Prop.~3]{Kochdumper2019} any zonotope can be represented as a polynomial zonotope, and any interval can be represented as a zonotope \cite[Prop.~2.1]{Althoff2010a}. Finally, a constrained zonotope is a special case of a CPZ where all polynomial functions are linear, so the conversion is straightforward. In summary, we therefore obtain the following conversion rules:
\begin{alignat}{2}
    & \text{Interval:} && \mathcal{I} = [l,u] = \langle 0.5(u+l),0.5 \, \text{diag}(u-l),I_n,[~],[~],[~] \rangle_{CPZ} \label{eq:convInterval}\\[7pt]
    & \text{Zonotope:} && \mathcal{Z} = \langle c,G \rangle_Z = \langle c,G,I_p,[~],[~],[~] \rangle_{CPZ} \label{eq:convZono}\\[7pt]
    & \text{Constrained zonotope:} ~~~ && \mathcal{CZ} = \langle c,G,A,b \rangle_{CZ} = \langle c, G, I_p, A, b, I_p \big \rangle_{CPZ} \label{eq:convConZono}\\[7pt]
	& \text{Polynomial zonotope:} ~~ && \mathcal{PZ} = \langle c,G,E \rangle_{PZ} = \langle c,G,E,[~],[~],[~] \rangle_{CPZ} \label{eq:convPolyZono}
\end{alignat}
The conversion of an interval has complexity $\mathcal{O}(n)$ with respect to the dimension $n$ due to the summation and subtraction of the vectors $l$ and $u$, while all other conversions have constant complexity $\mathcal{O}(1)$ since no computations are required.

\subsection{Polytopes}
\label{subsec:polytope}

There are two possibilities to represent a bounded polytope as a CPZ. According to \cite{Kochdumper2021} and \cite[Theorem 1]{Kochdumper2019}, every bounded polytope can be represented as a polynomial zonotope. Therefore, any bounded polytope can be converted to a CPZ by first representing it as a polynomial zonotope followed by a conversion of the polynomial zonotope to a CPZ using \eqref{eq:convPolyZono}. Moreover, it holds according to \cite[Theorem 1]{Scott2016} that any bounded polytope can be represented as a constrained zonotope. Consequently, the second possibility for the conversion of a bounded polytope to a CPZ is to first represent the polytope as a constrained zonotope, and then convert the constrained zonotope to a CPZ using \eqref{eq:convConZono}. Which of the two methods results in a more compact representation depends on the polytope.

\subsection{Ellipsoids}

Any ellipsoid can be converted to a CPZ:

\begin{proposition}
	(Conversion Ellipsoid) An ellipsoid $\mathcal{E} = \langle c,Q \rangle_E \subset \Rn$ can be equivalently represented by a CPZ:
	\begin{equation}
		\mathcal{E} =  \bigg \langle c, \underbrace{V \begin{bmatrix} \sqrt{\lambda_1} & & 0 \\ & \ddots & \\ 0 & & \sqrt{\lambda_n} \end{bmatrix}}_{G}, \underbrace{\begin{bmatrix} I_n \\ \mathbf{0} \end{bmatrix}}_{E}, \underbrace{\begin{bmatrix} \shortminus 0.5 & \mathbf{1} \end{bmatrix} }_{A},\underbrace{0.5}_{b}, \underbrace{\begin{bmatrix} \mathbf{0} & 2 I_n \\ 1 & \mathbf{0} \end{bmatrix}}_{\CE} \bigg \rangle_{CPZ},
		\label{eq:ellipsoid}
	\end{equation}
	where the eigenvalues $\lambda_1,\dots,\lambda_n$, the matrix of eigenvalues $D$, and the matrix of eigenvectors $V$ are obtained by the eigenvalue decomposition
	\begin{equation}
		Q = V \underbrace{\begin{bmatrix} \lambda_1 & & 0 \\ & \ddots & \\ 0 & & \lambda_n \end{bmatrix}}_{D} V^T.
		\label{eq:eigenvalue}
	\end{equation}
	The complexity of the conversion is $\mathcal{O}(n^3)$.
\end{proposition}

\begin{proof}
	The matrices $A,\CE$ and the vector $b$ in \eqref{eq:ellipsoid} define the constraint
	\begin{equation}
		\shortminus 0.5 \, \alpha_{n+1} + \alpha_1^2 + \dotsc + \alpha_n^2 = 0.5.
		\label{eq:proofEllipse1}
	\end{equation}
	Since $\alpha_{n+1} \in [\shortminus 1,1]$, \eqref{eq:proofEllipse1} is equivalent to the constraint
	\begin{equation}
		0 \leq \alpha_1^2 + \dotsc + \alpha_n^2 \leq 1.
		\label{eq:proofEllipse3}
	\end{equation} 
	Using the eigenvalue decomposition of the matrix $Q$ from \eqref{eq:eigenvalue} it holds that
	\begin{equation}
		Q^{-1} \overset{\eqref{eq:eigenvalue}}{=} (VDV^T)^{-1} = V D^{-1} V^T
		\label{eq:eigInverse}
	\end{equation}
	since $V$ is an orthonormal matrix satisfying $V^{-1} = V^T$. Inserting \eqref{eq:eigInverse} into the definition of an ellipsoid in Def.~\ref{def:ellipsoid} yields
	\begin{equation}
	\setlength{\jot}{12pt}
		\begin{split}
			\mathcal{E} \overset{\substack{\text{Def.}~\ref{def:ellipsoid}\\ \vspace{-3pt}}}{=}& \big \{ x ~ \big | ~ (x-c)^T Q^{-1} (x-c) \leq 1 \big \} = \big \{ c + x ~ \big | ~ x^T Q^{-1} x \leq 1 \big \} \overset{\substack{\eqref{eq:eigInverse} \\ \vspace{-3pt}}}{=} \\
			&\big \{ c + x ~ \big | ~ (V^T x)^T D^{-1} (V^T x) \leq 1 \big \} \overset{\substack{z:=V^Tx \\ \vspace{-3pt}}}{=} \\
			 &\big \{ c + Vz ~ \big | ~ z^T D^{-1} z \leq 1 \big \} \overset{\substack{\eqref{eq:eigenvalue} \\ \vspace{-3pt}}}{=} \bigg \{ c + Vz ~ \bigg | ~ \frac{z_{(1)}^2}{\lambda_1} + \dotsc + \frac{z_{(n)}^2}{\lambda_n} \leq 1 \bigg \}.
		\end{split}
		\label{eq:proofEllipse2}
	\end{equation}
	We define the factors $\alpha_k$ of the CPZ as $\alpha_k = \frac{z_{(k)}}{\sqrt{\lambda_k}}$, $k = 1,\dots,n$, so that 
	\begin{equation}
		z_{(k)} = \sqrt{\lambda_k} ~ \alpha_k.
		\label{eq:varTrans}
	\end{equation}
	Inserting \eqref{eq:varTrans} into \eqref{eq:proofEllipse2} finally yields
	\begin{equation*}
	\setlength{\jot}{12pt}
		\begin{split}
		& \bigg \{ c + Vz ~ \bigg | ~ \frac{z_{(1)}^2}{\lambda_1} + \dotsc + \frac{z_{(n)}^2}{\lambda_n} \leq 1 \bigg \} \overset{\substack{\eqref{eq:varTrans} \\ \vspace{-3pt}}}{=} \bigg \{ c + \sum_{k=1}^n \sqrt{\lambda_k}~ \alpha_k \, V_{(\cdot,k)} ~ \bigg | ~ \alpha_1^2 + \dotsc + \alpha_n^2 \leq 1 \bigg \} \overset{\substack{\eqref{eq:proofEllipse1} \\ \eqref{eq:proofEllipse3} \\ \vspace{-3pt}}}{=} \\
		&  \bigg \{ c + \sum_{k=1}^n \sqrt{\lambda_k} ~ \alpha_k \, V_{(\cdot,k)} ~ \bigg | ~ \shortminus 0.5 \, \alpha_{n+1} + \alpha_1^2 + \dotsc + \alpha_n^2 = 0.5, ~ \alpha_1,\dots,\alpha_{n+1} \in [\shortminus 1,1] \bigg \} \\ 
		& \overset{\substack{\eqref{eq:ellipsoid}\\ \vspace{-3pt}}}{=} \langle c,G,E,A,b,\CE \rangle_{CPZ},
		\end{split}
	\end{equation*}
	which concludes the proof. 

\Comp Computation of the eigenvalue decomposition $Q = V^T D V$ in \eqref{eq:eigenvalue} has complexity $\mathcal{O}(n^3)$ \cite{Pan1999}. The computation of $G$ in \eqref{eq:ellipsoid} requires $n^2$ multiplications and the calculation of $n$ square roots and therefore has complexity $\mathcal{O}(n^2) + \mathcal{O}(n) = \mathcal{O}(n^2)$. Since all other required operations are concatenations, the overall complexity results by adding the complexity of the eigenvalue decomposition and the complexity of computing $G$, which yields $\mathcal{O}(n^2) + \mathcal{O}(n^3) = \mathcal{O}(n^3)$. \hfill $\square$
\end{proof}


\section{Enclosure by Other Set Representations}
\label{sec:enclosure}

To speed up computations, one often encloses sets by simpler set representations in set-based computing. In this section, we therefore show how to enclose CPZs by constrained zonotopes, polynomial zonotopes, zonotopes, and intervals. The over-approximation error for all enclosures can be reduced by applying rescaling as described in Sec.~\ref{sec:rescaling} in advance. To demonstrate the tightness of the enclosures, we use the CPZ
\begin{equation}
	\mathcal{CPZ} = \bigg \langle \begin{bmatrix} 0 \\ 0 \end{bmatrix}, \begin{bmatrix} 1 & 0.5 & 1 & 0.5 \\ 0 & 1 & 1 & 0.5 \end{bmatrix}, \begin{bmatrix} 1 & 0 & 2 & 0 \\ 0 & 1 & 1 & 0 \\ 0 & 0 & 0 & 1 \end{bmatrix} \begin{bmatrix} 1 & \shortminus 0.5 & 0.5 \end{bmatrix}, 0.5, \begin{bmatrix} 1 & 0 & 0 \\ 0 & 1 & 2 \\ 0 & 1 & 0 \end{bmatrix} \bigg \rangle_{CPZ}
	\label{eq:exampleCPZenclose}
\end{equation}
as a running example throughout this section.

\subsection{Constrained Zonotopes}

We first show how to enclose a CPZ by a constrained zonotope:

\begin{proposition}
	(Constrained Zonotope Enclosure) Given $\mathcal{CPZ} = \langle c,G,E,A,b,\linebreak[3] \CE \rangle_{CPZ} \subset \R^n$, the operation \operator{conZono} returns a constrained zonotope that encloses $\mathcal{CPZ}$:
	\begin{equation*}
		\mathcal{CPZ} \subseteq \operator{conZono}(\mathcal{CPZ}) = \underbrace{\langle c_z,G_z,A_z,\shortminus b_z \rangle_{CZ}}_{\mathcal{CZ}}
	\end{equation*}
	with
	\begin{equation*}
		\underbrace{\bigg \langle \begin{bmatrix} c_z \\ b_z \end{bmatrix} ,\begin{bmatrix}G_z \\ A_z \end{bmatrix} \bigg \rangle_Z}_{\mathcal{Z}^+} = \operator{zono}\bigg( \operator{compact} \bigg( \underbrace{ \bigg\langle \begin{bmatrix} c \\ \shortminus b \end{bmatrix},\begin{bmatrix} G & \mathbf{0} \\ \mathbf{0} & A \end{bmatrix},\begin{bmatrix} E & \CE \end{bmatrix} \bigg\rangle_{PZ}}_{\mathcal{PZ}^+} \bigg) \bigg),
	\end{equation*}
    where the \operator{compact} operation as defined in \cite[Prop.~2]{Kochdumper2019} returns a regular polynomial zonotope and the \operator{zono} operation as defined in \cite[Prop.~5]{Kochdumper2019} returns a zonotope that encloses a polynomial zonotope. The computational complexity is $\mathcal{O}(\mu^2)$ with respect to the representation size $\mu$ and $\mathcal{O}(n^2 \log(n))$ with respect to the dimension $n$.
	\label{prop:conZonoEncloseCPZ}
\end{proposition}
\begin{proof} 
	To obtain an enclosing constrained zonotope we calculate a zonotope enclosure of the corresponding lifted polynomial zonotope as defined in Lemma~\ref{lemma:liftCPZ}. Back-transformation of the lifted zonotope to the original state space then yields an enclosing constrained zonotope:
	\begin{equation*}
	\setlength{\jot}{12pt} 
	\begin{split}
		\forall x \in \Rn, ~~ (x \in \mathcal{CPZ}) & \overset{\substack{\text{Lemma}~\ref{lemma:liftCPZ}\\ \vspace{-2pt} }}{\Rightarrow} \bigg( \begin{bmatrix} x \\ \mathbf{0}  \end{bmatrix} \in \mathcal{PZ}^+ \bigg) \\
		& \overset{\substack{\mathcal{PZ}^+ \subseteq \mathcal{Z}^+\\ \vspace{-1pt}}}{\Rightarrow} \bigg( \begin{bmatrix} x \\ \mathbf{0}  \end{bmatrix} \in \mathcal{Z}^+ \bigg) \overset{\substack{\text{Lemma}~\ref{lemma:liftCPZ}\\ \vspace{-2pt}}}{\Rightarrow} (x \in \mathcal{CZ}),
	\end{split}
	\end{equation*}
where we omitted the \operator{compact} operation since it only changes the representation of the set, but not the set itself.
	
	\Comp Let $n^+ = n+m$, $p^+ = p$, and $h^+ = h + q$ denote the dimension, the number of factors, and the number of generators of the lifted polynomial zonotope $\mathcal{PZ}^+$. According to \cite[Prop.~2]{Kochdumper2019}, the \operator{compact} operation for polynomial zonotopes has complexity $\mathcal{O}( p^+ h^+ \log(h^+)) = \mathcal{O}(p(h+q)\log(h+q))$. Moreover, the complexity for the \operator{zono} operation is $\mathcal{O}(p^+ h^+) + \mathcal{O}(n^+ h^+) = \mathcal{O}(p(h+q)) + \mathcal{O}((n+m)(h+q))$ according to \cite[Prop.~5]{Kochdumper2019}. The overall computational complexity is therefore 
	\begin{equation*}
		\mathcal{O}\big(\underbrace{p(h+q)\log(h+q)}_{\overset{\eqref{eq:repSize}}{\leq} \mu \log(\mu)}\big) + \mathcal{O}\big(\underbrace{p(h+q)}_{\overset{\eqref{eq:repSize}}{\leq} \mu}\big) + \mathcal{O}\big(\underbrace{(n+m)(h+q)}_{\overset{\eqref{eq:repSize}}{\leq} \mu^2}\big) = \mathcal{O}(\mu^2),
	\end{equation*}
	which is $\mathcal{O}(n^2 \log(n))$ using \eqref{eq:complexity}. \hfill $\square$ 
\end{proof}
The enclosing constrained zonotope for the CPZ in \eqref{eq:exampleCPZenclose} is shown in Fig.~\ref{fig:enclosure}.

\subsection{Polynomial Zonotopes} 

Clearly, an enclosing polynomial zonotope for a CPZ can simply be obtained by dropping the constraints. However, this might yield large over-approximation errors. Another possibility is to reduce all constraints using Prop.~\ref{prop:reduceConCPZ} introduced later in Sec.~\ref{sec:complexityReduction}. Which method results in the tighter enclosure depends on the CPZ. The resulting enclosing polynomial zonotope for the CPZ in \eqref{eq:exampleCPZenclose} obtained by dropping the constraints is visualized in Fig.~\ref{fig:enclosure}.

\subsection{Zonotopes and Intervals}

An enclosure of a CPZ by a zonotope or interval can be computed using the previously presented enclosures by constrained zonotopes or polynomial zonotopes. 
For polynomial zonotopes, an enclosing zonotope can be computed using \cite[Prop.~5]{Kochdumper2019}, and an enclosing interval can be computed based on the support function enclosure in \cite[Prop.~7]{Kochdumper2019}. For constrained zonotopes, an enclosing zonotope can be calculated by reducing all constraints as described in \cite[Sec.~4.2]{Scott2016}, and an enclosing interval can be computed using linear programming \cite[Prop.~1]{Rego2018}. 
 
\begin{figure}
  \centering
  \psfragfig[width=0.95\columnwidth]{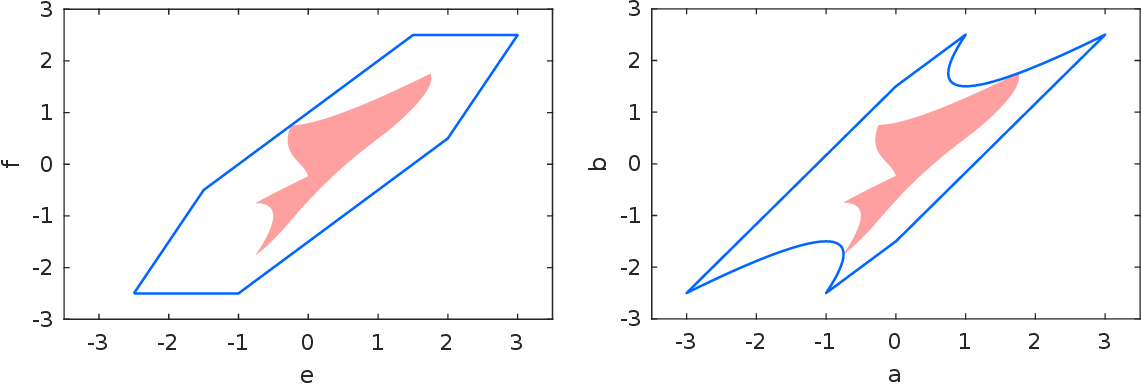}{
  \psfrag{a}[c][c]{$x_1$}
  \psfrag{e}[c][c]{$x_2$}
  \psfrag{b}[c][c]{\rotatebox[origin=c]{180}{$x_2$}}
  \psfrag{f}[c][c]{\rotatebox[origin=c]{180}{$x_2$}}
  }
  \caption{Enlosing constrained zonotope (left) and enclosing polynomial zonotope (right) for $\mathcal{CPZ}$ in \eqref{eq:exampleCPZenclose}.}
  \label{fig:enclosure}
\end{figure}


\section{Set Operations}

In this section, we derive closed-form expressions for all set operations introduced in Sec.~\ref{sec:Preliminaries} on CPZs. We begin with the linear map:

\begin{proposition}
	(Linear Map) Given $\mathcal{CPZ} = \langle c,G,E,A,b,\CE \rangle_{CPZ} \subset \Rn$ and a matrix $M \in \R^{w \times n}$, the linear map is
	\begin{equation*}
		M \otimes \mathcal{CPZ} = \langle M c, M G,E,A,b,\CE \rangle_{CPZ},
	\end{equation*}
	which has complexity $\mathcal{O}(w \mu)$ with respect to the representation size $\mu$ and complexity $\mathcal{O}(w n^2)$ with respect to the dimension $n$. The resulting CPZ is regular if $\mathcal{CPZ}$ is regular. 
	\label{prop:linearMap}
\end{proposition}

\begin{proof}
	The result follows directly from inserting the definition of CPZs in Def.~\ref{def:CPZ} into the definition of the operator $\otimes$ in \eqref{eq:defLinTrans}.

\Comp The complexity results from the complexity of matrix multiplications and is therefore $\mathcal{O}(wnh) + \mathcal{O}(wn) = \mathcal{O}(wnh)$. Since $nh \leq \mu$ according to \eqref{eq:repSize}, it holds that $\mathcal{O}(wnh) = \mathcal{O}(w\mu)$. Using \eqref{eq:complexity}, it furthermore holds that $\mathcal{O}(wnh) = \mathcal{O}(w n^2)$. \hfill $\square$
\end{proof}
Next, we consider the Minkowski sum:

\begin{proposition}
	(Minkowski Sum) Given $\mathcal{CPZ}_1 = \langle c_1,G_1, E_1, A_1, b_1, \CE_1 \rangle_{CPZ} \subset \Rn$ and $\mathcal{CPZ}_2 = \langle c_2, G_2, E_2, A_2,b_2, \CE_2 \rangle_{CPZ} \subset \Rn$, their Minkowski sum is
	\begin{equation*}
			\mathcal{CPZ}_1 \oplus \mathcal{CPZ}_2 = \bigg \langle c_1 + c_2, \begin{bmatrix} G_1 & G_2 \end{bmatrix}, \begin{bmatrix} E_1 & \mathbf{0} \\ \mathbf{0} & E_2 \end{bmatrix}, \begin{bmatrix} A_1 & \mathbf{0} \\ \mathbf{0} & A_2 \end{bmatrix}, \begin{bmatrix} b_1 \\ b_2 \end{bmatrix}, \begin{bmatrix} \CE_1 & \mathbf{0} \\ \mathbf{0} & \CE_2 \end{bmatrix}  \bigg \rangle_{CPZ},
	\end{equation*}	
	which has complexity $\mathcal{O}(n)$ with respect to the dimension $n$. The resulting CPZ is regular if $\mathcal{CPZ}_1$ and $\mathcal{CPZ}_2$ are regular.
	\label{prop:addition}
\end{proposition}
\begin{proof}
	The result is obtained by inserting the definition of CPZs in Def.~\ref{def:CPZ} into the definition of the Minkowski sum in \eqref{eq:defMinSum}:
	\begin{equation*}
		\begin{split}
		&\mathcal{CPZ}_1 \oplus \mathcal{CPZ}_2 \overset{\substack{ \eqref{eq:defMinSum} \\ \vspace{-3pt}}}{=} \big\{ s_1 + s_2 ~ \big|~ s_1 \in \mathcal{CPZ}_1,~s_2 \in \mathcal{CPZ}_2 \big\}  \overset{\substack{\text{Def.}~\ref{def:CPZ} \\ \vspace{-3pt}}}{=} \\
		& ~ \\
		& \bigg \{ c_1 + c_2 + \sum_{i=1}^{h_1} \bigg( \prod_{k=1}^{p_1} \alpha_{k}^{E_{1(k,i)}} \bigg) G_{1(\cdot,i)} + \sum_{i=1}^{h_2} \bigg( \prod_{k=1}^{p_2} \alpha_{p_1+k}^{E_{2(k,i)}} \bigg) G_{2(\cdot,i)} ~ \bigg | \\
		& ~~ \sum_{i=1}^{q_1} \bigg( \prod_{k=1}^{p_1} \alpha_k^{\CE_{1(k,i)}} \bigg) A_{1(\cdot,i)} = b_1, ~ \sum_{i=1}^{q_2} \bigg( \prod_{k=1}^{p_2} \alpha_{p_1+k}^{\CE_{2(k,i)}} \bigg) A_{2(\cdot,i)} = b_2, ~ \alpha_k,\alpha_{p_1 + k} \in [\shortminus 1,1]   \bigg \} \\
		& ~ \\
		& \overset{\substack{\eqref{eq:sumIdentity},\eqref{eq:conIdentity} \\ \vspace{-3pt} }}{=} \bigg \langle c_1 + c_2, \begin{bmatrix} G_1 & G_2 \end{bmatrix}, \begin{bmatrix} E_1 & \mathbf{0} \\ \mathbf{0} & E_2 \end{bmatrix}, \begin{bmatrix} A_1 & \mathbf{0} \\ \mathbf{0} & A_2 \end{bmatrix}, \begin{bmatrix} b_1 \\ b_2 \end{bmatrix}, \begin{bmatrix} \CE_1 & \mathbf{0} \\ \mathbf{0} & \CE_2 \end{bmatrix}  \bigg \rangle_{CPZ},
		\end{split}
	\end{equation*}
	where we used the identities \eqref{eq:sumIdentity} and \eqref{eq:conIdentity}.

\Comp The computation of the new constant offset $c_1 + c_2$ has complexity $\mathcal{O}(n)$. Since all other operations required for the construction of the resulting CPZ are concatenations, it holds that the overall complexity is $\mathcal{O}(n)$. \hfill $\square$
\end{proof}
Now, we provide a closed-form expression for the Cartesian product:

\begin{proposition}
	(Cartesian Product) Given $\mathcal{CPZ}_1 = \langle c_1, G_1, E_1, A_1,b_1, \CE_1 \rangle_{CPZ} \linebreak[3] \subset \Rn$ and $\mathcal{CPZ}_2 = \langle c_2, G_2, E_2, A_2, b_2, \CE_2 \rangle_{CPZ} \subset \R^w$, their Cartesian product is
	\begin{equation*}
			\mathcal{CPZ}_1 \times \mathcal{CPZ}_2 = \bigg \langle \begin{bmatrix} c_1 \\ c_2 \end{bmatrix} \begin{bmatrix} G_1 & \mathbf{0} \\ \mathbf{0} & G_2 \end{bmatrix}, \begin{bmatrix} E_1 & \mathbf{0} \\ \mathbf{0} & E_2 \end{bmatrix}, \begin{bmatrix} A_1 & \mathbf{0} \\ \mathbf{0} & A_2 \end{bmatrix}, \begin{bmatrix} b_1 \\ b_2 \end{bmatrix}, \begin{bmatrix} \CE_1 & \mathbf{0} \\ \mathbf{0} & \CE_2 \end{bmatrix}  \bigg \rangle_{CPZ},
	\end{equation*}	
	which has complexity $\mathcal{O}(1)$. The resulting CPZ is regular if $\mathcal{CPZ}_1$ and $\mathcal{CPZ}_2$ are regular.
	\label{prop:cartProduct}
\end{proposition}

\begin{proof}
	The result is obtained by inserting the definition of CPZs in Def.~\ref{def:CPZ} into the definition of the Cartesian product in \eqref{eq:defCartProduct}:
	\begin{equation*}
		\begin{split}
		&\mathcal{CPZ}_1 \times \mathcal{CPZ}_2 \overset{\substack{\eqref{eq:defCartProduct} \\ \vspace{-3pt}}}{=} \big\{ [s_1^T ~s_2^T]^T ~\big|~s_1 \in \mathcal{CPZ}_1,~s_2 \in \mathcal{CPZ}_2 \big\} \overset{\substack{\text{Def.}~\ref{def:CPZ} \\ \vspace{-4pt}}}{=} \\
		& ~ \\
		& \bigg \{ \begin{bmatrix} c_1 \\ \mathbf{0} \end{bmatrix} + \begin{bmatrix} \mathbf{0} \\ c_2 \end{bmatrix} + \sum_{i=1}^{h_1} \bigg( \prod_{k=1}^{p_1} \alpha_k^{E_{1(k,i)}} \bigg) \begin{bmatrix} G_{1(\cdot,i)} \\ \mathbf{0} \end{bmatrix} + \sum_{i=1}^{h_2} \bigg( \prod_{k=1}^{p_2} \alpha_{p_1 + k}^{E_{2(k,i)}} \bigg) \begin{bmatrix} \mathbf{0} \\ G_{2(\cdot,i)} \end{bmatrix}  ~ \bigg | \\
		& ~~~ \sum_{i=1}^{q_1} \bigg( \prod_{k=1}^{p_1} \alpha_k^{\CE_{1(k,i)}} \bigg) A_{1(\cdot,i)} = b_1, ~ \sum_{i=1}^{q_2} \bigg( \prod_{k=1}^{p_2} \alpha_{p_1 + k}^{\CE_{2(k,i)}} \bigg) A_{2(\cdot,i)} = b_2, ~ \alpha_k,\alpha_{p_1+k} \in [\shortminus 1,1]   \bigg \} \\
		& ~ \\
		& \overset{\substack{\eqref{eq:sumIdentity},\eqref{eq:conIdentity} \\ \vspace{-3pt}}}{=} \bigg \langle \begin{bmatrix} c_1 \\ c_2 \end{bmatrix} \begin{bmatrix} G_1 & \mathbf{0} \\ \mathbf{0} & G_2 \end{bmatrix}, \begin{bmatrix} E_1 & \mathbf{0} \\ \mathbf{0} & E_2 \end{bmatrix}, \begin{bmatrix} A_1 & \mathbf{0} \\ \mathbf{0} & A_2 \end{bmatrix}, \begin{bmatrix} b_1 \\ b_2 \end{bmatrix}, \begin{bmatrix} \CE_1 & \mathbf{0} \\ \mathbf{0} & \CE_2 \end{bmatrix}  \bigg \rangle_{CPZ},
		\end{split}
	\end{equation*}
	where we used the identities in \eqref{eq:sumIdentity} and \eqref{eq:conIdentity}.

\Comp The construction of the resulting CPZ only involves concatenations and therefore has constant complexity $\mathcal{O}(1)$. \hfill $\square$
\end{proof}
Before we examine the convex hull, we first derive a closed-form expression for the linear combination since we can reuse this result for the convex hull:

\begin{proposition}
	(Linear Combination) Given $\mathcal{CPZ}_1 = \langle c_1, G_1, E_1, A_1, b_1, \linebreak[3] \CE_1 \rangle_{CPZ} \subset \Rn$ and $\mathcal{CPZ}_2 = \langle c_2, G_2, E_2,A_2, b_2, \CE_2 \rangle_{CPZ} \subset \Rn$, their linear combination is
	\begin{equation*}
	\begin{split}
			comb(\mathcal{CPZ}_1, \mathcal{CPZ}_2) = \bigg \langle & \frac{1}{2} (c_1 + c_2), \frac{1}{2} \begin{bmatrix} (c_1 - c_2) & G_1 & G_1 & G_2 & \shortminus G_2 \end{bmatrix}, \\
		&  \begin{bmatrix} \mathbf{0} & E_1 & E_1 & \mathbf{0} & \mathbf{0} \\ \mathbf{0} &  \mathbf{0} & \mathbf{0} & E_2 & E_2 \\ 1 & \mathbf{0} & \mathbf{1} & \mathbf{0} & \mathbf{1} \end{bmatrix}, \begin{bmatrix} A_1 & \mathbf{0} \\ \mathbf{0} & A_2 \end{bmatrix},\begin{bmatrix} b_1 \\ b_2 \end{bmatrix}, \begin{bmatrix} \CE_1 & \mathbf{0} \\ \mathbf{0} & \CE_2 \\ \mathbf{0} & \mathbf{0} \end{bmatrix}  \bigg \rangle_{CPZ},
		\end{split}
	\end{equation*}	
	which has complexity $\mathcal{O}(\mu_1 + \mu_2)$ with respect to the representation sizes $\mu_1$ and $\mu_2$ and complexity $\mathcal{O}(n^2)$ with respect to the dimension $n$. The resulting CPZ is regular if $\mathcal{CPZ}_1$ and $\mathcal{CPZ}_2$ are regular.  
	\label{prop:linComb}
\end{proposition}

\begin{proof} 
The result is obtained by inserting the definition of CPZs in Def.~\ref{def:CPZ} into the definition of the linear combination in \eqref{eq:defLinComb}:
	\begin{align*}
		& comb(\mathcal{CPZ}_1, \mathcal{CPZ}_2) \overset{\substack{ \eqref{eq:defLinComb} \\ \vspace{-3pt}}}{=} \bigg\{ \frac{1+\lambda}{2} \, s_1 + \frac{1-\lambda}{2} \, s_2 \, \bigg| \, s_1 \in \mathcal{CPZ}_1,\,s_2 \in \mathcal{CPZ}_2, \, \lambda \in [\shortminus 1,1] \bigg\} \overset{\substack{\text{Def.}~\ref{def:CPZ} \\ \vspace{-4pt}}}{=} \\
		& ~ \\
		&  \bigg \{ \frac{1}{2}(c_1 + c_2) + \frac{1}{2}(c_1 - c_2) \lambda +  \frac{1}{2} \sum_{i=1}^{h_1} \bigg( \prod_{k=1}^{p_1} \alpha_k^{E_{1(k,i)}} \bigg) G_{1(\cdot,i)} +  \frac{1}{2} \sum_{i=1}^{h_1} \lambda \bigg( \prod_{k=1}^{p_1} \alpha_k^{E_{1(k,i)}} \bigg) G_{1(\cdot,i)} \\
		& ~~ + \frac{1}{2} \sum_{i=1}^{h_2} \bigg( \prod_{k=1}^{p_2} \alpha_{p_1+k}^{E_{2(k,i)}} \bigg) G_{2(\cdot,i)} - \frac{1}{2} \sum_{i=1}^{h_2} \lambda \bigg( \prod_{k=1}^{p_2} \alpha_{p_1+k}^{E_{2(k,i)}} \bigg) G_{2(\cdot,i)} ~ \bigg | ~  \\
		& ~~~ \sum_{i=1}^{q_1} \bigg( \prod_{k=1}^{p_1} \alpha_k^{\CE_{1(k,i)}} \bigg) A_{1(\cdot,i)} = b_1,~\sum_{i=1}^{q_2} \bigg( \prod_{k=1}^{p_2} \alpha_{p_1+k}^{\CE_{2(k,i)}} \bigg) A_{2(\cdot,i)} = b_2, ~ \alpha_k,\alpha_{p_1+k},\lambda \in [\shortminus 1,1]  \bigg \} \\
		& ~ \\
		& \overset{\substack{\eqref{eq:sumIdentity},\eqref{eq:conIdentity} \\ \alpha_{p_1 + p_2 + 1} := \lambda \\ \vspace{-1pt} }}{=} \bigg \langle  \frac{1}{2} (c_1 + c_2), \frac{1}{2} \begin{bmatrix} (c_1 - c_2) & G_1 & G_1 & G_2 & \shortminus G_2 \end{bmatrix}, \\
		& ~~~~~~~~~~~~~~~~~~ \begin{bmatrix} \mathbf{0} & E_1 & E_1 & \mathbf{0} & \mathbf{0} \\ \mathbf{0} &  \mathbf{0} & \mathbf{0} & E_2 & E_2 \\ 1 & \mathbf{0} & \mathbf{1} & \mathbf{0} & \mathbf{1} \end{bmatrix}, \begin{bmatrix} A_1 & \mathbf{0} \\ \mathbf{0} & A_2 \end{bmatrix},\begin{bmatrix} b_1 \\ b_2 \end{bmatrix}, \begin{bmatrix} \CE_1 & \mathbf{0} \\ \mathbf{0} & \CE_2 \\ \mathbf{0} & \mathbf{0} \end{bmatrix}  \bigg \rangle_{CPZ},
	\end{align*}
	where we used the identities in \eqref{eq:sumIdentity} and \eqref{eq:conIdentity}. For the transformation in the last line, we substituted $\lambda$ with an additional factor $\alpha_{p_1 + p_2 + 1}$. Since $\lambda \in [\shortminus 1,1]$ and $\alpha_{p_1 + p_2 + 1} \in [\shortminus 1,1]$, the substitution does not change the set. 

\Comp The construction of the constant offset $c = 0.5(c_1 + c_2)$ requires $n$ additions and $n$ multiplications. Moreover, the construction of the generator matrix $G = 0.5 \begin{bmatrix} (c_1 - c_2) & G_1 & G_1 & G_2 & \shortminus G_2 \end{bmatrix}$ requires $n$ subtractions and $n (2 h_1+2 h_2+ 1)$ multiplications. The overall complexity is therefore 
\begin{equation}
	\mathcal{O}(2n) + \mathcal{O}\big(n (2h_1 + 2 h_2 + 2)\big) = \mathcal{O}\big(\underbrace{n (h_1 + h_2)}_{\overset{\eqref{eq:repSize}}{\leq} \mu_1 + \mu_2}\big) = \mathcal{O}(\mu_1 + \mu_2),
	\label{eq:compLinComb}
\end{equation}
which is $\mathcal{O}(n^2)$ using \eqref{eq:complexity}. \hfill $\square$
\end{proof}
The convex hull can be computed based on the linear combination:

\begin{proposition}
	(Convex Hull) Given $\mathcal{CPZ}_1 = \langle c_1, G_1, E_1, A_1, b_1, \CE_1 \rangle_{CPZ} \subset \Rn$ and $\mathcal{CPZ}_2 \linebreak[3] = \langle c_2, G_2, E_2,A_2, b_2, \CE_2 \rangle_{CPZ} \subset \Rn$, their convex hull is
	\begin{equation*}
			conv(\mathcal{CPZ}_1, \mathcal{CPZ}_2) = \bigg \langle  a \, c,\begin{bmatrix} \overline{c} & \overline{G} & \overline{G} \end{bmatrix},\begin{bmatrix} \mathbf{0} & \overline{E} & \overline{E} \\ I_{a} & \mathbf{0} & \widehat{E} \end{bmatrix}, \begin{bmatrix} \overline{A} & \mathbf{0} \\ \mathbf{0} & \mathbf{1} \end{bmatrix}, \begin{bmatrix} \overline{b} \\ \shortminus n \end{bmatrix},\begin{bmatrix} \overline{\CE} & \mathbf{0} \\ \mathbf{0} & I_{a}  \end{bmatrix} \bigg\rangle_{CPZ},
	\end{equation*}	
	with
	\begin{equation}
	\setlength{\jot}{12pt} 
	\begin{split}
		& \langle c,G,E,A,b,\CE \rangle_{CPZ} = comb(\mathcal{CPZ}_1,\mathcal{CPZ}_2),~ a = n +1, ~ \overline{c} = \begin{bmatrix} c & \dots & c \end{bmatrix} \in \R^{n \times a}, ~  \\
		& \overline{G} = \begin{bmatrix} G & \dots & G \end{bmatrix} \in \R^{n \times ah},~\overline{E} = \begin{bmatrix} E & & \mathbf{0} \\ & \ddots &  \\ \mathbf{0} &  & E \end{bmatrix} \in \R^{ap \times ah},~ \widehat{E} = \begin{bmatrix} \mathbf{1} &  & \mathbf{0} \\  & \ddots &  \\ \mathbf{0} &  & \mathbf{1} \end{bmatrix} \in \R^{a \times ah}, \\
		& ~~~~ \overline{A} = \begin{bmatrix} A &  & \mathbf{0} \\  & \ddots &  \\ \mathbf{0} &  & A \end{bmatrix} \in \R^{a m \times a q},~\overline{b} = \begin{bmatrix} b \\ \vdots \\ b \end{bmatrix} \in \R^{am},~\overline{\CE} = \begin{bmatrix} \CE &  & \mathbf{0} \\  & \ddots &  \\ \mathbf{0} &  & \CE \end{bmatrix} \in \R^{ap \times a q},
	\end{split}
	\label{eq:convexHull2}
	\end{equation}	
	where the linear combination $comb(\mathcal{CPZ}_1,\mathcal{CPZ}_2)$ is calculated using Prop.~\ref{prop:linComb} and the scalars $p$, $h$, $q$, and $m$ denote respectively the number of factors, the number of generators, the number of constraint generators, and the number of constraints of the CPZ $\langle c,G,E,A,b,\CE \rangle_{CPZ}$. The complexity is $\mathcal{O}(\mu_1 + \mu_2)$ with respect to the representation sizes $\mu_1$ and $\mu_2$ and $\mathcal{O}(n^2)$ with respect to the dimension $n$. The resulting CPZ is regular if $\mathcal{CPZ}_1$ and $\mathcal{CPZ}_2$ are regular.  
	\label{prop:convHull}
\end{proposition}
\begin{proof}
	According to the definition of the convex hull in \eqref{eq:defConvHull}, the definition of the union in \eqref{eq:defUnion}, and the definition of the linear combination in \eqref{eq:defLinComb}, it holds that
	\begin{equation}
		\mathcal{CPZ}_1 \cup \mathcal{CPZ}_2 \subseteq comb(\mathcal{CPZ}_1,\mathcal{CPZ}_2) \subseteq conv(\mathcal{CPZ}_1,\mathcal{CPZ}_2).
		\label{eq:subsetsConvexHull}
	\end{equation}
	The relation in \eqref{eq:subsetsConvexHull} allows us to substitute the union in the definition of the convex hull in \eqref{eq:defConvHull} with the linear combination. This yields a resulting CPZ with fewer factors compared to using the union according to Theorem~\ref{theo:union}, which is often advantageous:
	\begin{align*}
		& conv(\mathcal{CPZ}_1,\mathcal{CPZ}_2) \overset{\eqref{eq:defConvHull}}{=} \bigg \{ \sum_{j=1}^{n+1} \lambda_j\, s_j ~ \bigg| ~ s_j \in \mathcal{CPZ}_1 \cup \mathcal{CPZ}_2,~\lambda_j \geq 0,~\sum_{j=1}^{n+1} \lambda_j = 1 \bigg\} \overset{\eqref{eq:subsetsConvexHull}}{=} \\
		& ~ \\[-5pt]
		& \bigg \{ \sum_{j=1}^{n+1} (1+\widehat{\lambda}_j) \, s_j ~ \bigg | ~ s_j \in comb(\mathcal{CPZ}_1,\mathcal{CPZ}_2),~\sum_{j=1}^{n+1} (1 + \widehat{\lambda}_j) = 1, ~ \widehat{\lambda}_j \in [\shortminus 1,1] \bigg\} \overset{\substack{\text{Def.}~\ref{def:CPZ} \\ \eqref{eq:convexHull2}}}{=}\\
		& ~ \\[-5pt]		
		& \bigg \{ \sum_{j=1}^{n+1} (1+\widehat{\lambda}_j) \bigg( c + \sum_{i = 1}^h \bigg( \prod_{k = 1}^p \alpha_{(j-1)p + k}^{E_{(k,i)}} \bigg) G_{(\cdot,i)} \bigg) ~ \bigg | ~  \alpha_{(j-1)p + k},\widehat{\lambda}_j \in [\shortminus 1,1] , \\
		& ~~~\sum_{j=1}^{n+1} (1 + \widehat{\lambda}_j) = 1,~ \forall j \in \{1,\dots,n+1\}:~\sum_{i=1}^q \bigg( \prod_{k=1}^p \alpha_{(j-1)p + k}^{\CE_{(k,i)}}\bigg) A_{(\cdot,i)} = b \bigg\} = \\
		& ~ \\[-5pt]
		& \bigg\{ (n+1)c + \sum_{j=1}^{n+1} \widehat{\lambda}_j \, c + \sum_{j=1}^{n+1} \sum_{i = 1}^h \bigg( \prod_{k = 1}^p \alpha_{(j-1)p + k}^{E_{(k,i)}} \bigg) G_{(\cdot,i)} \\
		& ~~ + \sum_{j=1}^{n+1} \sum_{i = 1}^h \widehat{\lambda}_j \bigg( \prod_{k = 1}^p \alpha_{(j-1)p + k}^{E_{(k,i)}} \bigg) G_{(\cdot,i)}  ~ \bigg| ~ \alpha_{(j-1)p + k},\widehat{\lambda}_j \in [\shortminus 1,1], \\
		& ~~~~\forall j \in \{1,\dots,n+1\}:~\sum_{i=1}^q \bigg( \prod_{k=1}^p \alpha_{(j-1)p + k}^{\CE_{(k,i)}}\bigg) A_{(\cdot,i)} = b,~\sum_{j=1}^{n+1} \widehat{\lambda}_j = \shortminus n \bigg \}  \\
		& ~ \\[-5pt]
		& \overset{\substack{\eqref{eq:sumIdentity},\eqref{eq:conIdentity} \\ \alpha_{ap + j} := \widehat{\lambda}_j \\ \vspace{-1pt} }}{=} \bigg \langle  a \, c,\begin{bmatrix} \overline{c} & \overline{G} & \overline{G} \end{bmatrix},\begin{bmatrix} \mathbf{0} & \overline{E} & \overline{E} \\ I_{a} & \mathbf{0} & \widehat{E} \end{bmatrix}, \begin{bmatrix} \overline{A} & \mathbf{0} \\ \mathbf{0} & \mathbf{1} \end{bmatrix}, \begin{bmatrix} \overline{b} \\ \shortminus n \end{bmatrix},\begin{bmatrix} \overline{\CE} & \mathbf{0} \\ \mathbf{0} & I_{a}  \end{bmatrix} \bigg\rangle_{CPZ},
	\end{align*}
	where we used the identities in \eqref{eq:sumIdentity} and \eqref{eq:conIdentity}. For the transformation in the last line, we substituted the scalars $\widehat{\lambda}_j$ by  additional factors $\alpha_{ap + j}$. Since $\widehat{\lambda}_j \in [\shortminus 1,1]$ and $\alpha_{ap + j} \in [\shortminus 1,1]$, the substitution does not change the set.
	
	\Comp The calculation of the linear combination $comb(\mathcal{CPZ}_1,\mathcal{CPZ}_2)$ using Prop.~\ref{prop:linComb} has complexity $\mathcal{O}(n (h_1 + h_2))$ according to \eqref{eq:compLinComb}. Moreover, the construction of the constant offset $a \, c$ requires $n$ multiplications and therefore has complexity $\Oc(n)$. Since all other operations that are required are initializations and concatenations which have constant complexity $\Oc(1)$, the overall complexity for the computation of the convex hull is
	\begin{equation*}
		\mathcal{O}\big(n (h_1 + h_2)\big) + \mathcal{O}(n) + \mathcal{O}(1) = \mathcal{O}\big(\underbrace{n (h_1 + h_2)}_{\overset{\eqref{eq:repSize}}{\leq} \mu_1 + \mu_2}\big) = \mathcal{O}(\mu_1 + \mu_2),
	\end{equation*}
	which is $\mathcal{O}(n^2)$ using \eqref{eq:complexity}. \hfill $\square$
\end{proof}
For the convex hull $conv(\mathcal{CPZ}) = conv(\mathcal{CPZ},\mathcal{CPZ})$ of a single set $\mathcal{CPZ}$, we can exploit that $\mathcal{CPZ} \cup \mathcal{CPZ} = \mathcal{CPZ}$ holds to obtain a more compact representation. Next, we consider the quadratic map:

\begin{proposition}
	(Quadratic Map) Given $\mathcal{CPZ} = \langle c,G,E,A,b,\CE \rangle_{CPZ} \subset \mathbb{R}^n$ and a discrete set of matrices $\mathcal{Q} = \{ Q_1,\dots,Q_w \}$ with $Q_{i} \in \mathbb{R}^{n \times n}, i = 1, \dots, w$, the quadratic map is
	\begin{equation*}
		sq(\mathcal{Q},\mathcal{CPZ}) = \bigg \langle \overline{c}, \begin{bmatrix} \widehat{G}_1 + \widehat{G}_2 & \overline{G}_1 & \dots & \overline{G}_{h} \end{bmatrix}, \begin{bmatrix} E & \overline{E}_1 & \dots & \overline{E}_{h} \end{bmatrix},A,b,\CE \bigg \rangle_{CPZ}
	\end{equation*}
	with
	\begin{gather*}
		\overline{c} = \begin{bmatrix} c^T Q_{1} c \\ \vdots \\ c^T Q_{w} c \end{bmatrix}, ~~ \widehat{G}_1 = \begin{bmatrix} c^T Q_{1} G \\ \vdots \\ c^T Q_{w} G \end{bmatrix}, ~~ \widehat{G}_2 = \begin{bmatrix} c^T Q_{1}^T G \\ \vdots \\ c^T Q_{w}^T G \end{bmatrix}, \\
		\overline{E}_j = E + E_{(\cdot,j)} \, \mathbf{1} , ~~ \overline{G}_j = \begin{bmatrix} G_{(\cdot,j)}^T Q_{1} G \\ \vdots \\ G_{(\cdot,j)}^T Q_{w} G \end{bmatrix}, ~ j = 1, \dots, h.
	\end{gather*}
	The \operator{compactGen} operation is applied to obtain a regular CPZ. The complexity is $\mathcal{O}(\mu^2 w) + \mathcal{O}(\mu^2 \log(\mu))$ with respect to the representation size $\mu$ and $\mathcal{O}(n^3( w + \log(n)))$ with respect to the dimension $n$.
	\label{prop:quadMap}
\end{proposition}

\begin{proof}
	The result is obtained by inserting the definition of CPZs in Def.~ \ref{def:CPZ} into the definition of the quadratic map in \eqref{eq:defQuadMap}, which yields
	\begin{align*}
			& sq(\mathcal{Q},\mathcal{CPZ}) \overset{\substack{\eqref{eq:defQuadMap} \\ \vspace{-3pt}}}{=} \big\{ x~\big|~ x_{(i)} = s^T Q_i s,~ s \in \mathcal{CPZ},~ i = 1,\dots,w \} \overset{\substack{ \text{Def.}~\ref{def:CPZ} \\ \vspace{-3pt}}}{=} \\
			& ~ \\			
			& \bigg\{ x ~ \bigg| ~ x_{(i)} = \bigg( c + \sum _{j=1}^{h} \bigg( \prod _{k=1}^{p} \alpha _k ^{E_{(k,j)}} \bigg)  G_{(\cdot,j)} \bigg)^T Q_i \bigg( c + \sum _{l=1}^{h} \bigg( \prod _{k=1}^{p} \alpha _k ^{E_{(k,l)}} \bigg) G_{(\cdot,l)} \bigg), \\
			&  ~~~~~~~~~~~~~~~~~~~~~~~~~~~~~ \sum_{j=1}^{q} \bigg( \prod_{k=1}^p \alpha_k^{\CE_{(k,j)}} \bigg) A_{(\cdot,j)} = b, ~i = 1, \dots, w,~ \alpha_k \in [\shortminus 1,1] \bigg \} = \\
			& ~ \\
			& \bigg \{ x ~ \bigg | ~ x_{(i)} = \underbrace{c^T Q_i c}_{ \overline{c}_{(i)}} + \sum_{l=1}^h \bigg( \prod_{k=1}^p \alpha_k^{E_{(k,l)}} \bigg) \underbrace{c^T Q_i G_{(\cdot,l)}}_{\widehat{G}_{1(i,l)}} + \sum_{j=1}^h \bigg( \prod_{k=1}^p \alpha_k^{E_{(k,j)}} \bigg) \underbrace{ G_{(\cdot,j)}^T Q_i c }_{\widehat{G}_{2(i,j)}} \\
			&  ~~~~~~~~~~~~~~~ + \sum_{j = 1}^{h} \sum_{l = 1}^{h} \bigg( \prod _{k=1}^{p} \underbrace{\alpha _k ^{E_{(k,j)} + E_{(k,l)}}}_{ \alpha_k^{\overline{E}_{j(k,l)}}} \bigg) \underbrace{G_{(\cdot,j)}^T Q_i G_{(\cdot,l)}}_{ \overline{G}_{j(i,l)}}, \\
			&  ~~~~~~~~~~~~~~~~~~~~~~~~~~~~~~ \sum_{j=1}^{q} \bigg( \prod_{k=1}^p \alpha_k^{\CE_{(k,j)}} \bigg) A_{(\cdot,j)} = b, ~i = 1, \dots, w,~ \alpha_k \in [\shortminus 1,1] \bigg \} = \\
			& ~~ \\
			& \bigg \langle \overline{c}, \begin{bmatrix} \widehat{G}_1 + \widehat{G}_2 & \overline{G}_1 & \dots & \overline{G}_{h} \end{bmatrix}, \begin{bmatrix} E & \overline{E}_1 & \dots & \overline{E}_{h} \end{bmatrix},A,b,\CE \bigg \rangle_{CPZ}.
	\end{align*}	
Note that only the generator matrix, but not the exponent matrix, is different for each dimension $x_{(i)}$.

\Comp The construction of the constant offset $\overline{c}$ has complexity $\mathcal{O}(w n^2)$ and the construction of the matrices $\widehat{G}_1$ and $\widehat{G}_2$ has complexity $\mathcal{O}(n^2 h w)$. Moreover, the construction of the matrices $\overline{E}_j$ has complexity $\mathcal{O}(h^2 p)$, and the construction of the matrices $\overline{G}_j$ has complexity $\mathcal{O}(n^2 h w) + \mathcal{O}(n h^2 w)$ if the results for $Q_i G$ are stored and reused. The resulting CPZ has dimension $\overline{n} = w$ and consists of $\overline{h} = h^2 + h$ generators. Consequently, subsequent application of the \operator{compactGen} operation has complexity $\mathcal{O}(p \overline{h} \log(\overline{h}) + \overline{n}\overline{h}) = \mathcal{O}(p (h^2 + h) \log(h^2 + h) + w(h^2 + h))$ according to Prop.~\ref{prop:compactGen}. The resulting overall complexity is
\begin{equation*}
\begin{split}
	& \mathcal{O}(w n^2) + \mathcal{O}(n^2 h w) + \mathcal{O}(h^2 p) + \mathcal{O}(n^2 h w) + \mathcal{O}(n h^2 w) \\
	& ~~~~~~~~~~~~~~~~~~~~~~~~~~~~~~~~ + \mathcal{O}\big(p (h^2 + h) \log(h^2 + h) + w(h^2 + h)\big) = \\
	& ~ \\
	& \mathcal{O}(\underbrace{n^2 h w}_{\overset{\eqref{eq:repSize}}{\leq} \mu^2 w}) + \mathcal{O}(\underbrace{n h^2 w}_{\overset{\eqref{eq:repSize}}{\leq} \mu^2 w}) + \mathcal{O}\big(\underbrace{p (h^2 + h)}_{\overset{\eqref{eq:repSize}}{\leq} \mu^2} \log(\underbrace{h^2 + h}_{\overset{\eqref{eq:repSize}}{\leq} \mu^2 }) + \underbrace{ w(h^2 + h)}_{\overset{\eqref{eq:repSize}}{\leq} \mu^2 w}\big) =  \mathcal{O}(\mu^2 w) + \mathcal{O}(\mu^2 \log(\mu)),
\end{split}
\end{equation*}
which is $\mathcal{O}(n^3( w + \log(n)))$ using \eqref{eq:complexity}. \hfill $\square$
\end{proof}
The extension to cubic or higher-order maps of sets as well as the extension to mixed quadratic maps involving two different CPZs are straightforward and therefore omitted. We continue with the intersection:

\begin{proposition}
	(Intersection) Given $\mathcal{CPZ}_1 = \langle c_1,G_1, E_1, A_1, b_1, \CE_1 \rangle_{CPZ} \subset \Rn$ and $\mathcal{CPZ}_2 = \langle c_2, G_2, E_2,A_2, b_2, \CE_2 \rangle_{CPZ} \subset \Rn$, their intersection is
	\begin{equation*}
			\mathcal{CPZ}_1 \cap \mathcal{CPZ}_2 = \bigg \langle c_1, G_1, \begin{bmatrix} E_1 \\ \mathbf{0} \end{bmatrix}, \begin{bmatrix} A_1 & \mathbf{0} & \mathbf{0} & \mathbf{0} \\ \mathbf{0} & A_2 & \mathbf{0} & \mathbf{0}  \\ \mathbf{0} & \mathbf{0} & G_1 & \shortminus G_2 \end{bmatrix}, \begin{bmatrix} b_1 \\ b_2 \\ c_2 - c_1 \end{bmatrix}, \begin{bmatrix} \CE_1 & \mathbf{0} & E_1 & \mathbf{0} \\ \mathbf{0} & \CE_2 & \mathbf{0} & E_2 \end{bmatrix}  \bigg \rangle_{CPZ},
	\end{equation*}	
	which has complexity $\mathcal{O}\left((\mu_1 + \mu_2)^2 \log (\mu_1 + \mu_2)\right)$ with respect to the representation sizes $\mu_1$ and $\mu_2$ and complexity $\mathcal{O}(n^2 \log(n))$ with respect to the dimension $n$. The \operator{compactCon} operation is applied to obtain a regular CPZ.
	\label{prop:intersection}
\end{proposition}

\begin{proof}
	The outline of the proof is inspired by \cite[Prop. 1]{Scott2016}. We compute the intersection by restricting the factors $\alpha_k$ of $\mathcal{CPZ}_1$ to values that belong to points that are located inside $\mathcal{CPZ}_2$, which is identical to adding the equality constraint
	\begin{equation*}
		\underbrace{c_1 + \sum_{i=1}^{h_1} \bigg( \prod_{k=1}^{p_1} \alpha_k^{E_{1(k,i)}} \bigg) G_{1(\cdot,i)}}_{x \, \in \, \mathcal{CPZ}_1} = \underbrace{c_2 + \sum_{i=1}^{h_2} \bigg( \prod_{k=1}^{p_2} \alpha_{p_1+k}^{E_{2(k,i)}} \bigg) G_{2(\cdot,i)}}_{x \, \in \, \mathcal{CPZ}_2}
	\end{equation*}
	to $\mathcal{CPZ}_1$:
	\begin{align*}
		& \mathcal{CPZ}_1 \cap \mathcal{CPZ}_2 = \bigg \{ c_1 + \sum_{i=1}^{h_1} \bigg( \prod_{k=1}^{p_1} \alpha_k^{E_{1(k,i)}} \bigg) G_{1(\cdot,i)} ~ \bigg | ~ \sum_{i=1}^{q_1} \bigg( \prod_{k=1}^{p_1} \alpha_k^{\CE_{1(k,i)}} \bigg) A_{1(\cdot,i)} = b_1, \\
		&   ~~~~~~~~~~~~~~~~~~~~~~~~~ \sum_{i=1}^{h_1} \bigg( \prod_{k=1}^{p_1} \alpha_k^{E_{1(k,i)}} \bigg) G_{1(\cdot,i)} - \sum_{i=1}^{h_2} \bigg( \prod_{k=1}^{p_2} \alpha_{p_1+k}^{E_{2(k,i)}} \bigg) G_{2(\cdot,i)} = c_2 - c_1, \\
		& ~~~~~~~~~~~~~~~~~~~~~~~~~~~~~~~~~~~~~~~~~~ \sum_{i=1}^{q_2} \bigg( \prod_{k=1}^{p_2} \alpha_{p_1+k}^{\CE_{2(k,i)}} \bigg) A_{2(\cdot,i)} = b_2,~\alpha_k,\alpha_{p_1+k} \in [\shortminus 1,1]   \bigg \} \\
		& ~ \\
		& \overset{\substack{\eqref{eq:conIdentity} \\ \vspace{-3pt}} }{=} \bigg \langle c_1, G_1, \begin{bmatrix} E_1 \\ \mathbf{0} \end{bmatrix}, \begin{bmatrix} A_1 & \mathbf{0} & \mathbf{0} & \mathbf{0} \\ \mathbf{0} & A_2 & \mathbf{0} & \mathbf{0}  \\ \mathbf{0} & \mathbf{0} & G_1 & \shortminus G_2 \end{bmatrix}, \begin{bmatrix} b_1 \\ b_2 \\ c_2 - c_1 \end{bmatrix}, \begin{bmatrix} \CE_1 & \mathbf{0} & E_1 & \mathbf{0} \\ \mathbf{0} & \CE_2 & \mathbf{0} & E_2 \end{bmatrix}  \bigg \rangle_{CPZ},
	\end{align*}
	where we used the identity in \eqref{eq:conIdentity}.

\Comp Computation of $c_2-c_1$ has complexity $\mathcal{O}(n)$. The resulting CPZ has $p = p_1 + p_2$ factors, $q = q_1 + q_2 + h_1 + h_2$ constraint generators, and $m = m_1 + m_2 + n$ constraints. Since the subsequent application of the \operator{compactCon} operation has complexity $\mathcal{O}(p q \log (q) + m q)$ according to Prop.~\ref{prop:compactCon}, we therefore obtain an overall complexity of 
\begin{equation*}
\begin{split}
	\mathcal{O}(n) &+ \mathcal{O}\big((\underbrace{p_1 + p_2}_{\overset{\eqref{eq:repSize}}{\leq} \mu_1 + \mu_2})(\underbrace{q_1 + q_2 + h_1 + h_2}_{\overset{\eqref{eq:repSize}}{\leq} \mu_1 + \mu_2}) \log (\underbrace{q_1 + q_2 + h_1 + h_2}_{\overset{\eqref{eq:repSize}}{\leq} \mu_1 + \mu_2})\big) \\
	& + \mathcal{O}\big( (\underbrace{m_1 + m_2 + n}_{\overset{\eqref{eq:repSize}}{\leq} \mu_1 + \mu_2})(\underbrace{q_1 + q_2 + h_1 + h_2}_{\overset{\eqref{eq:repSize}}{\leq} \mu_1 + \mu_2})\big) = \mathcal{O}\big((\mu_1 + \mu_2)^2 \log (\mu_1 + \mu_2)\big),
\end{split}
\end{equation*}
which is $\mathcal{O}(n^2 \log(n))$ using \eqref{eq:complexity}. \hfill $\square$
\end{proof}
As a last operation, we consider the union:

\begin{theorem}
	(Union) Given $\mathcal{CPZ}_1 = \langle c_1, G_1, E_1, A_1,b_1 \CE_1 \rangle_{CPZ} \subset \Rn$ and $\mathcal{CPZ}_2 = \langle c_2, G_2, \linebreak[3] E_2, A_2,b_2, \CE_2 \rangle_{CPZ} \subset \Rn$, their union is
	\begin{equation*}
		\begin{split}
			& \mathcal{CPZ}_1 \cup \mathcal{CPZ}_2 = \bigg \langle \underbrace{0.5(c_1 + c_2)}_{c}, \underbrace{\begin{bmatrix} 0.5(c_1 - c_2) & G_1 & G_2 \end{bmatrix}}_{G}, \underbrace{\begin{bmatrix} 1 & \mathbf{0} & \mathbf{0} \\ 0 & \mathbf{0} & \mathbf{0} \\ \mathbf{0} & E_1 & \mathbf{0} \\ \mathbf{0} & \mathbf{0} & E_2 \end{bmatrix}}_{E}, \\
			& ~~~~~~~~~~~~~~~~~~~~~~~~~ \underbrace{\begin{bmatrix} \widehat{A} & \mathbf{0} & \mathbf{0} & \mathbf{0} & 0 \\ \mathbf{0} & \overline{A} & \mathbf{0} & \mathbf{0} & 0 \\ \mathbf{0} & \mathbf{0} & A_1 & \mathbf{0} & \shortminus 0.5 \, b_1 \\ \mathbf{0} & \mathbf{0} & \mathbf{0} & A_2 & 0.5 \, b_2 \end{bmatrix}}_{A}, \underbrace{\begin{bmatrix} \widehat{b} \\ \overline{b} \\ 0.5 \, b_1 \\ 0.5 \, b_2 \end{bmatrix}}_{b}, \underbrace{\begin{bmatrix} \widehat{\CE} & \overline{\CE} & \begin{bmatrix} \mathbf{0} & \mathbf{0} & 1 \\ \mathbf{0} & \mathbf{0} & 0 \\ \CE_1 & \mathbf{0} & \mathbf{0} \\ \mathbf{0} & \CE_2 & \mathbf{0} \end{bmatrix}  \end{bmatrix}}_{\CE} \bigg \rangle_{CPZ}
		\end{split}
	\end{equation*}
	with
	\setlength{\jot}{12pt}
	\begin{gather*}
		\widehat{A} = 1, ~~ \widehat{b} = 1, ~~ \widehat{R} = \begin{bmatrix} 1 & 1 & \mathbf{0} \end{bmatrix}^T,\\
		\overline{A} = \begin{bmatrix} 1 & \shortminus 1 & \frac{1}{2 p_1} \mathbf{1} & \shortminus \frac{1}{2 p_1} \mathbf{1} & \shortminus \frac{1}{2 p_2} \mathbf{1} & \shortminus \frac{1}{2 p_2} \mathbf{1} & \shortminus \frac{1}{4 p_1 p_2} \mathbf{1} & \frac{1}{4 p_1 p_2} \mathbf{1} \end{bmatrix},~~ \overline{b} = 0, \\
		\overline{\CE} = \begin{bmatrix} \begin{bmatrix} 1 & 0 & \mathbf{0} & \mathbf{1} & \mathbf{0} & \mathbf{1}  \\
		0 & 1 & \mathbf{0} & \mathbf{0} & \mathbf{0} & \mathbf{0} \\
		\mathbf{0} & \mathbf{0} & 2I_{p_1} & 2I_{p_1} & \mathbf{0} & \mathbf{0} \\
		\mathbf{0} & \mathbf{0} & \mathbf{0} & \mathbf{0} & 2I_{p_2} & 2I_{p_2} \end{bmatrix} & \begin{bmatrix} \mathbf{0} \\ \mathbf{0} \\ H \end{bmatrix} & \begin{bmatrix} \mathbf{1} \\ \mathbf{0} \\ H \end{bmatrix} \end{bmatrix}, ~~ H = \begin{bmatrix} \begin{bmatrix} 2 & \dots & 2 \end{bmatrix} & & \mathbf{0} \\ & \ddots & \\ \mathbf{0} & & \begin{bmatrix} 2 & \dots & 2 \end{bmatrix} \\
		2 I_{p_2} & \dots & 2I_{p_2} \end{bmatrix},
	\end{gather*}
	which has complexity $\mathcal{O}\left( (\mu_1 + \mu_2)\mu_1 \mu_2 \log(\mu_1 \mu_2) \right)$ with respect to the representation sizes $\mu_1$ and $\mu_2$ and $\mathcal{O}(n^3 \log(n))$ with respect to the dimension $n$. The \operator{compactCon} operation is applied to obtain a regular CPZ.
	\label{theo:union}
\end{theorem}

\begin{proof}	
The proof is provided in Appendix~\ref{app:proofUnion}.

\Comp We first consider the assembly of the resulting CPZ. The computation of the vectors $0.5(c_1 + c_2)$ and $0.5(c_1 - c_2)$ requires $n$ additions, $n$ subtractions, and $2n$ multiplications. Moreover, computation of $\shortminus 0.5\, b_1$, $0.5\, b_1$ and $0.5 \, b_2$ requires $2m_1 + m_2$ multiplications. Computation of the matrix $\overline{A}$ requires $3$ multiplications and $2$ divisions. Since the construction of the remaining matrices only involves concatenations, the resulting complexity for the construction of the CPZ is 
\begin{equation}
	\mathcal{O}(4n + 2m_1 + m_2 + 5) = \mathcal{O}(\underbrace{n + m_1 + m_2}_{\overset{\eqref{eq:repSize}}{\leq} \mu_1 + \mu_2}) = \mathcal{O}(\mu_1 + \mu_2). 
	\label{eq:compUnion1}
\end{equation} 
Next, we consider the subsequent application of the \operator{compactCon} operation. The constraint generator matrix $A$ for the resulting CPZ has $q = 1 + \widehat{q} + \overline{q} + q_1 + q_2 = 4 + 2p_1 + 2p_2 + 2 p_1 p_2 + q_1 + q_2$ columns since $\widehat{A}$ has one column ($\widehat{q} = 1$), $\overline{A}$ has $\overline{q} = 2 + 2p_1 + 2p_2 + 2p_1 p_2$ columns, $A_1$ has $q_1$ columns, and $A_2$ has $q_2$ columns. Moreover, the matrix $A$ has $m = \widehat{m} + \overline{m} + m_1 + m_2 = 2 + m_1 + m_2$ rows since $\widehat{A}$ has one row ($\widehat{m} = 1$), $\overline{A}$ has one row ($\overline{m} = 1$), $A_1$ has $m_1$ rows, and $A_2$ has $m_2$ rows. The number of factors of the resulting CPZ is $p = p_1 + p_2 + 2$. Since the complexity of the \operator{compactCon} operation is $\mathcal{O}(pq\log(q) + m q)$ according to Prop.~\ref{prop:compactCon}, subsequent application of \operator{compactCon} has complexity
	\begin{equation}
	\begin{split}
		& \mathcal{O}\big((p_1 + p_2 + 2)(4 + 2p_1 + 2p_2 + 2 p_1 p_2 + q_1 + q_2)\log(4 + 2p_1 + 2p_2 + 2 p_1 p_2 + q_1 + q_2) \big) \\
		& ~~~~~~~~~~~~~~~~~~~~~~~~~~~~~~~~~~~~~~~~~~~ + \mathcal{O} \big((2 + m_1 + m_2) (2 + 2p_1 + 2p_2 + 2 p_1 p_2 + q_1 + q_2)\big)   \\
		& ~ \\
		& = \mathcal{O}\big((\underbrace{p_1 + p_2}_{\overset{\eqref{eq:repSize}}{\leq} \mu_1 + \mu_2})(\underbrace{p_1 + p_2 + p_1 p_2 + q_1 + q_2}_{\overset{\eqref{eq:repSize}}{\leq} \mu_1 \mu_2})\log(\underbrace{p_1 + p_2 + p_1 p_2 + q_1 + q_2}_{\overset{\eqref{eq:repSize}}{\leq} \mu_1 \mu_2})\big) \\
		& ~~~~~~~~~~~~~~~~~ + \mathcal{O}\big( (\underbrace{m_1 + m_2}_{\overset{\eqref{eq:repSize}}{\leq} \mu_1 + \mu_2}) (\underbrace{p_1 + p_2 + p_1 p_2 + q_1 + q_2}_{\overset{\eqref{eq:repSize}}{\leq} \mu_1 \mu_2})\big) = \mathcal{O}\big( (\mu_1 + \mu_2)\mu_1 \mu_2 \log(\mu_1 \mu_2) \big).
	\end{split} 
	\label{eq:compUnion2}
	\end{equation} 
	Combining \eqref{eq:compUnion1} and \eqref{eq:compUnion2} yields
	\begin{equation*}
		\mathcal{O}(\mu_1 + \mu_2) + \mathcal{O}\big( (\mu_1 + \mu_2)\mu_1 \mu_2 \log(\mu_1 \mu_2) \big) = \mathcal{O}\big( (\mu_1 + \mu_2)\mu_1 \mu_2 \log(\mu_1 \mu_2) \big)
	\end{equation*}
	for the overall complexity with respect to the representation sizes $\mu_1$ and $\mu_2$. Using \eqref{eq:complexity} it furthermore holds that the overall complexity resulting from the combination of \eqref{eq:compUnion1} and \eqref{eq:compUnion2} is identical to $\mathcal{O}(n^3 \log(n))$. \hfill $\square$
\end{proof}

\begin{table*}
\begin{center}
\caption{Growth of the number of generators $h$, the number of factors $p$, the number of constraints $m$, and the number of constraint generators $q$ for basic set operations on $n$-dimensional CPZs.}
\label{tab:repSize}
\begin{tabular}{ p{3cm} C{2.15cm} C{2.15cm} C{2.15cm} C{2.15cm}}
 \toprule
 \textbf{Set Operation} & \textbf{Generators} & \textbf{Factors} & \textbf{Constraints} & \textbf{Constraint Generators} 
 \\ \midrule
 Linear map  & \vspace{4pt} $h$ & \vspace{4pt} $p$ & \vspace{4pt} $m$ & \vspace{4pt} $q$ \\[9pt]
 Minkowski sum & $h_1 + h_2$ & $p_1 + p_2$ & $m_1 + m_2$ & $q_1 + q_2$ \\[9pt]
 Cartesian product & $h_1 + h_2$ & $p_1 + p_2$ & $m_1 + m_2$ & $q_1 + q_2$ \\[9pt]
 Linear combination & $2h_1 + 2h_2 + 1$ & $p_1 + p_2 + 1$ & $m_1 + m_2$ & $q_1 + q_2$ \\[5pt]
 Convex hull & $(n+1)(4h_1 + 4h_2 + 3)$ & $(n+1)(p_1 + p_2 + 2)$ & $(n+1)(m_1+m_2)+1$ & $(n+1)(q_1 + q_2 + 1)$ \\[9pt]
 Quadratic map & $h^2 + h$ & $p$ & $m$ & $q$ \\[9pt]
 Intersection & $h_1$ & $p_1 + p_2$ & $m_1 + m_2 + n$ & $q_1 + q_2 + h_1 + h_2$ \\[5pt]
 Union & $h_1 + h_2 + 1$ & $p_1+p_2+2$ & $m_1 + m_2 + 2$ & $q_1 + q_2 + 2(p_1 + p_2 + p_1 p_2)+4$ \\
 \bottomrule 
\end{tabular}
\end{center}
\end{table*}


\section{Representation Size Reduction}
\label{sec:complexityReduction}

As shown in Table~\ref{tab:repSize}, many operations on CPZs significantly increase the number of factors, generators, constraints, and constraint generators, and consequently also the representation size. For computational reasons, an efficient strategy for representation size reduction is therefore crucial when computing with CPZs. Thus, we now introduce the operations \operator{reduce} and \operator{reduceCon} for reducing the number of generators and the number of constraints of a CPZ. For both operations the tightness of the result can be improved by applying rescaling as described in Sec.~\ref{sec:rescaling} in advance. 

\subsection{Order Reduction}

Our method for reducing the number of generators is inspired by order reduction for constrained zonotopes \cite[Sec.~4.3]{Scott2016} and applies order reduction for polynomial zonotopes: 


\begin{proposition}
	(Order Reduction) Given $\mathcal{CPZ} = \langle c,G,E,A,b,\CE \rangle_{CPZ} \subset \Rn$ and a desired order $\rho_d \geq 2 \frac{n+m}{n}$, the operation \operator{reduce} returns a CPZ with an order smaller than or equal to $\rho_d$ that encloses $\mathcal{CPZ}$:
	\begin{equation*}
	\begin{split}
		\mathcal{CPZ} \subseteq \operator{reduce}(\mathcal{CPZ},\rho_d) = \big \langle & \overline{c},\overline{G}_{(\cdot,\mathcal{H})}, \overline{E}_{(\cdot,\mathcal{H})}, \overline{A}_{(\cdot,\mathcal{K})},\shortminus \overline{b}, \overline{E}_{(\cdot,\mathcal{K})} \big \rangle_{CPZ},
	\end{split}
	\end{equation*}
	with
	\begin{equation*}
	\setlength{\jot}{12pt} 
	\begin{split}
& \mathcal{PZ}^+ = \bigg\langle \begin{bmatrix} c \\ \shortminus b \end{bmatrix},\begin{bmatrix} G & \mathbf{0} \\ \mathbf{0} & A \end{bmatrix},\begin{bmatrix} E & \CE \end{bmatrix} \bigg\rangle_{PZ}, ~~ \rho_d^+ = \frac{\rho_d\, n}{2(n+m)}, \\
		&~~~ \bigg\langle \begin{bmatrix} \overline{c} \\ \overline{b} \end{bmatrix},\begin{bmatrix} \overline{G} \\ \overline{A} \end{bmatrix},\overline{E} \bigg\rangle_{PZ} = \operator{reduce}\big( \operator{compact}(\mathcal{PZ}^+), \rho_d^+ \big),
	\end{split}
	\end{equation*}	
	where the sets $\mathcal{H}$ and $\mathcal{K}$ defined as 
	\begin{equation}
		\mathcal{H} = \big\{i~\big|~ \exists j \in \{1,\dots,n\},~\overline{G}_{(j,i)} \neq 0 \big\}, ~~  \mathcal{K} = \big \{i~\big |~ \exists j \in \{1,\dots,m \},~\overline{A}_{(j,i)} \neq 0 \big \}
		\label{eq:reduceCPZ0}
	\end{equation}	
	store the indices of non-zero generators. The \operator{compact} operation as defined in \cite[Prop.~2]{Kochdumper2019} returns a regular polynomial zonotope and the \operator{reduce} operation for polynomial zonotopes as defined in \cite[Prop.~16]{Kochdumper2019} reduces the order to $\rho_d^+$.  The resulting CPZ is regular and the complexity is $\Oc(\mu^2) + \Oc(\operator{reduce})$ with respect to the representation size $\mu$ and $\Oc(n^2) + \Oc(\operator{reduce})$ with respect to the dimension $n$, where $\Oc(\operator{reduce})$ is the complexity of order reduction for zonotopes.
	\label{prop:reduceCPZ}
\end{proposition}
\begin{proof}
	To calculate a reduced-order CPZ, we reduce the order of the corresponding lifted polynomial zonotope as defined in Lemma~\ref{lemma:liftCPZ} using the \operator{reduce} operation for polynomial zonotopes in \cite[Prop.~16]{Kochdumper2019}. Back-transformation of the lifted polynomial zonotope to the original state space yields an over-approximative CPZ, which can be proven using Lemma~\ref{lemma:liftCPZ}:
	\begin{equation*}
	\setlength{\jot}{8pt}
	\begin{split}
		\forall x \in \Rn,~~  &(x \in \mathcal{CPZ}) \overset{\substack{\text{Lemma}~\ref{lemma:liftCPZ}\\ \vspace{-2pt} }}{\Rightarrow} \bigg( \begin{bmatrix} x \\ \mathbf{0}  \end{bmatrix} \in \mathcal{PZ}^+ \bigg) \overset{\substack{\mathcal{PZ}^+ \subseteq \, \operator{reduce}(\mathcal{PZ}^+,\rho_d^+)\\ \vspace{-2pt}}}{\Rightarrow} \\
		 & \bigg( \begin{bmatrix} x \\ \mathbf{0}  \end{bmatrix} \in \operator{reduce}(\mathcal{PZ}^+,\rho_d^+) \bigg) \overset{\substack{\text{Lemma}~\ref{lemma:liftCPZ}\\ \vspace{-2pt}}}{\Rightarrow} \big( x \in \operator{reduce}(\mathcal{CPZ},\rho_d) \big),
	\end{split}
	\end{equation*}
	where we omitted the \operator{compact} operation since it only changes the representation of the set, but not the set itself. It remains to show that the order of the resulting CPZ is smaller than or equal to the desired order $\rho_d$. According to \cite[Prop.~16]{Kochdumper2019}, we have 
	\begin{equation}
		\frac{\overline{h}}{n^+} = \frac{\overline{h}}{n + m} \leq \rho_d^+ = \frac{\rho_d \, n}{2(n+m)},
		\label{eq:reduceCPZ1}
	\end{equation}
	where $\overline{h}$ denotes the number of columns of the matrix $\overline{G}$ and $n^+ = n + m$ is the dimension of the lifted polynomial zonotope $\mathcal{PZ}^+$. Solving \eqref{eq:reduceCPZ1} for $\rho_d$ yields $2 \, \overline{h} / n \leq \rho_d$ so that
	\begin{equation*}
		\rho = \frac{|\mathcal{H}| + |\mathcal{K}|}{n} \overset{\substack{\eqref{eq:reduceCPZ0} \\ \vspace{-2pt}}}{\leq} 2 \, \frac{\overline{h}}{n} \leq \rho_d
	\end{equation*}
	holds since the number of elements in the sets $\mathcal{H}$ and $\mathcal{K}$ is at most $\overline{h}$ according to \eqref{eq:reduceCPZ0}.
	
	\Comp Let $n^+ = n+m$, $p^+ = p$, and $h^+ = h + q$ denote the dimension, the number of factors, and the number of generators of the lifted polynomial zonotope $\mathcal{PZ}^+$. According to \cite[Prop.~2]{Kochdumper2019}, the \operator{compact} operation has complexity $\Oc(p^+ h^+ \log(h^+)) = \Oc(p+(h+q)\log(h+q))$ and the complexity of order reduction of a polynomial zonotope using \cite[Prop.~16]{Kochdumper2019} is $\Oc(h^+ (n^+ + p^+ + \log(h^+))) + \Oc(\operator{reduce}) = \Oc((h+q)(n+m+p+\log(h+q))) + \Oc(\operator{reduce})$, where $\Oc(\operator{reduce})$ denotes the complexity of order reduction for zonotopes, which depends on the method that is used. Moreover, construction of the sets $\mathcal{H}$ and $\mathcal{K}$ has complexity $\Oc((h+q)(n+m))$ in the worst case. The overall computational complexity is therefore 
	\begin{equation*}
	\setlength{\jot}{8pt}
	\begin{split}
		& \Oc\big(\underbrace{p+(h+q)\log(h+q)}_{\overset{\eqref{eq:repSize}}{\leq} \mu \log(\mu)}\big) + \Oc\big(\underbrace{(h+q)(n+m+p+\log(h+q))}_{\overset{\eqref{eq:repSize}}{\leq} \mu^2 + \mu \log(\mu)}\big) + \Oc(\operator{reduce})  \\
		& = \Oc\big(\mu \log(\mu)\big) + \Oc\big(\mu^2 + \mu \log(\mu)\big) + \Oc(\operator{reduce}) = \Oc(\mu^2) + \Oc(\operator{reduce}),
	\end{split}
	\end{equation*}
	which is $\Oc(n^2) + \Oc(\operator{reduce})$ using \eqref{eq:complexity}.
\end{proof}
Let us demonstrate the tightness of our order reduction method for CPZs by an example:

\begin{figure}
  \centering
  \psfragfig[width=0.95\columnwidth]{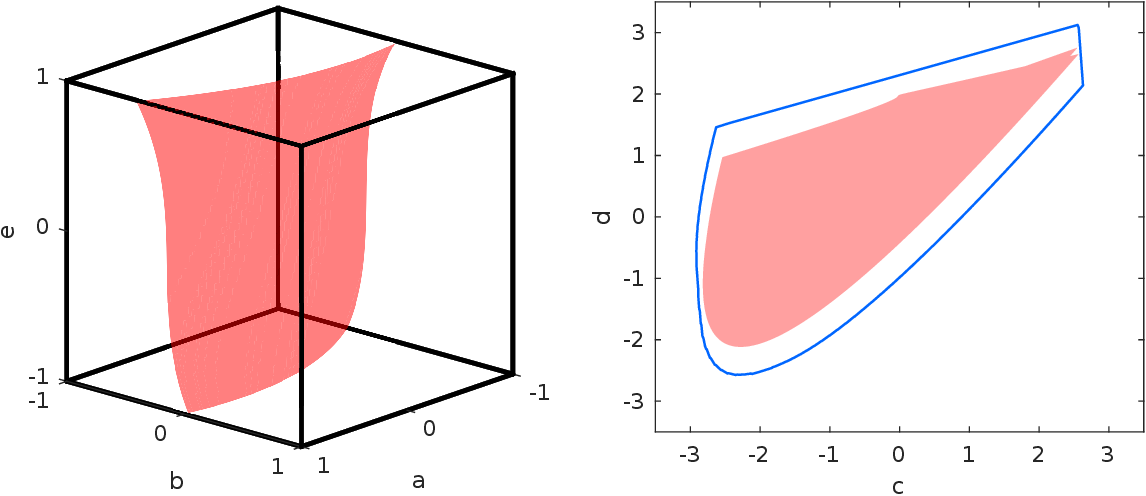}{
  \psfrag{a}[c][c]{$\alpha_1$}
  \psfrag{b}[c][c]{$\alpha_2$}
  \psfrag{c}[c][c]{$x_1$}
  \psfrag{d}[c][c]{\rotatebox[origin=c]{180}{$x_2$}}
  \psfrag{e}[c][c]{\rotatebox[origin=c]{180}{$\alpha_3$}}
  }
  \caption{Visualization of order reduction using Prop.~\ref{prop:reduceCPZ} for the CPZ from Example~\ref{ex:reduceCPZ}. The original set $\mathcal{CPZ}$ (right) and the corresponding constraint (left) are depicted in red, and the reduced order CPZ $\operator{reduce}(\mathcal{CPZ},6)$ (right) is depicted in blue.}
  \label{fig:reduceCPZ}
\end{figure}

\begin{example}
	We consider the CPZ
	\begin{equation*}
	\begin{split}
		\mathcal{CPZ} = \bigg \langle & \begin{bmatrix} \shortminus 2 \\ \shortminus 2 \end{bmatrix}, \begin{bmatrix} 2.5 & 0 & ~2~ & 0.05 & 0.02 & \shortminus 0.03 & 0 \\ 0 & \shortminus 4 & 3 & 0.02 & \shortminus 0.01 & 0 & 0.02 \end{bmatrix}, \begin{bmatrix} 1 & 0 & 2 & 1 & 1 & 1 & 0 \\ 0 & 1 & 0 & 0 & 0 & 0 & 0 \\ 3 & 0 & 0 & 1 & 0 & 2 & 3 \end{bmatrix}, \\
		& ~~~~~~~~~~~~~~~~~~~ \begin{bmatrix} 1 & 5 & 1 & 0.1 & \shortminus 0.2 & 0.2 & 0.05 \end{bmatrix}, 0, \begin{bmatrix} 1 & 0 & 2 & 1 & 1 & 1 & 0 \\ 0 & 1 & 0 & 0 & 0 & 0 & 0 \\ 3 & 0 & 0 & 1 & 0 & 2 & 3 \end{bmatrix} \bigg \rangle_{CPZ},
	\end{split}
	\end{equation*}
	which has order $\rho = 7$. The resulting CPZ after order reduction to the desired order $\rho_d = 6$ using Prop.~\ref{prop:reduceCPZ} is visualized in Fig.~\ref{fig:reduceCPZ}, where we used principal component analysis for order reduction of zonotopes \cite[Sec.~III.A]{Kopetzki2017}, which is required for order reduction of polynomial zonotopes using \cite[Prop.~16]{Kochdumper2019}.
	\label{ex:reduceCPZ}
\end{example}

\subsection{Constraint Reduction}

Next, we present an approach for reducing the number of constraints of a CPZ, which is inspired by constraint reduction for constrained zonotopes \cite[Sec.~4.2]{Scott2016}:

\begin{proposition}
	(Constraint Reduction) Given $\mathcal{CPZ} = \langle c,G,E,A,b,\CE \rangle_{CPZ}\linebreak[3] \subset \R^n$, the index of one constraint $r \in \{1,\dots,m \}$, and indices $d,s \in \mathbb{N}$ satisfying
	\begin{equation}
		\forall i \in \{1,\dots,p\},~ E_{(i,d)} = \CE_{(i,s)} ~~ \text{and} ~~ A_{(r,s)} \neq 0 ,
		\label{eq:reduceConCondCPZ}
	\end{equation}
	the operation \operator{reduceCon} removes the constraint with index $r$ and returns a CPZ that encloses $\mathcal{CPZ}$:
	\begin{equation*}
		\mathcal{CPZ} \subseteq \operator{reduceCon}(\mathcal{CPZ},r,d,s) = \big \langle \overline{c},\overline{G},\overline{E}_{(\mathcal{N},\cdot)},\overline{A},\overline{b},\overline{\CE}_{(\mathcal{N},\cdot)} \big \rangle_{CPZ},
	\end{equation*}
	where
	\begin{equation*}
	\setlength{\jot}{6pt}
		\begin{split}
		 & ~~ \overline{G} = \begin{bmatrix} G_{(\cdot,\{1,\dots,d-1\})} & \shortminus \frac{1}{A_{(r,s)}} A_{(r,\mathcal{H})} G_{(\cdot,d)} & G_{(\cdot,\{d+1,\dots,h\})} \end{bmatrix}, ~~ \overline{\CE} = \CE_{(\cdot,\mathcal{H})}, \\
		 & \overline{E} = \begin{bmatrix} E_{(\cdot,\{1,\dots,d-1\})} & \overline{\CE} & E_{(\cdot,\{d+1,\dots,h\})} \end{bmatrix}, ~~ \overline{A} = A_{(\mathcal{K},\mathcal{H})}- \frac{1}{A_{(r,s)}} A_{(r,\mathcal{H})} A_{(\mathcal{K},s)}, \\
		 & ~~~~~~~~~~~~~~ \overline{c} = c + \frac{b_{(r)}}{A_{(r,s)}} G_{(\cdot,d)}, ~~ \overline{b} = b_{(\mathcal{K})} - \frac{b_{(r)}}{A_{(r,s)}} A_{(\mathcal{K},s)},
		\end{split}
	\end{equation*}
	and the sets $\mathcal{H}$, $\mathcal{K}$, and $\mathcal{N}$ are defined as
	\begin{equation}
	\setlength{\jot}{6pt}
	\begin{split}
		& \mathcal{H} = \{1,\dots,q \} \setminus s, ~~ \mathcal{K} = \{1,\dots,m \} \setminus r, \\
		& \mathcal{N} = \big \{ i ~ \big| ~ \exists j,k,~ \overline{E}_{(i,j)} \neq 0  \vee \overline{\CE}_{(i,k)} \neq 0 \big \}.
	\end{split}
		\label{eq:conRedCPZ0}
	\end{equation}
	The \operator{compactGen} operation is applied to make the resulting CPZ regular and the complexity is $\Oc(\mu^2)$ with respect to the representation size $\mu$ and $\Oc(n^2 \log(n))$ with respect to the dimension $n$. 
	\label{prop:reduceConCPZ}
\end{proposition}

\begin{proof}
	To remove the constraint with index $r$, we solve the corresponding equation for the term that is multiplied with the constraint generator with index $s$:
	\begin{equation}
		\prod_{k=1}^p \alpha_k^{\CE_{(k,s)}} = \frac{1}{A_{(r,s)}} \bigg( - \sum_{i \in \mathcal{H}} \bigg( \prod_{k=1}^p \alpha_k^{\CE_{(k,i)}} \bigg) A_{(r,i)} + b_{(r)} \bigg).
		\label{eq:conRedCPZ1}
	\end{equation}
	For the constraint with index $r$, we obtain with the substitution from \eqref{eq:conRedCPZ1}
	\begin{equation}
		\sum_{i \in \mathcal{H}} \bigg( \prod_{k=1}^p \alpha_k^{\CE_{(k,i)}} \bigg) \underbrace{\bigg( A_{(r,i)} - \frac{1}{A_{(r,s)}} A_{(r,i)} A_{(r,s)} \bigg)}_{= 0} = \underbrace{b_{(r)} - \frac{b_{(r)}}{A_{(r,s)}} A_{(r,s)}}_{=0}
		\label{eq:conRedCPZ5}
	\end{equation}	
	the trivial constraint $0 = 0$, which can be removed by restricting the indices of the constraints to the set $\mathcal{K}$ as defined in \eqref{eq:conRedCPZ0}. Finally, inserting the substitution in \eqref{eq:conRedCPZ1} into the definition of a CPZ in Def.~\ref{def:CPZ} yields
	\begin{align*}
		& \mathcal{CPZ} = \bigg \{ c + \sum_{i=1}^{h} \bigg( \prod_{k=1}^p \alpha_k^{E_{(k,i)}} \bigg) G_{(\cdot,i)} ~ \bigg | ~ \sum_{i=1}^{q} \bigg( \prod_{k=1}^p \alpha_k^{\CE_{(k,i)}} \bigg) A_{(\cdot,i)} = b, ~ \alpha_k \in [\shortminus 1,1] \bigg \} \nonumber \\
		& ~ \nonumber \\
		& \overset{\eqref{eq:conRedCPZ0}}{=} \bigg\{ c + \sum_{i=1}^{d-1} \bigg( \prod_{k=1}^p \alpha_k^{E_{(k,i)}} \bigg) G_{(\cdot,i)} + \bigg(\prod_{k=1}^p \alpha_k^{\CE_{(k,s)}} \bigg) G_{(\cdot,d)} + \sum_{i=d+1}^{h} \bigg( \prod_{k=1}^p \alpha_k^{E_{(k,i)}} \bigg) G_{(\cdot,i)} ~ \bigg | \nonumber \\
		 & ~~~~~~~~~~~~~~~~~~~~~~~~~~~~ \sum_{i \in \mathcal{H}} \bigg( \prod_{k=1}^p \alpha_k^{\CE_{(k,i)}} \bigg) A_{(\cdot,i)} + \bigg(\prod_{k=1}^p \alpha_k^{\CE_{(k,s)}} \bigg) A_{(\cdot,s)} = b, ~ \alpha_k \in [\shortminus 1,1] \bigg\} \nonumber \\
		 & ~ \nonumber \\
		 & \overset{\eqref{eq:conRedCPZ1}}{\subseteq} \bigg\{ c + \sum_{i=1}^{d-1} \bigg( \prod_{k=1}^p \alpha_k^{E_{(k,i)}} \bigg) G_{(\cdot,i)} + \frac{1}{A_{(r,s)}} \bigg( - \sum_{i \in \mathcal{H}} \bigg( \prod_{k=1}^p \alpha_k^{\CE_{(k,i)}} \bigg) A_{(r,i)} + b_{(r)} \bigg) G_{(\cdot,d)}  \nonumber \\
		 & ~~~~~~~~ + \sum_{i=d+1}^{h} \bigg( \prod_{k=1}^p \alpha_k^{E_{(k,i)}} \bigg) G_{(\cdot,i)} ~ \bigg | ~ \sum_{i \in \mathcal{H}} \bigg( \prod_{k=1}^p \alpha_k^{\CE_{(k,i)}} \bigg) A_{(\cdot,i)} + \nonumber \\
		 & ~~~~~~~~~~~~~~~~~~~~~~~~~~ \frac{1}{A_{(r,s)}} \bigg( - \sum_{i \in \mathcal{H}} \bigg( \prod_{k=1}^p \alpha_k^{\CE_{(k,i)}} \bigg) A_{(r,i)} + b_{(r)} \bigg) A_{(\cdot,s)} = b, ~ \alpha_k \in [\shortminus 1,1] \bigg\} \nonumber \\
		 & ~ \nonumber \\
		 & \overset{\eqref{eq:conRedCPZ5}}{=} \bigg\{ \underbrace{c + \frac{b_{(r)}}{A_{(r,s)}} G_{(\cdot,d)}}_{\overline{c}} + \sum_{i=1}^{d-1} \bigg( \prod_{k=1}^p \alpha_k^{E_{(k,i)}} \bigg) G_{(\cdot,i)}  - \sum_{i \in \mathcal{H}} \bigg( \prod_{k=1}^p \alpha_k^{\CE_{(k,i)}} \bigg) \frac{1}{A_{(r,s)}}A_{(r,i)} G_{(\cdot,d)} \nonumber \\
		 & ~~~~~~~~~ + \sum_{i=d+1}^{h} \bigg( \prod_{k=1}^p \alpha_k^{E_{(k,i)}} \bigg) G_{(\cdot,i)} ~ \bigg | ~\alpha_k \in [\shortminus 1,1], \nonumber \\ 
		 & ~~~~~~~~~~~~~~~~~~~ \sum_{i \in \mathcal{H}} \bigg( \prod_{k=1}^p \alpha_k^{\CE_{(k,i)}} \bigg) \underbrace{ \bigg( A_{(\mathcal{K},i)} - \frac{1}{A_{(r,s)}} A_{(r,i)} A_{(\mathcal{K},s)} \bigg)}_{\overline{A}_{(\cdot,i)}} = \underbrace{b_{(\mathcal{K})} - \frac{b_{(r)}}{A_{(r,s)}} A_{(\mathcal{K},s)}}_{\overline{b}} \bigg \} \nonumber \\
		 & ~ \nonumber \\
		 & ~=  \big \langle \overline{c},\overline{G},\overline{E}_{(\mathcal{N},\cdot)},\overline{A},\overline{b},\overline{\CE}_{(\mathcal{N},\cdot)} \big \rangle_{CPZ} = \operator{reduceCon}(\mathcal{CPZ},r,d,s). \nonumber
	\end{align*}
	The set $\mathcal{N}$ as defined in \eqref{eq:conRedCPZ0} only removes all-zero rows from the exponent matrix and the constraint exponent matrix, and therefore does not change the set.
	
	\Comp The construction of the matrix $\overline{G}$ requires $nq$ multiplications, the construction of the matrix $\overline{A}$ requires $(m-1)q$ multiplications and $(m-1)(q-1)$ subtractions, and the construction of the vectors $\overline{c}$ and $\overline{b}$ requires $n+m-1$ multiplications, $n$ additions, and $m-1$ subtractions. Let $\overline{p} = p$, $\overline{h} = h+q - 2$, and $\overline{q} = q - 1$ denote the number of factors, the number of generators, and the number of constraint generators of the resulting CPZ. Subsequent application of the \operator{compactGen} operation has complexity $\Oc(\overline{p} \overline{h} \log(\overline{h}) + n \overline{h}) = \Oc(p (h+q-2) \log(h+q-2) + n(h+q-2))$ according to Prop.~\ref{prop:compactGen}, and the construction of the set $\mathcal{N}$ in \eqref{eq:conRedCPZ0} has in the worst case complexity $\Oc(\overline{p}(\overline{h}+\overline{q})) = \Oc(p(h+2q-2))$. The overall complexity is therefore 
	\begin{equation*}
		\begin{split}
		& \Oc(nq) + \Oc\big((m-1)(2q-1)\big) + \Oc(2n+2m-2) \\
		& ~~~~~~~~~~~~~ + \Oc\big((h+ q - 2)(n+p \log(h + q - 2))\big) + \Oc\big(p(h+2q-2)\big)\\
		& ~ \\
		& ~~~~~ = \Oc(\underbrace{mq}_{\overset{\eqref{eq:repSize}}{\leq}\mu}) + \Oc\big(\underbrace{(h+q)(n+p \log(h+q))}_{\overset{\eqref{eq:repSize}}{\leq}\mu^2 + \mu \log(\mu)}\big) = \Oc(\mu) + \Oc(\mu^2 + \mu \log(\mu)) = \Oc(\mu^2),
		\end{split}
	\end{equation*}
	which is $\Oc(n^2 \log(n))$ using \eqref{eq:complexity}. \hfill $\square$ 
\end{proof}
The crucial point for Prop.~\ref{prop:reduceConCPZ} is the selection of the constraint with index $r$ that is removed, as well as the selection of suitable indices $s,d$ that satisfy the conditions in \eqref{eq:reduceConCondCPZ} and can therefore be used for reduction. Clearly, we want to select the indices $r$, $s$, and $d$ such that the over-approximation resulting from constraint reduction is minimized. Since it is computationally infeasible to determine the optimal indices for reduction, we instead present some heuristics on how to choose good values for $r$, $s$, and $d$. When removing a constraint from a CPZ using Prop.~\ref{prop:reduceConCPZ}, there are two sources contributing to the resulting over-approximation:
\begin{enumerate} 
 \item Over-approximation due to lost bounds on factors.
 \item Over-approximation due to a loss of dependency.
\end{enumerate}
We now explain both sources in detail and provide illustrative examples.

\myparagraph{Over-Approximation due to Lost Bounds}

\noindent The over-approximation due to lost bounds results from the fact that due to the replacement of the term $\prod_{k=1}^p \alpha_k^{\substack{R_{(k,s)} \\ \vspace{-5pt}}}$ with the solved constraint in \eqref{eq:conRedCPZ1}, we lose the ability to enforce the bounds $\prod_{k=1}^p \alpha_k^{\substack{R_{(k,s)} \\ \vspace{-5pt}}} \in \prod_{k=1}^p [\shortminus 1,1]^{R_{(k,s)}}$, where $\prod^p_{k=1} [\shortminus 1, 1]^{R_{(k,s)}}$ denotes interval multiplication and exponentiation. Consequently, constraint reduction results in an over-approximation if the solved constraint in \eqref{eq:conRedCPZ1} has feasible values outside the domain $\prod_{k=1}^p [\shortminus 1,1]^{R_{(k,s)}}$, which is equivalent to the condition
\begin{equation}
	\bigg \{ \frac{1}{A_{(r,s)}} \bigg( - \sum_{i \in \mathcal{H}} \bigg( \prod_{k=1}^p \alpha_k^{\CE_{(k,i)}} \bigg) A_{(r,i)} + b_{(r)} \bigg)~\bigg | ~ \alpha_k \in [\shortminus 1,1] \bigg \} \nsubseteq \prod_{k=1}^p [\shortminus 1,1]^{R_{(k,s)}}.
	\label{eq:solvedConstraint}
\end{equation} 
Let us demonstrate this by an example:

\begin{figure}
  \centering
  \psfragfig[width=0.95\columnwidth]{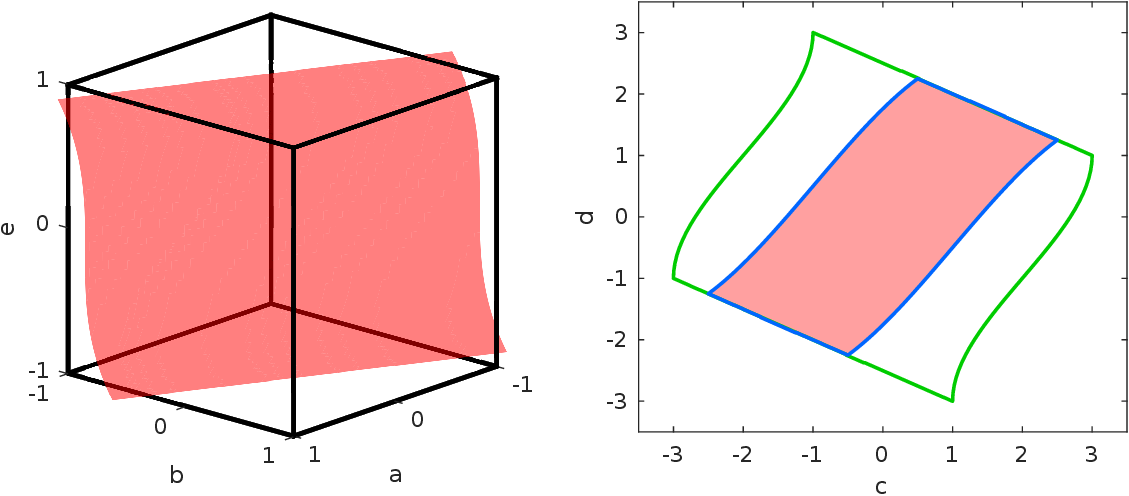}{
  \psfrag{a}[c][c]{$\alpha_1$}
  \psfrag{b}[c][c]{$\alpha_2$}
  \psfrag{c}[c][c]{$x_1$}
  \psfrag{d}[c][c]{\rotatebox[origin=c]{180}{$x_2$}}
  \psfrag{e}[c][c]{\rotatebox[origin=c]{180}{$\alpha_3$}}
  }
  \caption[Example for constraint reduction of constrained polynomial zonotopes]{Visualization of constraint reduction using Prop.~\ref{prop:reduceConCPZ} for $\mathcal{CPZ}$ from Example~\ref{ex:reduceConCPZ1} (red,~right), where the corresponding constraint is visualized on the left. While $\operator{reduceCon}(\mathcal{CPZ},1,1,1)$ (green) results in an over-approximation, $\operator{reduceCon}(\mathcal{CPZ},1,2,2)$ (blue) is exact.}
  \label{fig:reduceConCPZ1}
\end{figure}

\begin{example}
	We consider the CPZ
	\begin{equation*}
	\setlength{\jot}{12pt}
	\begin{split}
		\mathcal{CPZ} &= \bigg \langle \begin{bmatrix} 0 \\ 0 \end{bmatrix}, \begin{bmatrix} 1 & ~0 & 1.5 \\ 0 & ~1 & 2 \end{bmatrix}, \begin{bmatrix} 1 & 0 & 0 \\ 0 & 1 & 0 \\ 0 & 0 & 1 \end{bmatrix}, \begin{bmatrix} 1 & 2 & 0.5 \end{bmatrix}, 0, \begin{bmatrix} 1 & 0 & 0 \\ 0 & 1 & 0 \\ 0 & 0 & 3 \end{bmatrix} \bigg \rangle_{CPZ} \\
		& = \bigg\{ \begin{bmatrix} 1 \\ 0 \end{bmatrix} \alpha_1 + \begin{bmatrix} 0 \\ 1 \end{bmatrix} \alpha_2 + \begin{bmatrix} 1.5 \\ 2 \end{bmatrix} \alpha_3~ \bigg|~ \alpha_1 + 2 \alpha_2 + 0.5 \alpha_3^3 = 0,~ \alpha_1,\alpha_2,\alpha_3 \in [\shortminus 1,1] \bigg\},
	\end{split}
	\end{equation*}
	which is visualized in Fig.~\ref{fig:reduceConCPZ1}. We first choose the indices $s=1$ and $d=1$ that correspond to the term $\alpha_1$ for constraint reduction. In this case, solving the constraint $\alpha_1 + 2 \alpha_2 + 0.5 \alpha_3^3 = 0$ for $\alpha_1$ yields $\alpha_1  = -2 \alpha_2 - 0.5 \alpha_3^3$. As visible in Fig.~\ref{fig:reduceConCPZ1} (left), the solved constraint has feasible values outside the domain $\alpha_1 \in [\shortminus 1,1]$:
	\begin{equation*}
		\{-2 \alpha_2 - 0.5 \alpha_3^3 ~|~\alpha_2,\alpha_3 \in [\shortminus 1,1] \} = [\shortminus 2.5,2.5] \nsubseteq [\shortminus 1,1].
	\end{equation*}
	Constraint reduction using Prop.~\ref{prop:reduceConCPZ} with the indices $s=1$ and $d=1$ therefore results in an over-approximation $\operator{reduceCon}(\mathcal{CPZ},1,1,1) \supset \mathcal{CPZ}$ (see Fig.~\ref{fig:reduceConCPZ1} (right)). Next, we consider the indices $s = 2$ and $d = 2$ that correspond to the term $\alpha_2$. Solving the constraint for $\alpha_2$ yields $\alpha_2 = -0.5 \alpha_1 - 0.25\alpha_3^3$, so that the set of feasible values for the solved constraint is a subset of the domain $\alpha_2 \in [\shortminus 1,1]$:
	\begin{equation*}
		\{ -0.5 \alpha_1 - 0.25\alpha_3^3 ~|~ \alpha_1,\alpha_3 \in [\shortminus 1,1] \} = [\shortminus 0.75,0.75] \subset [\shortminus 1,1].
	\end{equation*}
	 Consequently, constraint reduction using Prop.~\ref{prop:reduceConCPZ} with the indices $s = 2$ and $d = 2$ does not result in an over-approximation, so that $\operator{reduceCon}(\mathcal{CPZ},1,2,2) = \mathcal{CPZ}$ (see Fig.~\ref{fig:reduceConCPZ1} (right)).
	\label{ex:reduceConCPZ1}
\end{example}

Computing the exact bounds for the solved constraint in \eqref{eq:solvedConstraint} is in general computationally infeasible. Instead, one can use range bounding to compute over-approximations of the bounds. Another heuristic that we observed to perform well in practice is to first enclose the CPZ with a constrained zonotope using Prop.~\ref{prop:conZonoEncloseCPZ}, and then select the constraint that is removed as well as the indices that are used for reduction based on the constrained zonotope enclosure. For constrained zonotopes, sophisticated methods for selecting the constraints and indices that result in the least over-approximation are available in \cite[Appendix]{Scott2016}. 

\begin{figure}
  \centering
  \psfragfig[width=0.95\columnwidth]{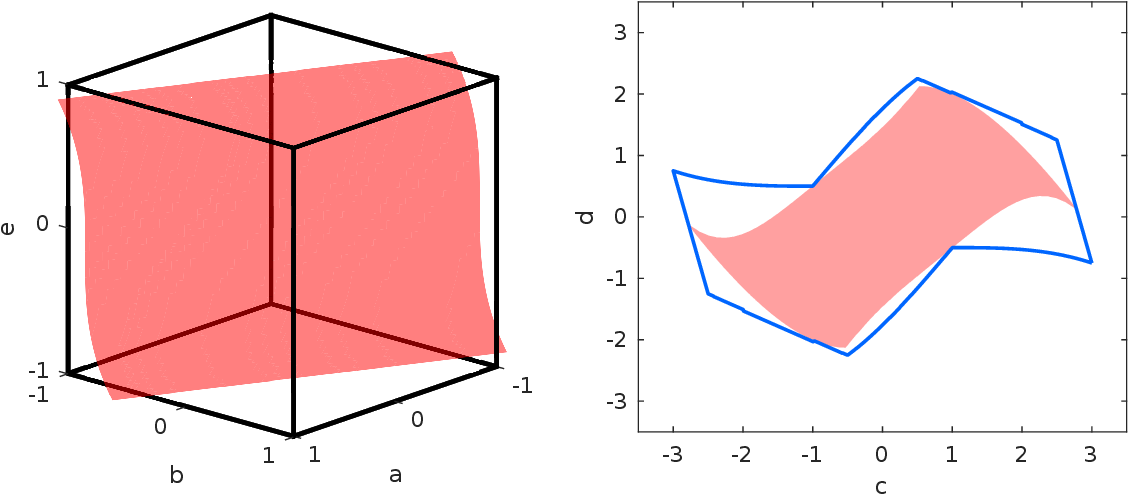}{
  \psfrag{a}[c][c]{$\alpha_1$}
  \psfrag{b}[c][c]{$\alpha_2$}
  \psfrag{c}[c][c]{$x_1$}
  \psfrag{d}[c][c]{\rotatebox[origin=c]{180}{$x_2$}}
  \psfrag{e}[c][c]{\rotatebox[origin=c]{180}{$\alpha_3$}}
  }
  \caption[Example demonstrating the loss of dependency during constraint reduction of constrained polynomial zonotopes]{Visualization of constraint reduction using Prop.~\ref{prop:reduceConCPZ} for $\mathcal{CPZ}$ from Example~\ref{ex:reduceConCPZ2} (red,~right), where the corresponding constraint is visualized on the left. Due to the loss of dependency, constraint reduction $\operator{reduceCon}(\mathcal{CPZ},1,2,2)$ (blue) results in an over-approximation.}
  \label{fig:reduceConCPZ2}
\end{figure}

\myparagraph{Over-Approximation due to Loss of Dependency}

\noindent The second source contributing to the over-approximation during constraint reduction is the loss of dependency. This loss arises since if we substitute the term $\prod_{k=1}^p \alpha_k^{\substack{R_{(k,s)} \\ \vspace{-5pt}}}$ with the solved constraint in \eqref{eq:conRedCPZ1}, the dependency between the factors $\alpha_k$ in $\prod_{k=1}^p \alpha_k^{\substack{R_{(k,s)} \\ \vspace{-4pt}}}$ and the factors $\alpha_k$ in other terms gets lost. Let us demonstrate this by an example:

\begin{example}
	We consider the CPZ
	\begin{equation*}
		~\mathcal{CPZ} = \bigg\{ \hspace{-3pt} \begin{bmatrix} 1 \\ 0 \end{bmatrix} \hspace{-2pt} \alpha_1 + \begin{bmatrix} 0 \\ 1 \end{bmatrix} \hspace{-2pt} \alpha_2 + \begin{bmatrix} 1.5 \\ 2 \end{bmatrix} \hspace{-2pt} \alpha_3 + \begin{bmatrix} 0.5 \\ \shortminus 2 \end{bmatrix} \hspace{-2pt} \alpha_2^2 \alpha_3 \, \bigg|\, \alpha_1 + 2 \alpha_2 + 0.5 \alpha_3^3 = 0,\, \alpha_1,\alpha_2,\alpha_3 \in [\shortminus 1,1] \hspace{-2pt} \bigg\},
	\end{equation*}	
	which is identical to the CPZ from Example~\ref{ex:reduceConCPZ1}, except that we added the term $\alpha_2^2 \alpha_3$. Since the polynomial constraint is identical to the constraint of the CPZ in Example~\ref{ex:reduceConCPZ1}, constraint reduction using the indices $s=2$ and $d=2$ does not result in an over-approximation due to lost bounds, as we demonstrated in Example~\ref{ex:reduceConCPZ1}. However, with the indices $s=2$ and $d=2$, we substitute $\alpha_2$ by the solved constraint $\alpha_2 = -0.5 \alpha_1 - 0.25\alpha_3^3$, so that the dependency between $\alpha_2$ in the term $\alpha_2$ for the second generator and $\alpha_2$ in the additional term $\alpha_2^2 \alpha_3$ gets lost. Due to this loss of dependency, constraint reduction using Prop.~\ref{prop:reduceConCPZ} with indices $s=2$ and $d=2$ results in an over-approximation, as visualized in Fig.~\ref{fig:reduceConCPZ2}.
	\label{ex:reduceConCPZ2} 
\end{example} 

In the above example we could prevent the loss of dependency by substituting $\alpha_2^2$ with $\alpha_2^2 = (-0.5 \alpha_1 - 0.25\alpha_3^3)^2$, which corresponds to the square of the solved constraint $\alpha_2 = -0.5 \alpha_1 - 0.25\alpha_3^3$. Similarly, it is possible to relax the condition in \eqref{eq:reduceConCondCPZ} to 
\begin{equation*}
	\forall i \in \{1,\dots,p\},~ \CE_{(i,s)}^e = E_{(i,d)}~~ \text{and} ~~ A_{(r,s)} \neq 0 ,
\end{equation*}
which allows powers of the selected term with arbitrary exponents $e \in \mathbb{N}$. While this relaxation often enables us to reduce constraints with less over-approximation, taking powers of the solved constraint significantly increases the number of generators of the resulting CPZ, and therefore also the computational complexity for the subsequent \operator{compactGen} operation.


\section{Numerical Example}

For the numerical experiments, we implemented CPZs in the MATLAB toolbox CORA \cite{Althoff2015a}, which is available at \url{https://cora.in.tum.de}. All computations are carried out on a 2.9GHz quad-core i7 processor with 32GB memory.

One often occurring task in set-based computing is to calculate the image of a given initial set under a nonlinear function. We demonstrate by a numerical example that with CPZs the image can be computed exactly if the nonlinear function is polynomial. Let us consider the nonlinear function

\begin{equation}
	\begin{split}
	& f(x) = \begin{cases} \begin{bmatrix} 0.1\,x_{(1)}^2 -1.2\,x_{(1)} x_{(2)} - 0.5\, x_{(2)}^2 \\ \shortminus x_{(1)}^2 + 2\, x_{(2)}^2 \end{bmatrix}, & 0.5 \, x_{(1)}^2 \leq x_{(2)} \\ ~ \\\begin{bmatrix} 1.2 \, x_{(1)} - x_{(2)} \\ \shortminus x_{(1)} + 0.1 \, x_{(2)} \end{bmatrix}, &  \mathrm{otherwise} \end{cases}
	\end{split}
	\label{eq:nonlinFun}
\end{equation}
and the polytope $\mathcal{P}$ with vertices $[\shortminus 1~1]^T$, $[0~\shortminus 1]^T$, and $[1~0]^T$. The task is to compute the image of $\mathcal{P}$ under the nonlinear function $f(x)$. First, we convert the polytope $\mathcal{P}$ to a CPZ according to Sec. \ref{subsec:polytope} based on a conversion to a polynomial zonotope \cite{Kochdumper2021}, which yields
\begin{equation*}
	\mathcal{P} = \bigg \{ \begin{bmatrix} \shortminus 0.25 \\ 0.25 \end{bmatrix} + \begin{bmatrix} \shortminus 0.75 \\ 0.75 \end{bmatrix} \alpha_1 + \begin{bmatrix} \shortminus 0.25 \\ \shortminus 0.25 \end{bmatrix} \alpha_2 + \begin{bmatrix} 0.25 \\ 0.25 \end{bmatrix} \alpha_1 \alpha_2 ~\bigg | ~ \alpha_1,\alpha_2 \in [\shortminus 1,1] \bigg \}.
\end{equation*}
Using the linear map in Prop.~\ref{prop:linearMap}, the quadratic map in Prop.~\ref{prop:quadMap}, the intersection in Prop.~\ref{prop:intersection}, and the union in Theorem~\ref{theo:union}, we can compute the image exactly as
\begin{equation*}
	\big \{ f(x) ~|~ x \in \mathcal{P} \big \} = sq(\mathcal{Q},\mathcal{P} \cap \mathcal{CPZ}_1) \cup \big( M \otimes (\mathcal{P} \cap \mathcal{CPZ}_2)\big),
\end{equation*}
where the parameter $M$ and $\mathcal{Q}$ for the linear and quadratic map are
\begin{equation*}
	\mathcal{Q} = \{Q_1,Q_2\}, ~~ Q_1 = \begin{bmatrix} 0.1 & \shortminus 1.2 \\ 0 & \shortminus 0.5 \end{bmatrix}, ~~ Q_2 = \begin{bmatrix} \shortminus 1 & 0 \\ 0 & 2 \end{bmatrix}, ~~ M = \begin{bmatrix} 1.2 & \shortminus 1 \\ \shortminus 1 & 0.1 \end{bmatrix},
\end{equation*}
and the CPZs
\begin{equation*}
\setlength{\jot}{12pt}
\begin{split}
	& \mathcal{CPZ}_1 = \bigg\{ \begin{bmatrix} 1 \\ 0 \end{bmatrix} \alpha_1 + \begin{bmatrix} 0 \\ 1\end{bmatrix} \alpha_2 ~\bigg |~ 0.5 \, \alpha_1^2 - \alpha_2 + \alpha_3 = \shortminus 1,~ \alpha_1,\alpha_2,\alpha_3 \in [\shortminus 1,1] \bigg\} \\ 
	& \mathcal{CPZ}_2 = \bigg\{ \begin{bmatrix} 1 \\ 0 \end{bmatrix} \alpha_1 + \begin{bmatrix} 0 \\ 1\end{bmatrix} \alpha_2 ~\bigg |~ 0.5 \, \alpha_1^2 - \alpha_2 + \alpha_3 = 1,~ \alpha_1,\alpha_2,\alpha_3 \in [\shortminus 1,1] \bigg\}
\end{split}
\end{equation*}
represent the regions $0.5\, x_{(1)}^2 \leq x_{(2)}$ and $0.5 \, x_{(1)}^2 > x_{(2)}$, respectively. The resulting image is visualized in Fig. \ref{fig:numExample}. Computation of the image takes $0.02$ seconds, and the resulting CPZ has $p = 12$ factors, $h = 13$ generators, $m = 8$ constraints, and $q = 85$ constraint generators, so that the representation size is $\mu = 1892$.

\begin{figure}
\centering
  \psfragfig[width=0.95\columnwidth]{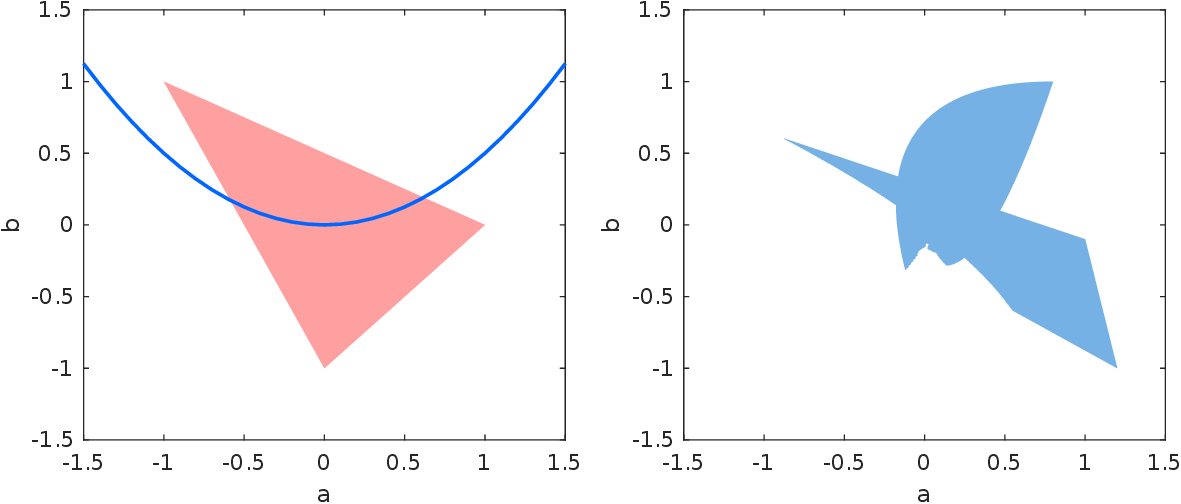}{
  \psfrag{a}[c][c]{$x_1$}
  \psfrag{b}[c][c]{\rotatebox[origin=c]{180}{$x_2$}}
  }
	\caption{Visualization of the polytope $\mathcal{P}$ (left, red), the boundary of the region $0.5~x_1^2 \leq x_2$ in \eqref{eq:nonlinFun} (left, blue), and the computed image of $\mathcal{P}$ under the nonlinear function $f(x)$ in \eqref{eq:nonlinFun} (right).}
	\label{fig:numExample}
\end{figure}


\section{Conclusion}

We introduced constrained polynomial zonotopes, a novel non-convex set representation that is closed under linear map, Minkowski sum, Cartesian product, convex hull, intersection, union, and quadratic and higher-order maps. We derived closed-form expressions for all relevant set operations and showed that the computational complexity of all relevant set operations is at most polynomial in the representation size. In addition, we derived closed-form expressions for the representation of zonotopes, polytopes, polynomial zonotopes, Taylor models, and ellipsoids as constrained polynomial zonotopes. Moreover, we demonstrated how to enclose constrained polynomial zonotopes by simpler set representations. In combination with our efficient techniques for representation size reduction, constrained polynomial zonotopes are well suited for many algorithms that compute with sets.

\begin{acknowledgements}
We gratefully acknowledge financial support by the project justITSELF funded by   the European Research Council (ERC) under grant number 817629 and the German Research Foundation (DFG) project faveAC under grant
number AL 1185/5-1. 
\end{acknowledgements}

\begin{appendices}

\section{Proof for the Union}
\label{app:proofUnion}

We now provide the proof for the closed-form expression for the union of two CPZs as specified in Theorem~\ref{theo:union}. As a prerequisite for the proof, we introduce the $\operator{constrDom}$ operation:

\begin{definition}
	Given a constraint defined by the constraint generator matrix $A \in \R^{m \times q}$, the constraint vector $b \in \R^{m}$, and the constraint exponent matrix $\CE \in \mathbb{N}_{0}^{p \times q}$, $\operator{constrDom}$ returns the set of values satisfying the constraint:
	\begin{equation*}
		\operator{constrDom}(A,b,\CE) = \bigg \{ \alpha ~\bigg |~ \sum_{i=1}^{q} \bigg( \prod_{k=1}^p \alpha_k^{\CE_{(k,i)}} \bigg) A_{(\cdot,i)} = b, ~ \alpha_k \in [\shortminus 1,1] \bigg \},
	\end{equation*}
	where $\alpha = [\alpha_1 ~ \dots ~\alpha_p]^T$.
	\label{def:constrDom} \hfill $\square$
\end{definition}
Using Def. \ref{def:constrDom}, it is straightforward to see that the following identity holds:
\begin{equation}
\setlength{\jot}{8pt}
	\begin{split}
		& \bigg \langle c,G,E,\begin{bmatrix} A_1 & \mathbf{0} \\ \mathbf{0} & A_2 \end{bmatrix},\begin{bmatrix} b_1 \\ b_2 \end{bmatrix}, \begin{bmatrix} \CE_1 & \CE_2 \end{bmatrix} \bigg \rangle_{CPZ} = \\
		& ~~~~~~~~~ \bigg \{ c + \sum_{i=1}^{h} \bigg( \prod_{k=1}^{p} \alpha_k^{E_{(k,i)}} \bigg) G_{(\cdot,i)} ~ \bigg | ~ \sum_{i=1}^{q_1} \bigg( \prod_{k=1}^p \alpha_k^{\CE_{1(k,i)}} \bigg) A_{1(\cdot,i)} = b_1, ~ \alpha \in \mathcal{D} \bigg \},
	\end{split}
	\label{eq:domElim}
\end{equation}
where $\mathcal{D} = \operator{constrDom}(A_2,b_2,\CE_2)$ and $\alpha = [\alpha_1 ~\dots ~\alpha_p]^T$. As another prerequisite, we introduce the following lemma:

\begin{lemma}
	Given a constant offset $c \in \Rn$, a generator matrix $G \in \R^{n \times h}$, an exponent matrix $E \in \mathbb{N}_0^{p \times h}$, and two domains $\mathcal{D}_1,\mathcal{D}_2 \subseteq [\shortminus \mathbf{1},\mathbf{1}] \subset \R^p$, it holds that
	\begin{equation*}
	\setlength{\jot}{12pt}
		\begin{split}
			& \bigg\{ c + \sum_{i=1}^{h} \bigg( \prod_{k=1}^{p} \alpha_k^{E_{(k,i)}} \bigg) G_{(\cdot,i)} ~ \bigg | ~ \alpha \in (\mathcal{D}_1 \cup \mathcal{D}_2) \bigg\} = \\
			& ~~~~ \underbrace{\bigg\{ c + \sum_{i=1}^{h} \bigg( \prod_{k=1}^{p} \alpha_k^{E_{(k,i)}} \bigg) G_{(\cdot,i)} ~ \bigg | ~ \alpha \in \mathcal{D}_1 \bigg\}}_{\mathcal{CPZ}_1} \cup \underbrace{\bigg\{ c + \sum_{i=1}^{h} \bigg( \prod_{k=1}^{p} \alpha_k^{E_{(k,i)}} \bigg) G_{(\cdot,i)} ~ \bigg | ~ \alpha \in \mathcal{D}_2 \bigg\}}_{\mathcal{CPZ}_2},
		\end{split}
	\end{equation*}
	where $\alpha = [\alpha_1 ~ \dots ~ \alpha_p]^T$.
	\label{lemma:union}
\end{lemma}

\begin{proof}
	Using the definition of CPZs in Def.~\ref{def:CPZ} and the definition of the union in \eqref{eq:defUnion}, we obtain
	\begin{align*}
		& \bigg\{ c + \sum_{i=1}^{h} \bigg( \prod_{k=1}^{p} \alpha_k^{E_{(k,i)}} \bigg) G_{(\cdot,i)} ~ \bigg | ~ \alpha \in (\mathcal{D}_1 \cup \mathcal{D}_2) \bigg\} = \\
		& ~ \\
		& \bigg\{ c + \sum_{i=1}^{h} \bigg( \prod_{k=1}^{p} \alpha_k^{E_{(k,i)}} \bigg) G_{(\cdot,i)} ~ \bigg | ~ \alpha \in \mathcal{D}_1 \vee \alpha \in \mathcal{D}_2 \bigg\} \overset{\substack{\text{Def.}~\ref{def:CPZ} \\ \vspace{-3pt}}}{=} \\
		& ~ \\
		& \{ x ~ | ~ x \in \mathcal{CPZ}_1 \vee x \in \mathcal{CPZ}_2 \} \overset{\substack{\eqref{eq:defUnion} \\ \vspace{-3pt}}}{=} \mathcal{CPZ}_1 \cup \mathcal{CPZ}_2,
	\end{align*}
	which concludes the proof. \hfill $\square$
\end{proof}
The outline of the proof is as follows: We first show in Sec.~\ref{sec:proofUnion1} that the constraints of the resulting CPZ $\langle c,G,E,A,b,\CE \rangle_{CPZ}$ from Theorem~\ref{theo:union} restrict the values for the factors $\alpha_k$ to the domain $\mathcal{D} = \mathcal{D}_1 \cup \mathcal{D}_2$ corresponding to the union of two domains $\mathcal{D}_1$ and $\mathcal{D}_2$. Afterward, in Sec.~\ref{sec:proofUnion2}, we apply Lemma \ref{lemma:union} to express the resulting CPZ from Theorem~\ref{theo:union} as $\langle c,G,E,A,b,\CE \rangle_{CPZ} = \mathcal{S}_1 \cup \mathcal{S}_2$, where $\mathcal{S}_1,\mathcal{S}_2$ represent the sets corresponding to the domains $\mathcal{D}_1,\mathcal{D}_2$, respectively. Finally, we show in Sec.~\ref{sec:proofUnion3} that $\mathcal{S}_1 = \mathcal{CPZ}_1$ and $\mathcal{S}_2 = \mathcal{CPZ}_2$ holds, which concludes the proof.

\subsection{Domain Defined by the Constraints}	
\label{sec:proofUnion1}

First, we show that the constraints of the resulting CPZ from Theorem~\ref{theo:union} restrict the values for the factors $\alpha_k$ to the domain $\mathcal{D} = \mathcal{D}_1 \cup \mathcal{D}_2$. The matrices $\widehat{A},\widehat{\CE}$ and the vector $\widehat{b}$ in Theorem~\ref{theo:union} define the constraint
	\begin{equation}
		\alpha_1 \alpha_2 = 1.
		\label{eq:proofUnion1}
	\end{equation}
	The only solutions for \eqref{eq:proofUnion1} within the domain $\alpha_1 \in [\shortminus 1,1]$, $\alpha_2 \in [\shortminus 1,1]$ are the two points $\alpha_1 = 1$, $\alpha_2 = 1$ and $\alpha_1 = \shortminus 1$, $\alpha_2 = \shortminus 1$. The constraint \eqref{eq:proofUnion1} therefore restricts the values for the factors $\alpha_k$ to the domain
	\begin{equation}
	\setlength{\jot}{12pt}
	\begin{split}
		\widehat{\mathcal{D}} & = \operator{constrDom}(\widehat{A},\widehat{b},\widehat{R}) \\
		& = \underbrace{ 1 \times 1 \times [\shortminus 1,1] \times \dotsc \times [\shortminus 1,1]}_{\widehat{\mathcal{D}}_1} \cup \underbrace{ \shortminus 1 \times \shortminus 1 \times [\shortminus 1,1] \times \dotsc \times [\shortminus 1,1] }_{\widehat{\mathcal{D}}_2}.
		\label{eq:unionDomain}
	\end{split}
	\end{equation}
	The matrices $\overline{A},\overline{\CE}$ and the vector $\overline{b}$ in Theorem~\ref{theo:union} define the constraint
	\begin{equation}
	\begin{split}
		& \sum_{i=1}^{\overline{q}} \bigg( \prod_{k=1}^{p_1+p_2+2} \alpha_k^{\overline{\CE}_{(k,i)}} \bigg) \overline{A}_{(\cdot,i)} =  \alpha_1 - \alpha_2 + \frac{1}{2} f_1(\alpha) - \frac{1}{2} \alpha_1 f_1(\alpha) - \frac{1}{2} f_2(\alpha)  \\
		& ~~~~~~~~~~~~~~~~~~~~~~~~~~~~~~~~~~~~~ - \frac{1}{2} \alpha_1 f_2(\alpha) - \frac{1}{4} f_1(\alpha)f_2(\alpha) + \frac{1}{4} \alpha_1 f_1(\alpha) f_2(\alpha) = \\
		& ~ \\
		& ~~~~~~~~~~~~~~~~~~~ =  \underbrace{\left(1 + \alpha_1 + \frac{1}{2} f_1(\alpha) (1 - \alpha_1) \right)\left(1 - \frac{1}{2} f_2(\alpha) \right) - \alpha_2 - 1}_{g(\alpha)} = 0 = \overline{b},
		\label{eq:union2}
	\end{split}
	\end{equation}
	where
	\begin{equation}
		f_1(\alpha) = \frac{1}{p_1} \sum_{k=3}^{p_1+2} \alpha_{(k)}^2, ~~ f_2(\alpha) = \frac{1}{p_2} \sum_{k=p_1+3}^{p_1 + p_2 + 2} \alpha_{(k)}^2,
		\label{eq:union2b}
	\end{equation}
	with $\overline{q}$ denoting the number of columns of matrix $\overline{A}$ and $\alpha = [\alpha_1~\dotsc~\alpha_{p_1 + p_2 + 2}]^T$. Let $\overline{\mathcal{D}} = \operator{constrDom}(\overline{A},\overline{b},\overline{R})$ be the restricted domain for the factor values corresponding to the constraint $g(\alpha) = 0$ in \eqref{eq:union2}. Then the factor domain $\mathcal{D}$ for the combination of the constraints defined by $\widehat{A},\widehat{b},\widehat{\CE}$ and $\overline{A},\overline{b},\overline{\CE}$ is
	\begin{equation}
	\setlength{\jot}{12pt}
	\begin{split}
		\mathcal{D} &= \operator{constrDom}\bigg( \begin{bmatrix} \widehat{A} & 0 \\ 0 & \overline{A} \end{bmatrix}, \begin{bmatrix} \widehat{b} \\ \overline{b} \end{bmatrix}, \begin{bmatrix} \widehat{R} & \overline{R} \end{bmatrix} \bigg) \\
		 &= \operator{constrDom}(\widehat{A},\widehat{b},\widehat{\CE}) \cap \operator{constrDom}(\overline{A},\overline{b},\overline{\CE}) = \widehat{\mathcal{D}} \cap \overline{\mathcal{D}} \overset{\substack{\eqref{eq:unionDomain} \\ \vspace{-3pt}}}{=} \underbrace{(\widehat{\mathcal{D}}_1 \cap \overline{\mathcal{D}})}_{\mathcal{D}_1} \cup \underbrace{(\widehat{\mathcal{D}}_2 \cap \overline{\mathcal{D}})}_{\mathcal{D}_2}.
	\end{split}
		\label{eq:unionOverallDom}
	\end{equation}
	To compute the domain $\mathcal{D}_1 = \widehat{\mathcal{D}}_1 \cap \overline{\mathcal{D}}$ in \eqref{eq:unionOverallDom}, we insert the values $\alpha_1 = 1$, $\alpha_2 = 1$ from $\widehat{\mathcal{D}}_1$ into the constraint $g(\alpha) = 0$. This yields the constraint $f_2(\alpha) = 0$ according to \eqref{eq:union2}, which is only satisfiable for $\alpha_{p_1 + 3} = 0, \dots, \alpha_{p_1 + p_2 + 2} = 0$. Moreover, inserting the values $\alpha_1 = \shortminus 1$, $\alpha_2 = \shortminus 1$ from $\widehat{\mathcal{D}}_2$ into the constraint $g(\alpha) = 0$ yields according to \eqref{eq:union2} the constraint 
	\begin{equation}
		f_1(\alpha) \underbrace{\bigg( 1 - \frac{1}{2} \underbrace{f_2(\alpha)}_{\in [0,1]} \bigg)}_{\in [0.5,1]} = 0,
	\end{equation}
	which is only satisfiable for $f_1(\alpha) = 0$. Since the constraint $f_1(\alpha) = 0$ is only satisfiable for $\alpha_{3} = 0, \dots, \alpha_{p_1 + 2} = 0$ according to \eqref{eq:union2b}, the constraint $g(\alpha) = 0$ is consequently also only satisfiable for $\alpha_{3} = 0, \dots, \alpha_{p_1 + 2} = 0$. In summary, the domain for the combination of the constraint $\alpha_1 \alpha_2 = 1$ in \eqref{eq:proofUnion1} and the constraint $g(\alpha) = 0$ in \eqref{eq:union2} is therefore
	\begin{equation}
		\begin{split}
			\mathcal{D} &= \operator{constrDom}\bigg( \underbrace{\begin{bmatrix} \widehat{A} & 0 \\ 0 & \overline{A} \end{bmatrix}}_{A_U}, \underbrace{\begin{bmatrix} \widehat{b} \\ \overline{b} \end{bmatrix}}_{b_U}, \underbrace{\begin{bmatrix} \widehat{R} & \overline{R} \end{bmatrix}}_{R_U} \bigg) \overset{\eqref{eq:unionOverallDom}}{=} \underbrace{(\widehat{\mathcal{D}}_1 \cap \overline{\mathcal{D}})}_{\mathcal{D}_1} \cup \underbrace{(\widehat{\mathcal{D}}_2 \cap \overline{\mathcal{D}})}_{\mathcal{D}_2} \\
			& = \underbrace{ \big \{ \begin{bmatrix} 1 & 1 & \alpha_3 & \dots & \alpha_{p_1 + 2} & \mathbf{0} \end{bmatrix}^T ~ | ~ \alpha_3,\dots, \alpha_{p_1 + 2} \in [\shortminus 1,1] \big \} }_{\mathcal{D}_1} \cup \\
			& ~~~~ \underbrace{ \big \{ \begin{bmatrix} \shortminus 1 & \shortminus 1 & \mathbf{0} & \alpha_{p_1 + 3} & \dots & \alpha_{p_1 + p_2 + 2} \end{bmatrix}^T ~ | ~ \alpha_{p_1 + 3},\dots, \alpha_{p_1 + p_2 + 2} \in [\shortminus 1,1] \big \} }_{\mathcal{D}_2},
		\end{split}
		\label{eq:proofUnion3}
	\end{equation}	
	which enables us to apply Lemma~\ref{lemma:union} in the next section.

\subsection{Reformulation as Union of Sets} 
\label{sec:proofUnion2}

We now prove that the resulting CPZ $\langle c,G,E,A,b,\CE \rangle_{CPZ}$ from Theorem~\ref{theo:union} defines the union of two sets $\mathcal{S}_1$ and $\mathcal{S}_2$. Using the domain $\mathcal{D}$ as defined in \eqref{eq:proofUnion3} and introducing
	\begin{equation}
		A_L = \begin{bmatrix} A_1 & \mathbf{0} & \shortminus 0.5 \, b_1 \\ \mathbf{0} & A_2 & 0.5 \, b_2 \end{bmatrix}, ~~ b_L = \begin{bmatrix} 0.5 \, b_1 \\ 0.5 \, b_2 \end{bmatrix}, ~~ \CE_L = \begin{bmatrix} \mathbf{0} & \mathbf{0} & 1 \\ \mathbf{0} & \mathbf{0} & 0 \\ \CE_1 & \mathbf{0} & \mathbf{0} \\ \mathbf{0} & \CE_2 & \mathbf{0} \end{bmatrix},
		\label{eq:proofUnion6}
	\end{equation}
the resulting CPZ $\langle c,G,E,A,b,\CE \rangle_{CPZ}$ from Theorem~\ref{theo:union} can be equivalently represented as
	\begin{align}
		& \bigg \langle c,G,E, \underbrace{\begin{bmatrix} \widehat{A} & \mathbf{0} & \mathbf{0} & \mathbf{0} & 0 \\ \mathbf{0} & \overline{A} & \mathbf{0} & \mathbf{0} & 0 \\ \mathbf{0} & \mathbf{0} & A_1 & \mathbf{0} & \shortminus 0.5 \, b_1 \\ \mathbf{0} & \mathbf{0} & \mathbf{0} & A_2 & 0.5 \, b_2 \end{bmatrix}}_{A}, \underbrace{\begin{bmatrix} \widehat{b} \\ \overline{b} \\ 0.5 \, b_1 \\ 0.5 \, b_2 \end{bmatrix}}_{b}, \underbrace{\begin{bmatrix} \widehat{\CE} & \overline{\CE} & \begin{bmatrix} \mathbf{0} & \mathbf{0} & 1 \\ \mathbf{0} & \mathbf{0} & 0 \\ \CE_1 & \mathbf{0} & \mathbf{0}\\ \mathbf{0} & \CE_2 & \mathbf{0} \end{bmatrix}  \end{bmatrix}}_{\CE} \bigg \rangle_{CPZ} \nonumber \\[10pt]
		& ~ \overset{\substack{\eqref{eq:proofUnion3},\eqref{eq:proofUnion6} \\ \vspace{-2pt}}}{=} \bigg \langle c,G,E,\begin{bmatrix} A_U & \mathbf{0} \\ \mathbf{0} & A_L \end{bmatrix}, \begin{bmatrix} b_U \\ b_L \end{bmatrix}, \begin{bmatrix} R_U & R_L \end{bmatrix} \bigg \rangle_{CPZ} \nonumber \\[10pt]
		& ~ \overset{\substack{\eqref{eq:domElim},\eqref{eq:proofUnion3} \\ \vspace{-2pt}}}{=}\bigg \{ c + \sum_{i=1}^{h} \bigg( \prod_{k=1}^{p} \alpha_k^{E_{(k,i)}} \bigg) G_{(\cdot,i)} ~ \bigg | ~ \sum_{i=1}^{q_L} \bigg( \prod_{k=1}^{p} \alpha_k^{\CE_{L(k,i)}} \bigg) A_{L(\cdot,i)} = b_L, ~ \alpha \in \mathcal{D}  \bigg \} \label{eq:proofUnion7}\\[10pt]
		& \overset{\substack{\text{Lemma}~\ref{lemma:union} \\ \vspace{-2pt} \\ \mathcal{D} = \mathcal{D}_1 \cup \, \mathcal{D}_2 \\ \vspace{-3pt}}}{=}  \underbrace{\bigg \{ c + \sum_{i=1}^{h} \bigg( \prod_{k=1}^{p} \alpha_k^{E_{(k,i)}} \bigg) G_{(\cdot,i)} ~ \bigg | ~ \sum_{i=1}^{q_L} \bigg( \prod_{k=1}^{p} \alpha_k^{\CE_{L(k,i)}} \bigg) A_{L(\cdot,i)} = b_L, ~ \alpha \in \mathcal{D}_1  \bigg \}}_{ \mathcal{S}_1} \nonumber \\
		& ~~~~~~~~~~ \cup \underbrace{\bigg \{ c + \sum_{i=1}^{h} \bigg( \prod_{k=1}^{p} \alpha_k^{E_{(k,i)}} \bigg) G_{(\cdot,i)} ~ \bigg | ~ \sum_{i=1}^{q_L} \bigg( \prod_{k=1}^{p} \alpha_k^{\CE_{L(k,i)}} \bigg) A_{L(\cdot,i)} = b_L, ~ \alpha \in \mathcal{D}_2  \bigg \}}_{\mathcal{S}_2}, \nonumber
	\end{align}
	where $p = p_1 + p_2 + 2$, $q_L = q_1 + q_2 + 1$, and $\alpha = [\alpha_1~\dotsc~\alpha_p]^T$. 
	
\subsection{Equivalence of Sets}	
\label{sec:proofUnion3}

It remains to show that $\mathcal{S}_1 = \mathcal{CPZ}_1$ and $\mathcal{S}_2 = \mathcal{CPZ}_2$. Inserting the definition of the domain $\mathcal{D}_1$ in \eqref{eq:proofUnion3} into the definition of the set $\mathcal{S}_1$ in \eqref{eq:proofUnion7} yields
	\begin{align*}
			& \mathcal{S}_1 \overset{\eqref{eq:proofUnion7}}{=} \bigg \{ c + \sum_{i=1}^{h} \bigg( \prod_{k=1}^{p} \alpha_k^{E_{(k,i)}} \bigg) G_{(\cdot,i)} ~ \bigg | ~ \sum_{i=1}^{q_L} \bigg( \prod_{k=1}^{p} \alpha_k^{\CE_{L(k,i)}} \bigg) A_{L(\cdot,i)} = b_L, ~ \alpha \in \mathcal{D}_1  \bigg \} \overset{\eqref{eq:proofUnion3}}{=} \\[5pt]
			&  \bigg \{ c + \sum_{i=1}^{h} \bigg( \prod_{k=1}^{p} \alpha_k^{E_{(k,i)}} \bigg) G_{(\cdot,i)} ~ \bigg | ~ \sum_{i=1}^{q_L} \bigg( \prod_{k=1}^{p} \alpha_k^{\CE_{L(k,i)}} \bigg) A_{L(\cdot,i)} = b_L, ~ \alpha_1,\alpha_2 = 1, \nonumber \\
			& ~~~~~~~~~~~~~~~~~~~~~~~~~~~~~~~~ \alpha_3,\dots,\alpha_{p_1 + 2} \in [\shortminus 1,1], ~~ \alpha_{p_1+3}, \dots, \alpha_{p_1 + p_2 + 2} = 0  \bigg \} \overset{\eqref{eq:proofUnion6}, \text{Thm.}~\ref{theo:union}}{=} \\[5pt]
			&  \bigg \{ \underbrace{0.5 (c_1 + c_2) + 0.5 (c_1 - c_2) \alpha_1}_{\overset{\substack{\scriptscriptstyle \alpha_1 = 1 \\ \vspace{-4pt}}}{=} c_1} + \sum_{i=1}^{h_1} \bigg( \prod_{k=1}^{p_1} \alpha_{2+k}^{E_{1(k,i)}} \bigg) G_{1(\cdot,i)} + \sum_{i=1}^{h_2} \underbrace{\bigg( \prod_{k=1}^{p_2} \alpha_{2+p_1+k}^{E_{2(k,i)}} \bigg)}_{= 0} G_{2(\cdot,i)} ~ \bigg | \\
			& ~\sum_{i=1}^{q_1} \bigg( \prod_{k=1}^{p_1} \alpha_{2+k}^{\CE_{1(k,i)}} \bigg) A_{1(\cdot,i)} = \underbrace{0.5 \,b_1 + 0.5 \, b_1 \alpha_1}_{\overset{\substack{\scriptscriptstyle \alpha_1 = 1 \\ \vspace{-4pt} }}{=} b_1}, \, \sum_{i=1}^{q_2} \underbrace{\bigg( \prod_{k=1}^{p_2} \alpha_{2+p_1+k}^{\CE_{2(k,i)}} \bigg)}_{ = 0} A_{2(\cdot,i)} = \underbrace{ 0.5\, b_2 - 0.5 \, b_2 \alpha_1}_{\overset{\substack{\scriptscriptstyle \alpha_1 = 1 \\ \vspace{-4pt}}}{=} \mathbf{0}}, \\
			& ~~~ \alpha_1,\alpha_2 = 1,~~ \alpha_3,\dots,\alpha_{p_1 + 2} \in [\shortminus 1,1], ~~ \alpha_{p_1+3}, \dots, \alpha_{p_1 + p_2 + 2} = 0 \bigg \} =  \\[5pt]
			& \underbrace{\bigg \{ c_1 + \sum_{i=1}^{h_1} \bigg( \prod_{k=1}^{p_1} \alpha_{2+k}^{E_{1(k,i)}} \bigg) G_{1(\cdot,i)} ~ \bigg | ~ \sum_{i=1}^{q_1} \bigg( \prod_{k=1}^{p_1} \alpha_{2+k}^{\CE_{1(k,i)}} \bigg) A_{1(\cdot,i)} = b_1,~ \alpha_3,\dots,\alpha_{p_1 + 2} \in [\shortminus 1,1] \bigg \}}_{ = \mathcal{CPZ}_1}. \\
	\end{align*}
	The proof that $\mathcal{S}_2 = \mathcal{CPZ}_2$ is similar to the proof for $\mathcal{S}_1$ and therefore omitted.

\section{Rescaling}
\label{app:rescaling}

We now show how the set obtained by rescaling as described in Sec.~\ref{sec:rescaling} can be represented as a CPZ. For this, we first introduce the operation \operator{subset}:

\begin{proposition}
	(Subset) Given $\mathcal{CPZ} = \langle c,G,E,A,b,R\rangle_{CPZ} \subset \mathbb{R}^n$, the index of one factor $\ind \in \{1,\dots,p\}$, and an interval $[l,u] \subseteq [\shortminus 1,1]$, the operation $\operator{subset}$ substitutes the domain for the factor $\alpha_r$ by $\alpha_r \in [l,u]$, which yields a CPZ that is a subset of $\mathcal{CPZ}$:
	\begin{align*}
		& \operator{subset}\big(\mathcal{CPZ},\ind,[l,u]\big) = \\[5pt]
		& \bigg \{ c + \sum_{i=1}^{h} \bigg( \prod_{k=1}^p \alpha_k^{E_{(k,i)}} \bigg) G_{(\cdot,i)} ~ \bigg | ~ \sum_{i=1}^{q} \bigg( \prod_{k=1}^p \alpha_k^{\CE_{(k,i)}} \bigg) A_{(\cdot,i)} = b, ~\alpha_k = [\shortminus 1,1], ~\alpha_r \in [l,u] \bigg \} \\[10pt]
		& = \Big \langle c,\big[ \widehat{G}_{1} ~ \dots ~ \widehat{G}_{h} \big], \big[ \widehat{E}_{1} ~ \dots ~ \widehat{E}_{h} \big],\big[ \widehat{A}_{1} ~ \dots ~ \widehat{A}_{q} \big],b, \big[ \widehat{R}_{1} ~ \dots ~ \widehat{R}_{q} \big] \Big \rangle_{CPZ} \subseteq \mathcal{CPZ}
	\end{align*}	
	with
	\begin{equation*}
	\setlength{\jot}{12pt}
	\begin{gathered}
		\widehat{E}_{i} = \begin{bmatrix} E_{(\{1,\dots ,\ind-1 \},i)} & E_{(\{1,\dots ,\ind-1 \},i)} & \dots & E_{(\{1,\dots ,\ind-1 \},i)} & E_{(\{1, \dots ,\ind-1\},i)} \\ \vspace{-8pt} \\ 0 & 1 & \dots &  E_{(\ind,i)}-1 & E_{(\ind,i)} \\  E_{(\{\ind+1,\dots ,p \},i)} & E_{(\{\ind+1,\dots ,p \},i)} & \dots & E_{(\{\ind+1,\dots ,p \},i)} & E_{(\{\ind+1, \dots ,p\},i)} \end{bmatrix}, \\
		\widehat{R}_{i} = \begin{bmatrix} R_{(\{1,\dots ,\ind-1 \},i)} & R_{(\{1,\dots ,\ind-1 \},i)} & \dots & R_{(\{1,\dots ,\ind-1 \},i)} & R_{(\{1, \dots ,\ind-1\},i)} \\ \vspace{-8pt} \\ 0 & 1 & \dots &  R_{(\ind,i)}-1 & R_{(\ind,i)} \\  R_{(\{\ind+1,\dots ,p \},i)} & R_{(\{\ind+1,\dots ,p \},i)} & \dots & R_{(\{\ind+1,\dots ,p \},i)} & R_{(\{\ind+1, \dots ,p\},i)} \end{bmatrix}, \\
		\widehat{G}_{i} = \begin{bmatrix} d_{i,0} \, G_{(\cdot,i)} & \dots & d_{i,E_{(\ind,i)}} G_{(\cdot,i)} \end{bmatrix}, ~ \widehat{A}_{i} = \begin{bmatrix} o_{i,0} \, A_{(\cdot,i)} & \dots & o_{i,R_{(\ind,i)}} A_{(\cdot,i)}	\end{bmatrix}, \\
		d_{i,j} =  v_1^{E_{(\ind,i)}} \, v_2^j \, {{E_{(\ind,i)}}\choose{j}},~ o_{i,j} =  v_1^{R_{(\ind,i)}} \, v_2^j \, {{R_{(\ind,i)}}\choose{j}} , ~ v_1 = \frac{u+l}{2}, ~ v_2 = \frac{u-l}{u+l},
	\end{gathered}
	\end{equation*}
	where ${w} \choose {z}$, $w,z \in \mathbb{N}_0$ denotes the binomial coefficient. The \operator{compactGen} and \operator{compactCon} operations are applied to obtain a regular CPZ.
	\label{prop:subset}
\end{proposition}
\begin{proof}
The domain $\alpha_r \in [l,u]$ for the factor $\alpha_r$ can be equivalently represented by an auxiliary variable $\widehat{\alpha}_r \in [\shortminus 1,1]$:
	\begin{equation}
	\begin{split}
		& \big \{ \alpha_\ind ~\big|~ \alpha_\ind \in [l,u] \big \} = \big \{ 0.5 (u+l) + 0.5(u-l) \, \widehat{\alpha}_r ~\big|~ \widehat{\alpha}_r \in [\shortminus 1,1] \big\} \\[5pt]
		& = \bigg\{ \underbrace{\frac{u+l}{2}}_{v_1} \Big(1 + \underbrace{\frac{u-l}{u+l}}_{v_2} \widehat{\alpha}_r \Big)~\bigg|~ \widehat{\alpha}_r \in [\shortminus 1,1] \bigg\} = \big \{ v_1 (1 + v_2 \, \widehat{\alpha}_r) ~\big|~ \widehat{\alpha}_r \in [\shortminus 1,1] \big\}.
	\end{split}
	\label{eq:subset1}
	\end{equation}
	Moreover, due to the properties of the binomial coefficient, it holds that
	\begin{equation}
	\begin{split}
		& \big( v_1 (1 + v_2 \, \widehat{\alpha}_r)\big)^{E_{(r,i)}} = d_{i,0} + d_{i,1} \, \widehat{\alpha}_{\ind} + d_{i,2} \, \widehat{\alpha}_{\ind}^2 + \dots + d_{i,E_{(\ind,i)}} \widehat{\alpha}_{\ind}^{E_{(\ind,i)}} \\
		& \big( v_1 (1 + v_2 \, \widehat{\alpha}_r)\big)^{R_{(r,i)}} = o_{i,0} + o_{i,1} \, \widehat{\alpha}_{\ind} + o_{i,2} \, \widehat{\alpha}_{\ind}^2 + \dots + o_{i,R_{(\ind,i)}} \widehat{\alpha}_{\ind}^{R_{(\ind,i)}}.
	\end{split}
	\label{eq:subset2}
	\end{equation}
	Using \eqref{eq:subset1} and \eqref{eq:subset2}, we obtain
	\begin{align*}
		& \operator{subset}\big(\mathcal{CPZ},\ind,[l,u]\big) = \\[5pt]
		& \bigg \{ c + \sum_{i=1}^{h} \bigg( \prod_{k=1}^p \alpha_k^{E_{(k,i)}} \bigg) G_{(\cdot,i)} ~ \bigg | ~ \sum_{i=1}^{q} \bigg( \prod_{k=1}^p \alpha_k^{\CE_{(k,i)}} \bigg) A_{(\cdot,i)} = b, ~\alpha_k = [\shortminus 1,1], ~\alpha_r \in [l,u] \bigg \} \\[10pt]
		& \overset{\eqref{eq:subset1}}{=} \bigg \{ c + \sum_{i=1}^{h} \bigg( \prod_{\substack{k=1 \\ k \neq r}}^p \alpha_k^{E_{(k,i)}} \bigg) {\underbrace{\big( v_1 (1 + v_2 \, \widehat{\alpha}_r)\big)}_{\alpha_r}}^{E_{(r,i)}} G_{(\cdot,i)} ~ \bigg | \\
		& ~~~~~~~~~~ \sum_{i=1}^{q} \bigg( \prod_{\substack{k=1 \\ k \neq r}}^p \alpha_k^{\CE_{(k,i)}} \bigg) {\underbrace{\big( v_1 (1 + v_2 \, \widehat{\alpha}_r)\big)}_{\alpha_r}}^{R_{(r,i)}} A_{(\cdot,i)} = b, ~\alpha_k,\widehat{\alpha}_r = [\shortminus 1,1] \bigg \} \\[3pt]
		& \overset{\eqref{eq:subset2}}{=} \Big \langle c,\big[ \widehat{G}_{1} ~ \dots ~ \widehat{G}_{h} \big], \big[ \widehat{E}_{1} ~ \dots ~ \widehat{E}_{h} \big],\big[ \widehat{A}_{1} ~ \dots ~ \widehat{A}_{q} \big],b, \big[ \widehat{R}_{1} ~ \dots ~ \widehat{R}_{q} \big] \Big \rangle_{CPZ},
	\end{align*}	
	which concludes the proof.
\end{proof}
The set obtained by rescaling $\mathcal{CPZ} = \langle c,G,E,A,b,R \rangle_{CPZ}$ can be computed by applying the \operator{subset} operation to each factor
\begin{align*}
& \bigg \{ c + \sum_{i=1}^{h} \bigg( \prod_{k=1}^p \alpha_k^{E_{(k,i)}} \bigg) G_{(\cdot,i)} ~ \bigg | ~ \sum_{i=1}^{q} \bigg( \prod_{k=1}^p \alpha_k^{\CE_{(k,i)}} \bigg) A_{(\cdot,i)} = b, ~ [\alpha_1~\dots~\alpha_p]^T \in [l,u] \bigg \} \\[5pt]
& \overset{\substack{\text{Prop.}~\ref{prop:subset} \\ \vspace{-2pt}}}{=} \operator{subset}\Big( \operator{subset} \big( ~ \dots ~ \operator{subset}(\mathcal{CPZ},1,[l_{(1)},u_{(1)}]) ~ \dots ~ \big),p,[l_{(p)},u_{(p)}] \Big), 
\end{align*}
which yields a CPZ.

\end{appendices}

\bibliographystyle{spmpsci}     
\bibliography{kochdumper,cpsGroup}   

\end{document}